\documentclass[11pt,a4paper]{amsart}
\usepackage{upref}
\usepackage{enumitem,dsfont,amssymb,bm}
\usepackage{graphicx,color}
\usepackage[colorlinks=true, pdftitle={K\L{}S-inequality for gradient
flows in metric spaces}, pdfauthor={Daniel Hauer, Jos\'e M. Maz\'on}]{hyperref}

\title[K\L{}S-inequality for gradient
flows in metric spaces]{Kurdyka-\L{}ojasiewicz-Simon inequality for gradient flows in metric
  spaces}

\author{Daniel Hauer}
\address[Daniel Hauer]{School of Mathematics and Statistics\\ The University of Sydney\\
  NSW 2006\\ Australia}
\email{\href{mailto:daniel.hauer@sydney.edu.au}{\nolinkurl{daniel.hauer@sydney.edu.au}}}

\author{Jos\'e M. Maz\'on}
\address[Jos\'e M. Maz\'on]{Departament
  d'An\`alisi Matem\`atica, Universitat de Val\`encia, Valencia, Spain}
\email{\href{mailto:mazon@uv.es}{\nolinkurl{mazon@uv.es}}}

\subjclass[2010]{49J52 - 35K90 - 35B40 - 49Q20 - 58J35 - 39B62}

\keywords{Gradient flows in metric spaces,
  Kurdyka-\L{}ojasiewicz-Simon inequality, Wasserstein distances,
  Logarithmic Sobolev, Talagrand's entropy-transportation inequality}

\numberwithin{equation}{section}

\theoremstyle{theorem}
\newtheorem{theorem}{Theorem}[section]
\newtheorem{proposition}[theorem]{Proposition}
\newtheorem{lemma}[theorem]{Lemma}
\newtheorem{corollary}[theorem]{Corollary}

\theoremstyle{definition}
\newtheorem{definition}[theorem]{Definition}
\newtheorem{assumption}{Assumption}[section]
\newtheorem{examples}[theorem]{Example}

\theoremstyle{remark}
\newtheorem{remark}[theorem]{Remark}
\newtheorem{notation}[theorem]{Notation}

\newcommand\R{{\mathbb{R}}}
\newcommand\N{\mathbb{N}}

\newcommand\E{\mathcal{E}}
\newcommand\HE{\mathcal{H}}
\newcommand\U{\mathcal{U}}
\newcommand\D{\mathbb{D}}
\newcommand\td{\mathrm{d}}
\newcommand\dx{\mathrm{d}x }
\newcommand\dy{\mathrm{d}y }
\newcommand\dr{\mathrm{d}r }
\newcommand\ds{\mathrm{d}s }
\newcommand\dt{\mathrm{d}t }

\newcommand\fdt{\frac{\mathrm{d}}{\dt}}
\newcommand\fdr{\frac{\mathrm{d}}{\dr}}
\newcommand\fds{\frac{\mathrm{d}}{\ds}}

\DeclareMathOperator*{\divergence}{div}

\DeclareMathOperator{\inter}{int}

\def\1{\raisebox{2pt}{\rm{$\chi$}}}

\newcommand\abs[1]{\lvert#1\rvert}
\newcommand\labs[1]{\left\lvert#1\right\rvert}
\newcommand\norm[1]{\lVert#1\rVert}

\definecolor{darkred}{rgb}{0.7,0.1,0.1}

\begin{document}
\date{\today}
\maketitle
%

\begin{abstract}
 This paper is dedicated to providing new tools and methods for studying the
  trend to equilibrium of gradient flows in metric spaces
  $(\mathfrak{M},d)$ in the \emph{entropy} and \emph{metric} sense, to
  establish \emph{decay rates}, \emph{finite time of extinction}, and to
  characterize \emph{Lyapunov stable equilibrium points}. More
  precisely, our main results are:
\begin{itemize}[topsep=3pt,itemsep=-3pt,partopsep=1ex,parsep=1ex]
\item Introduction of a \emph{gradient inequality} in the metric space
  framework, which in the Euclidean space $\R^{N}$ was obtained by
  \L{}ojasie\-wicz [\'Edi\-tions du C.N.R.S., Paris, 1963], later
  improved by Kurdyka~[Ann. Inst. Fourier, 48 (3), 1998], and
  generalized to the Hilbert space framework by Simon~[Ann. of
  Math. (2) 118, 1983].
\item Obtainment of the trend to equilibrium in the \emph{entropy} and
  \emph{metric} sense of gradient flows generated by a functional
  $\E : \mathfrak{M}\to (-\infty,+\infty]$ satisfying a
  Kurdyka-\L{}ojasiewicz-Simon inequality in a neighborhood of an
  equilibrium point of $\E$. Sufficient conditions are
  given implying \emph{decay rates} and \emph{finite time of
    extinction} of gradient flows.
\item Construction of a \emph{talweg curve} in
  $\mathfrak{M}$ with an \emph{optimal} growth function yielding the
  validity of a Kurdyka-\L{}ojasiewicz-Simon inequality.
\item Characterization of \emph{Lyapunov stable} equilibrium points of
  $\E$ satisfying a Kurdyka-\L{}ojasiewicz-Simon inequality near such
  points.
\item  Characterization of the \emph{entropy-entropy production
    inequality} with the Kur\-dyka-\L{}ojasiewicz-Simon inequality.
\end{itemize}
As an application of these results, the following properties are established.
\begin{itemize}[topsep=3pt,itemsep=-3pt,partopsep=1ex,parsep=1ex]
\item New upper bounds on the extinction time of gradient flows
  associated with the total variational flow.
\item If the metric space $\mathfrak{M}$ is the $p$-Wasserstein space
  $\mathcal{P}_{p}(\R^{N})$, $1<p<\infty$, then new HWI-, Talagrand-,
  and logarithmic Sobolev inequalities are obtained for functionals
  $\E$ associated with \emph{nonlinear diffusion} problems modeling
  drift, potential and interaction phenomena.\newline \mbox{}$\quad$
  It is shown that these inequalities are equivalent to the
  Kurdyka-\L{}ojasiewicz-Simon inequality and hence, they imply trend
  to equilibrium of the gradient flows of $\E$ with decay rates or
  arrival in finite time.
\end{itemize}
\end{abstract}

\tableofcontents

\section{Introduction}

The long-time asymptotic behavior of gradient flows is one
fundamental task in the study of gradient systems and
often studied separately from well-posedness. Given a metric space
$(\mathfrak{M},d)$, we call a curve $v$ a \emph{gradient flow} in
$\mathfrak{M}$ if there is an \emph{energy} functional
$\E : \mathfrak{M}\to (-\infty,+\infty]$ such that $v$ is a
\emph{$2$-curve of maximal slope of $\E$} (\cite{AGS-ZH}, see
Definition~\ref{def:p-max-slope}); in brief,
$v : [0,+\infty)\to \mathfrak{M}$ is locally absolutely continuous
with \emph{metric derivative} $\abs{v'}\in L^{2}_{loc}(0,+\infty)$
(see~Definition~\ref{def:metric-derivative}), $\E\circ v$ is locally
absolutely continuous, and there is a second functional
$g : \mathfrak{M}\to [0,+\infty]$ called \emph{strong upper gradient
  of $\E$} satisfying
\begin{displaymath}
  \labs{\fdt\E(v(t))}\le g(v(t))\,\abs{v'}(t)\qquad\text{for a.e. $t>0$,}
\end{displaymath}
(see~Definition~\ref{def:strong-uppergradient}) and such that
\emph{energy dissipation inequality}
\begin{equation}\label{EDI11}
    \fdt\E(v(t)) \le - \frac{1}{2} \abs{v'}^2(t) - \frac{1}{2}
    g^2(v(t)) \qquad\text{holds for a.e. $t>0$.}
\end{equation}

This general notion of gradient flows in metric spaces is consistent
with the notion of \emph{strong solutions} of gradient systems in the
\emph{Hilbert spaces framework}~(developed in \cite{BrezisBook},
cf~\cite[Corollary~1.4.2]{AGS-ZH}):
\begin{equation}\label{ACP1}
    v'(t) + \partial \E(v(t))\ni 0\qquad \text{for $t> 0$,}
\end{equation}
where $\partial\E$ is a \emph{subdifferential operator} of a
semi-convex, proper, lower semicontinuous functional $\E$ on a Hilbert
space (see~Section~\ref{sec:Hilbertspacesframework}).

Since the pioneering work~\cite{JKO} by Jordan, Kinderlehrer and Otto, we know
that solutions of some diffusion equations for probability distributions
can also be derived as gradient flows of a given functional $\E$ with
respect to a differential structure induced by the $p$-Wasserstein
space $(\mathcal{P}_{p}(M),W_{p})$. 
 Thus, we call this setting the \emph{metric space framework} (see
 also~\cite{MR1964483,AGS-ZH} and Section~\ref{GradProbmeasures}). 
%

While for the Hilbert space framework, an
extensive literature on determining the long-time asymptotic behavior
of gradients flows already exists (for instance,
see~\cite{HaJenBook,Ha91} or more recently~\cite{MR2763076}, and the
references therein), it seems that the methods available for the
metric space framework (cf, for instance, \cite{ACDDJLMTV}
or~\cite{VillaniCompMath2004}) are mainly based on the idea in
establishing global \emph{entropy (entropy production/transport)
  inequalities}. 


Before passing to the main results of this article, we review the
concept of the classical \emph{entropy method} as it is used, for
instance, in studying the convergence to equilibrium of solutions of
kinetic equations such as the Boltzmann or Fokker-Planck equation
(cf~\cite[\S3]{ACDDJLMTV}, \cite{MR2006305}, \cite[\S 9]{MR1964483}, or
\cite{CJMTU,MR2407976,VillaniCompMath2004}). This will reveal the
connection of the famous \emph{entropy-entropy production inequality}
(see~\eqref{eq:13} below) and the \emph{Kurdyka-\L{}ojasiewicz-Simon
  inequality} (see~\eqref{eq:KL32} below), which is the main object of this paper.

With the task to determine the asymptotic behavior for large time of gradient flows,
the following three problems arise naturally.
\begin{enumerate}[label=(\Roman*),leftmargin=2\parindent]
  \item \label{enum:1} Does every gradient flow $v(t)$ trend to an equilibrium
    $\varphi$ of $\E$ as $t\to+\infty$?
  \item \label{enum:2} In which sense or topology the trend of $v(t)$ to
    $\varphi$ holds as~$t\to+\infty$?
  \item \label{enum:3} What is the \emph{decay rate} of $d(v(t),\varphi)$ as $t\to+\infty$?
\end{enumerate}

Due to inequality~\eqref{EDI11}, $\E$ is a \emph{Lyapunov functional}
of each of its gradient flow curve $v$, that is, $\E\circ v$ is
monotonically decreasing $[0,+\infty)$. In addition, if $\E$ and $g$
are both lower semicontinuous on $\mathfrak{M}$, then $\E$ is even a
\emph{strict Lyapunov functional}, that is, the condition
$\E\circ v\equiv \textrm{const.}$ on $[t_{0},+\infty)$ for some
$t_{0}\ge 0$ implies that $ v\equiv \textrm{const.}$ on
$[t_{0},+\infty)$
(cf~Proposition~\ref{propo:omega-limit-gradientflow}) and so, the
\emph{$\omega$-limit~set}
\begin{displaymath}
    \omega(v):=\Big\{\varphi \in \mathfrak{M}\,\Big\vert\, \text{
      there is $t_{n}\uparrow +\infty$ s.t. }\lim_{n\to \infty}d(v(t_{n}),\varphi)=0 \Big\}
\end{displaymath}
is contained in the set $\mathbb{E}_{g}:=g^{-1}(\{0\})$ of
\emph{equilibrium points} of $\E$ with respect to the strong upper
gradient $g$ (see Definition~\ref{def:equilibrium-points}). Therefore,
Problem~\ref{enum:1} is positively answered.

Given that the strict Lyapunov functional $\E$ of a gradient flow $v$
attains an equilibrium point at $\varphi\in \omega(v)$, one
possibility to measure the discrepancy between $v$ and the
equilibrium $\varphi$ is given by \emph{relative entropy}
\begin{displaymath}
  \E(v(t)\vert\varphi):=\E(v(t))-\E(\varphi)\ge 0.
\end{displaymath}
Thus concerning Problem~\ref{enum:2}, there are two types of \emph{trend to
  equilibrium}; namely the following (\emph{weaker} type)

\begin{enumerate}[topsep=3pt,leftmargin=1.6\parindent,partopsep=1ex,parsep=1ex]
\item[(1)] {\bfseries Trend to equilibrium in the entropy sense}: a
  gradient flow $v$ of $\E$ is said to \emph{trend to equilibrium
    $\varphi$ in the entropy sense} if
  \begin{displaymath}
    \E(v(t)\vert \varphi)\to 0\qquad\text{ as $t\to +\infty$,}
  \end{displaymath}
\end{enumerate}
and the (\emph{stronger} type)
\begin{enumerate}[topsep=3pt,leftmargin=1.6\parindent,itemsep=1ex,partopsep=1ex,parsep=1ex]
\item[(2)] {\bfseries Trend to equilibrium in the metric sense}: a
  gradient flow $v$ is said to \emph{trend to equilibrium $\varphi$
    of $\E$ in the metric sense}~if
  \begin{displaymath}
    \lim_{t\to+\infty}d(v(t),\varphi)=0.
  \end{displaymath}
\end{enumerate}

Now, the classical \emph{entropy method} suggests
to find a strictly increasing function $\Phi\in C([0,+\infty))$
satisfying $\Phi(0)=0$ such that $\E$ 
satisfies a (global) {\bfseries entropy-entropy production/dissipation (EEP-)inequality}
\begin{equation}
  \label{eq:13}
  \D(v)\ge \Phi(\E(v\vert \varphi))
\end{equation}
for all $v\in D(\E)$ at an equilibrium point $\varphi$. Here, the map $\D :
\mathfrak{M}\to [0,+\infty]$ is
called the \emph{entropy production functional} and satisfies
\begin{equation}
  \label{eq:44}
  \D(v(t))=-\frac{\td}{\dt}\E(v(t))=g^{2}(v(t))\qquad\text{for $t>0$}
\end{equation}
for every gradient flow $v$ of $\E$ with strong upper gradient $g$
(cf~Proposition~\ref{propo:chara-p-curves} in
Section~\ref{sec:curves-max-slope} below).
Combining~\eqref{eq:13} with~\eqref{eq:44}, rearranging and then
integrating the resulting inequality yields for each curve $v$
generated by $\E$ trend to equilibrium $\varphi$ in the entropy sense.


If the function $\Phi$ in~\eqref{eq:13} is known, then the entropy
method has certainly the advantage that EEP-inequality~\eqref{eq:13}
provides decay rates (Problem~\ref{enum:3}) to the trend to equilibrium in the
entropy sense. For instance, if $\Phi$ is \emph{linear}, that is,
$\Phi(s)=\lambda s$, ($s\ge 0$, $\lambda>0$), then
EEP-inequality~\eqref{eq:13} implies \emph{exponential decay rates}
and if $\Phi$ is \emph{polynomial} $\Phi(s)=K s^{1+\alpha}$,
($\alpha\in (0,1)$, $K>0$), then EEP-inequality~\eqref{eq:13} implies
that $\E(v(t)\vert \varphi)$ decays to $0$ at least like
$t^{-1/\alpha}$ (\emph{polynomial decay}).

On the other hand, the fact that $\E$ satisfies a (global)
EEP-inequality~\eqref{eq:13} requires in the Hilbert as well as in the
metric space framework additional rather strong conditions on
$\E$. For example, in the Hilbert space framework $H=L^{2}(\Omega)$ of
a bounded extension domain $\Omega$ in $\R^{N}$, ($N\ge 1$), the fact
that a functional $\E : H\to (-\infty,+\infty]$ satisfies an
EEP-inequality~\eqref{eq:13} at $\varphi=0$ is equivalent to the fact
that $\E$ satisfies an abstract \emph{Poincar\'e-Sobolev inequality}
(see~\eqref{eq:PoincareSobolevE} in
Section~\ref{subsubsec:Derivation-LS-inequality}). Or, for instance,
in the metric space framework of the $p$-Wasserstein space
$\mathfrak{M}=\mathcal{P}_{p}(\Omega)$: if a functional
$\E : \mathfrak{M}\to (-\infty,+\infty]$ is
\emph{$\lambda$-geodesically convex} for some $\lambda>0$ (see
Definition~\ref{def:lambdageoconvex}), then $\E$ satisfies a (global)
EEP-inequality~\eqref{eq:13}. See Section~\ref{GradProbmeasures} for
further details.


%
%
%


Now, to obtain trend to equilibrium in the metric sense, an {\bfseries
  entropy-transportation (ET-)inequality}
\begin{equation}
  \label{eq:13-transport0}
  d(v,\varphi)\le \Psi(\E(v\vert \varphi)),\qquad (v\in D(\E)),
\end{equation}
can be very useful (see, for instance, \cite{CMVIbero,MR1964483}
or~\cite{MR2016985}), where $\Psi\in C[0,+\infty)$ is a strictly
increasing function satisfying $\Psi(0)=0$. If $\Phi$ and $\Psi$ are
known, then \emph{decay estimates} to the trend to
equilibrium in the metric sense can be derived by combining the two
inequalities~\eqref{eq:13} and~\eqref{eq:13-transport0}.



However, there are many important examples of functionals $\E$ that do
not satisfy a \emph{global} EEP-inequality~\eqref{eq:13} (see, for
instance,~\cite[Theorem~2, p~21]{MR2407976} or \cite[p
98]{VillaniCompMath2004}). In particular, if $\E$ is not
$\lambda$-geodesically convex for a $\lambda>0$, then it might not
have a unique global equilibrium point. But the study of trend to
equilibrium of gradient flows generated by this class of functionals
remains an important (open) problem and requires more sophisticated tools
and arguments.


Our approach to attack the above mentioned problems is via a functional
inequality, which in the Hilbert space framework is known as the
{\bfseries Kurdyka-\L{}ojasiewicz (K\L{}-)inequality}: a proper
functional $\E : \mathfrak{M}\to (-\infty,+\infty]$ with strong upper
gradient $g$ and equilibrium point $\varphi\in D(\E)$ is said to
satisfy a \emph{Kurdyka-\L{}ojasiewicz inequality} on a set
$\mathcal{U}\subseteq g^{-1}((0,+\infty))\cap \big\{v\in
\mathfrak{M}\,\big\vert\, \theta'(\E(\cdot\vert\varphi))>0\big\}$ if
there is a strictly increasing function $\theta\in W^{1,1}_{loc}(\R)$
satisfying $\theta(0)=0$ such that
\begin{equation}
    \label{eq:KL32}
     \theta'(\E(v\vert \varphi))\, g(v)\ge 1\qquad
     \text{for all $v\in \mathcal{U}$.}
\end{equation}

Since in the literature, EEP-inequality~\eqref{eq:13} is rather used
from the community studying kinetic equations, while
Kurdyka-\L{}ojasiewicz inequality~\eqref{eq:KL32} is rather familiar
for communities from \emph{algebraic geometry} classifying
singularities of manifolds or by groups studying evolution equations
that can be written in a Hilbert space setting, it is important to
stress that both communities actually work with the same
inequality. To see this, recall that in the metric space framework,
the entropy production functional $\D$ of a proper functional
$\E : \mathfrak{M}\to (-\infty,+\infty]$ with strong upper gradient
$g$ is given by~\eqref{eq:44}, which in applications $g$, is usually
given by the descending slope $\abs{D^{-}\E}$ of $\E$ (see
Definition~\ref{propo:lconvex-slope}). Thus, if $\theta$ satisfies, in
addition, that $\theta\in C^1((0,+\infty))$ and
$\lim_{s\to0+}\theta'(s)=+\infty$, then
\begin{center}
    \emph{K\L{}-inequality~\eqref{eq:KL32} is, in fact,
     EEP-inequality~\eqref{eq:13} for $\Phi(s):= \frac{1}{(\theta'(s))^{2}}$.}
\end{center}

In the pioneering works \cite{Lo63,Lo65}, \L{}ojasiewicz showed that
every \emph{real-analytic} energy functional $\E : \mathcal{U}\to \R$
defined on a open subset $\mathcal{U}\subseteq\R^{N}$ satisfies near
each equilibrium point $\varphi\in \mathcal{U}$ a {\bfseries
  \L{}ojasiewicz (\L{}-)inequality}
\begin{equation}
      \label{eq:Loj-Rd}
    \abs{\E(v\vert\varphi)}^{1-\alpha}\le
    C\,\norm{\nabla\E(v)}_{\R^{N}}\quad\text{for all $v\in
      \tilde{\mathcal{U}}_{r}:=\{v\in \mathcal{U}\,\vert\,\abs{v-\varphi}<r\}$}
\end{equation}
for some exponent $\alpha \in (0,1/2]$ and $r>0$. 
In these two papers, \L{}ojasiewicz developed a method
proving that the (\emph{local}) validity of \emph{gradient
  inequality}~\eqref{eq:Loj-Rd} in a neighborhood
$\tilde{\mathcal{U}}_{r}$ of $\varphi$ is
sufficient for establishing convergence to equilibrium in the metric
sense. 
Simon~\cite{Si83} was the first who made \L{}ojasiewicz's
gradient-inequality~\eqref{eq:Loj-Rd} available for evolution problems
formulated in an infinite dimensional Hilbert space framework and
generalized \L{}ojasiewicz's method. As an
application, Simon established the long-time convergence in the metric sense
of analytic solutions of semi-linear parabolic equations and of geometric flows. His ideas were
further developed by many authors
(see, for instance,~\cite{MR1609269,MR2019030,MR2226672, HaJenBook} concerning the
long-time asymptotic behavior of solutions of semi-linear heat and
wave equations, see~\cite{2014arXiv1409.1525F} concerning gradient
flows associated with geometric flows).

The condition \emph{$\E$ being real-analytic} is rather a geometric
property than a regularity property of $\E$. This is well demonstrated by
the gradient system in $\R^{2}$ given by the
\emph{Mexican hat} functional $\E$ due to Palis and de Melo
\cite[p~14]{MR669541}. In this example, $\E$ belongs to the class
$C^{\infty}(\R^{2})$, but $\E$ admits a bounded gradient flow $v$ with
an $\omega$-limit set $\omega(v)$ which is isomorphic to the unit
circle $S^{1}$. The geometric properties of real-analytic functions
and \L{}-inequality~\eqref{eq:Loj-Rd} were studied systematically with
tools from algebraic geometry and generalized to the class of
\emph{definable} functionals (see~\cite{MR1633348}). By introducing
the concept of a \emph{talweg curve}, Kurdyka~\cite{MR1644089} showed
that every definable $C^{1}$ functional $\E$ on $\U\subseteq \R^{N}$
satisfies near every equilibrium point a
K\L{}-inequality~\eqref{eq:KL32} and with this, he established of
every bounded definable gradient flow in $\R^{N}$, the trend to
equilibrium in the metric sense.

First versions of \emph{local} and \emph{global}
K\L{}-inequality~\eqref{eq:KL32} for proper, lower semicontinuous, and
(semi-)convex functionals $\E : H\to (-\infty,+\infty]$ on a Hilbert
space $H$ were introduced by Bolte et al.~\cite{Bolte2010} (see
also~\cite{MR2274510} and~\cite{MR3834660}). They adapted
Kurdyka's notion of a talweg curve to the Hilbert space framework and
characterized the validity of a K\L{}-inequality~\eqref{eq:KL32} with
the existence of a talweg. In addition, first formulations of
K\L{}-inequality~\eqref{eq:KL32} in a metric space framework were also
given in \cite{Bolte2010} (see also the recent
work~\cite{MR3832005}).

In this paper, we introduce \emph{local} and \emph{global}
K\L{}-inequalities~\eqref{eq:KL32} for proper functionals
$\E : \mathfrak{M}\to (-\infty,+\infty]$ defined on a metric space
$(\mathfrak{M},d)$ (see Definition~\ref{def:KL-inequality} in
Section~\ref{subsec:KL-inequality-prliminary}). Our definition here is
slightly different to the one in~\cite{Bolte2010,MR3832005}, but
consistent with one in the Hilbert space framework given, for
instance, by~\cite{MR3834660}. This enables us to provide
new fine tools for determining the trend to equilibrium in both the
\emph{entropy sense} (in Section~\ref{subsec:dynamicalsystem}) and the
\emph{metric sense} (in Section~\ref{secbehaviour}) of gradient flows
in $\mathfrak{M}$. More precisely, we show in
Theorem~\ref{thm:finite-length} that if $\E$ is bounded from below and
satisfies a K\L{}-inequality~\eqref{eq:KL32} on a set $\U$ then every
gradient flow $v$ of $\E$ satisfying $v([t_{0},+\infty))\subseteq\U$
for some $t_{0}\ge 0$ has finite length. In Section~\ref{sec:sufKL},
we study sufficient conditions implying that a functional $\E$ on
$\mathfrak{M}$ satisfies K\L{}-inequality~\eqref{eq:KL32} near an
equilibrium point $\varphi$ of $\E$. In particular,
Theorem~\ref{thm:3talweg-kl} provides optimal conditions on the
\emph{talweg curve} in $\mathfrak{M}$ ensuring the validity of a
K\L{}-inequality~\eqref{eq:KL32} by $\E$ near $\varphi$. In
Section~\ref{sec:chara-KL-talweg}, we characterize the (local)
validity of a K\L{}-inequality~\eqref{eq:KL32} by a functional $\E$
with the existence of a talweg curve. We adapt \L{}ojasiewicz's and
Kurdyka's convergence method from \cite{Lo63,Lo65}
and~\cite{MR1644089} to the metric space framework
(Section~\ref{secbehaviour}) and establish the trend to equilibrium in
the metric sense of every gradient flow of $\E$
(Theorem~\ref{thm:convergence}). We define \L{}ojasiewicz's
inequality~\eqref{eq:Loj-Rd} for proper functionals $\E$ on
$\mathfrak{M}$ (see Definition~\ref{def:LS-inequality}) and deduce
from it \emph{decay rates} of the trend to equilibrium in the metric
and entropy sense, and give upper bounds on the extinction time of
gradient flows (Theorem~\ref{thm:decayrates}). Note, these results are
consistent with the Hilbert space framework (cf, for
instance,~\cite{MR2019030,MR2289546,Bolte2010}). Section~\ref{sec:Lyapunov-stability}
is concerned with the characterization of \emph{Lyapunov stable}
equilibrium points $\varphi$ under the assumption that $\E$ satisfies
a K\L{}-inequality~\eqref{eq:KL32} in a neighborhood of $\varphi$,
and in Section~\ref{sec:character-global-KLandET}, we demonstrate that
K\L-inequality~\eqref{eq:KL32} is equivalent to a (generalized)
ET-inequality~\eqref{eq:13-transport0}.

Before outlining some applications to the theory developed in
Section~\ref{sec:Hilbertspacesframework} of this paper, we briefly
review the example of the \emph{linear Fokker-Planck equation}
\begin{equation}
    \label{eq:35}
    \partial_{t}v=\Delta v+\nabla\cdot (v\nabla
    V)\qquad\text{on $M\times (0,+\infty)$,}
\end{equation}
where $M$ a complete $C^{2}$-Riemannian manifold of dimension
$N\ge 1$, $\mathcal{L}^{N}$ is the standard volume measure on $M$, and
$V\in C^{2}(M)$ is a given potential. It was demonstrated
in~\cite{JKO} (see also~\cite{MR1964483}) that if one fixes a \emph{reference} probability measure
$\nu=v_{\infty}\mathcal{L}^{N}\in \mathcal{P}_{2}(M)$ with
$v_{\infty}=e^{-V}$, then solutions $v$ of equation~\eqref{eq:35} can
be written as the the probability distribution $v(t)$ of the gradient
flows $\mu(t)=v(t)\,\mathcal{L}^{N}$ generated by the
\emph{Boltzmann $H$-functional}
  \begin{equation}
    \label{eq:39}
    \mathcal{H}(\mu\vert \nu):=\int_{M}
    \tfrac{\td\mu}{\td\nu}\,\log\tfrac{\td\mu}{\td\nu}\,\td\nu\qquad
    \text{for every $\mu=v \mathcal{L}^{N}\in \mathcal{P}_{2}(M)$.}
\end{equation}
Since
\begin{displaymath}
  \mathcal{H}(\mu\vert \nu)=\int_{M}v\,\log v\,\dx + \int_{M} v\,V\,\dx=:\E(v),
\end{displaymath}
one sees that EEP-inequality~\eqref{eq:13} is the {\bfseries logarithmic
Sobolev inequality}
  \begin{equation}
    \label{eq:123}
    \mathcal{H}(\mu\vert \nu)\le \frac{1}{2\lambda}
    \mathcal{I}(\mu\vert \nu),
  \end{equation}
 for some $\lambda>0$, which holds true if $D^2(V)+\textrm{Ric}\ge \lambda$
 due to~\cite{MR889476}. The functional
 $\mathcal{I}(\mu\vert \nu)$ in~\eqref{eq:123} is called the \emph{relative Fisher
   information} of $\mu$ with respect to $\nu$ and coincides
 with the entropy production
\begin{displaymath}
  \D(\mu)=
  \int_{M}\labs{\nabla \log \tfrac{\td\mu}{\td\nu}}^{2}\,\td\mu
  = 4 \int_{M}\labs{\nabla \sqrt{\tfrac{\td\mu}{\td\nu}}}^{2}\,\td\nu
  =: \mathcal{I}(\mu\vert \nu).
\end{displaymath}
Thus, in other words, the logarithmic Sobolev inequality~\eqref{eq:123} is
actually K\L{}-inequality~\eqref{eq:KL32} for
$\theta(s)=2\,c\abs{s}^{-\frac{1}{2}}s$. Moreover, ET-inequality~\eqref{eq:13-transport0} is, in fact, {\bfseries
  Talagrand's entropy transportation inequality}
(cf~\cite{OttoVillani})
  \begin{equation}
    \label{eq:41}
    W_{2}(\mu,\nu)\le \sqrt{\tfrac{2}{\lambda} \mathcal{H}(\mu\vert\nu)}.
\end{equation}
%

Thanks to Corollary~\ref{equivalent1} of this paper,
Talagrand's inequality~\eqref{eq:41} and the logarithmic Sobolev
inequality~\eqref{eq:123} are equivalent to each other. 


Section~\ref{sec:application} is dedicated to applications:
Section~\ref{sec:Hilbertspacesframework} is concerned with the Hilbert
space framework and Section~\ref{GradProbmeasures} with the
Wasserstein framework. Section~\ref{subsubsec:Extinction-Dirichlet}
and~\ref{subsubsec:Extinction-Neumann} provide two examples of
gradient flows associated with the \emph{total variational
  flow}. These examples illustrate that an
K\L-inequality~\eqref{eq:KL32} involving known growth functions
$\theta$ provide upper bounds on the extinction time of gradient
flows. 
In Section~\ref{GradProbmeasures}, we establish new
(generalized) \emph{entropy transportation inequalities},
\emph{logarithmic Sobolev inequalities} and \emph{HWI-inequalities}
(see Theorem~\ref{thm:inequalitiesWp} and
cf~\cite{OttoVillani,MR2016985,MR2053603} and~\cite{MR2079071})
associated with proper, lower semicontinuous energy functionals $\E$
defined on the $p$-Wasserstein spaces $\mathcal{P}_{p}(\Omega)$,
($\Omega\subseteq \R^{N}$ an open set and $1<p<\infty$). We recall,
the abbreviation \emph{HWI} was introduced by Otto and
Villani~\cite{OttoVillani}, where $H$ denotes the relative
\emph{Boltzmann $H$-functional}, $W$ the \emph{$2$-Wasserstein distance}, and
$I$ the relative \emph{Fisher information}. The probability densities
of the gradient flows of these functionals $\E$ are solutions of
parabolic doubly nonlinear equations. We prove exponential decay rates
and finite time of extinction of those gradient flows
(Corollary~\ref{cor:exprate-RN}). We note that all results in
Section~\ref{sec:application} provide new insights in the study of the
examples.

During the preparation of this paper, we got aware of the recent
work~\cite{MR3832005} by Blanchet and Bolte. They show in the case
$\lambda=0$ that for proper, lower semicontinuous and
\emph{$\lambda$-geodesically convex} functionals $\E$ on
$\mathcal{P}_{2}(\R^{N})$, a (local) \L{}ojasiewicz -inequality
is equivalent to a (local and generalized)
ET-inequality~\eqref{eq:13-transport0}. Then, as an application of
this, they establish the equivalence between Talagrand's
inequality~\eqref{eq:41} and the logarithmic Sobolev inequality~\eqref{eq:123} for
the Boltzmann $H$-functional~\eqref{eq:39} on $\R^{N}$ for $V=0$. The
method in~\cite{MR3832005} is based on a similar idea to ours (cf
Theorem~\ref{thm:charact-local-KLET}), but our results presented in
Section~\ref{sec:character-global-KLandET} are concerned with a much
general framework. In our forthcoming paper, we show how these results
help us to study the geometry of metric measure length space with a
Ricci-curvature bound in the sense of Sturm~\cite{Sturm1,Sturm} and
Lott-Villani~\cite{LV1}.

We continue this paper with Section~\ref{sec:gradientflows-metric} by
summarizing some important notions and results from~\cite{AGS-ZH}
on the theory of gradient flows in metric spaces.

\subsection*{Acknowledgment}

We warmly thank the anonymous referee for reading the manuscript of
this paper very thoroughly, carefully and in depth.
 
Part of this work was made during a research stay at the Universitat
de Val\`encia by the first author and a second one at the University
of Sydney by the second author. We are very grateful for the kind invitations
   and the hospitality.  The second
  author has been partially supported by the Spanish MINECO and FEDER,
  project MTM2015-70227-P. 

  Both authors warmly thank Professor Luigi Ambrosio for the helpful and
  interesting discussion on $RCD(K,\infty)$ spaces. The first author
  also thanks warmly Professor Adam Parusi\'nski for sharing his
  knowledge on Kurdyka's inequality in $\R^{d}$.

\section{A brief primer on gradient flows in metric spaces}
\label{sec:gradientflows-metric}

Throughout this section $(\mathfrak{M},d)$ denotes a complete metric
space provided nothing more specifically was mentioned and
$-\infty\le a<b\le +\infty$.

\subsection{Metric derivative of curves in metric spaces}

Here, $1\le p\le \infty$. We begin with the following definition.

\begin{definition}[{\bfseries \emph{The space $AC^p$}}]
  A curve $v : (a,b)\to \mathfrak{M}$ is said to belong to the class
  $AC^{p}(a,b;\mathfrak{M})$ if there exists an $m\in L^{p}(a,b)$
  satisfying
  \begin{equation}
    \label{eq:1}
    d(v(s),v(t))\le \int_{s}^{t}m(r)\,\dr\qquad\text{for all $a<s\le t<b$.}
  \end{equation}
  Similarly, a curve $v : (a,b)\to \mathfrak{M}$ belongs to
  the class $AC^{p}_{loc}(a,b;\mathfrak{M})$ if there is an $m\in L^{p}_{loc}(a,b)$
  satisfying~\eqref{eq:1}.
\end{definition}

\begin{notation}
 In the case $p=1$, we simply write
  $AC(a,b;\mathfrak{M})$ for the class
  $AC^{1}(a,b;\mathfrak{M})$ and
  $AC_{loc}(a,b;\mathfrak{M})$ instead of
  $AC^{1}_{loc}(a,b;\mathfrak{M})$. In addition, if $a>-\infty$, then we write $v\in
    AC^{p}_{loc}([a,b);\mathfrak{M})$ to stress that the function $m$
    satisfying~(2.1) belongs to $L^{p}_{loc}([a,b))$.
\end{notation}

\begin{remark}
  \label{rem:AC-left-right-limits}
  Since $L^{p}_{loc}(a,b)\subseteq L^{1}_{loc}(a,b)$, the class
  $AC_{loc}^{p}(a,b;\mathfrak{M})$ is included in
  $AC_{loc}(a,b;\mathfrak{M})$. If $a>-\infty$,
  then for every $v\in AC(a,b;\mathfrak{M})$ and every sequence
  $(s_{n})_{n\ge 1}\subseteq (a,b)$ converging to $a$,
  inequality~\eqref{eq:1} implies that $(v(s_{n}))_{n\ge 1}$ is a
  Cauchy sequence in $\mathfrak{M}$. Thus, the completeness of
  $\mathfrak{M}$ implies that the limit
  \begin{displaymath}
    \lim_{t\to a+}v(t)=:v(a+)\qquad\text{exists in $\mathfrak{M}$}
  \end{displaymath}
  and~\eqref{eq:1} holds also for $s=a$. Analogously, if $b<+\infty$,
  then there is $v(b-)\in \mathfrak{M}$ such that
  $\lim_{t\to b-}v(t)=v(b-)$ exists in $\mathfrak{M}$. In particular,
  if $a>-\infty$ and $b<+\infty$, then every curve
  $v\in AC(a,b;\mathfrak{M})$, inequality~\eqref{eq:1} holds for $s=a$
  and $t=b$, $v$ is absolutely continuous on the closed interval
  $[a,b]$ and hence, $v$ is uniformly continuous on $[a,b]$.
\end{remark}

\begin{proposition}[{\cite[Theorem~1.1.2]{AGS-ZH}, {\bf Metric derivative}}]\label{MDer}
  \label{propo:1}
   For every curve $v\in AC^{p}(a,b;\mathfrak{M})$, the limit
  \begin{equation}
    \label{eq:2}
    \abs{v'}(t):=\lim_{s\to t}\frac{d(v(s),v(t))}{\abs{s-t}}
  \end{equation}
  exists for a.e. $t\in (a,b)$ and $\abs{v'}\in L^{p}(a,b)$
  satisfies~\eqref{eq:1}. Moreover, among all
  functions $m\in L^{p}(a,b)$ satisfying~\eqref{eq:1}, one has
  $\abs{v'}(t)\le m(t)$ for a.e. $t\in (a,b)$.
\end{proposition}

\begin{definition}[{\bfseries \emph{Metric derivative}}]
  \label{def:metric-derivative}
  For a curves $v\in
  AC^{p}(a,b;\mathfrak{M})$, one calls the function $\abs{v'}\in
  L^{p}(a,b)$ given by~\eqref{eq:2} the \emph{metric derivative} of $v$.
\end{definition}

\begin{lemma}[{\cite[Lemma~1.1.4]{AGS-ZH}, {\bf Arc-length
      reparametrization}}]
  \label{lem:arclengthreparametrisation}
  Let $v\in AC(a,b;\mathfrak{M})$ with length
  \begin{equation}
    \label{eq:52}
    \gamma=\gamma(v):=\int_{a}^{b}\abs{v'}(t)\,\dt.
  \end{equation}
 Then there exists an increasing absolutely continuous map
 \begin{displaymath}
   s : (a,b)\to [0,\gamma]\qquad\text{satisfying $s(0+)=0$, $s(b-)=\gamma$,}
 \end{displaymath}
 and a curve $\hat{v}\in AC^{\infty}(0,\gamma;\mathfrak{M})$ such that
 \begin{displaymath}
   v(t)=\hat{v}(s(t)),\quad \abs{\hat{v}^{\prime}}=1.
 \end{displaymath}
\end{lemma}

\subsection{Strong upper gradients of $\E$}
\label{subsec:strong-upper-gradient}
In this subsection, we introduce the first main tool for establishing
the existence gradient flows in metric spaces.

\begin{definition}[{\bfseries \emph{Proper (energy) functionals}}]
  A functional $\E : \mathfrak{M}\to (-\infty,\infty]$ is called
  \emph{proper} if there is a $v\in \mathfrak{M}$ such that
  $\E(v)<+\infty$ and the set $D(\E):=\{v\in
    \mathfrak{M}\,\vert\,\E(v)< +\infty\}$ is called the
    \emph{effective domain} of $\E$.
\end{definition}

\begin{definition}[{\bfseries \emph{Strong upper gradient}}]\label{def:strong-uppergradient}
  For a proper functional $\E : \mathfrak{M}\to (-\infty,\infty]$, a
  proper functional $g : \mathfrak{M}\to [0,+\infty]$ is called a
  \emph{strong upper gradient of $\E$} if for every
  curve $v\in AC(a,b;\mathfrak{M})$, the composition function
  $g\circ v : (0,+\infty)\to [0,\infty]$ is Borel-measurable and
  \begin{equation}
    \label{eq:3}
    \abs{\E(v(t))-\E(v(s))}\le \int_{s}^{t}
    g(v(r))\,\abs{v'}(r)\,\dr\qquad\text{for all $a<s\le t<b$.}
  \end{equation}
\end{definition}

\begin{remark}
  \mbox{}
  \begin{enumerate}
   \item We always assume that the effective domain $D(g)$ of a strong upper
    gradient $g$ of $\E$ is a subset of the effective domain $D(\E)$
    of $\E$.

   \item  In Definition~\ref{def:strong-uppergradient}, it is
    not assumed that
    \begin{equation}
      \label{eq:7}
      (g\circ v)\,\abs{v'}\in L^{1}_{loc}(a,b).
    \end{equation}
    Thus, the value of the integral in~\eqref{eq:3} might be infinity.
  \end{enumerate}
\end{remark}

The following notion is taken from~\cite{AGS-invent}.

\begin{definition}[{\bfseries \emph{Ascending and descending slope}}]
 \label{def:descding-slope}
 For a given functional $\E : \mathfrak{M}\to (-\infty,\infty]$, the \emph{local Lipschitz
  constant} at $u\in D(\E)$ is given by
\begin{displaymath}
 \abs{D \E} (u):=  \limsup_{v \to u} \frac{\abs{\E(v) - \E(u)}}{d(v,u)}
\end{displaymath}
and $\abs{D \E}(u):= +\infty$ for every $u\in \mathfrak{M}\setminus
D(\E)$. Similarly, the \emph{ascending slope} $\abs{D^{+}\E} :
\mathfrak{M}\to [0,+\infty]$ of $\E$ is defined by
\begin{displaymath}
\abs{D^+\E}(u):=
\begin{cases}
  \displaystyle\limsup_{v \to u} \frac{[\E(v) - \E(u)]^+}{d(v,u)} & \text{if $u\in
    D(\E)$,}\\
  +\infty & \text{if otherwise}
\end{cases}
\end{displaymath}
and the \emph{descending slope} $\abs{D^{-}\E} :
\mathfrak{M}\to [0,+\infty]$ of $\E$ is given by
\begin{displaymath}
  \abs{D^- \E}(u):=
\begin{cases}
  \displaystyle\limsup_{v \to u} \frac{[\E(v) - \E(u)]^-}{d(v,u)} & \text{if $u\in
    D(\E)$,}\\
  +\infty & \text{if otherwise.}
\end{cases}
\end{displaymath}
\end{definition}

We come back to the descending slope $\abs{D^{-}\E}$ in
Section~\ref{subsec:geo-convex}. Our next proposition follows
immediately from Definition~\ref{def:strong-uppergradient}. We state
this standard result for later use.

\begin{proposition}
  \label{prop:Energy-abs-inequality}
  If $g : \mathfrak{M}\to [0,+\infty]$ is a strong upper gradient of
  a proper functional $\E : \mathfrak{M}\to (-\infty,\infty]$, then
  for every $v\in AC_{loc}(a,b;\mathfrak{M})$ satisfying~\eqref{eq:7}, the
  composition function $\E\circ v : (a,b)\to (-\infty,\infty]$ is
  locally absolutely continuous and
  \begin{equation}
    \label{eq:4}
    \labs{\fdt\E(v(t)}\le g(v(t))\,\abs{v'}(t)\qquad\text{for
      almost every $t\in (a,b)$.}
  \end{equation}
\end{proposition}

\subsection{$p$-Gradient flows in metric spaces}
\label{sec:curves-max-slope}

In this subsection, let $1<p<\infty$ and
  $p^{\mbox{}_{\prime}}:=\tfrac{p}{p-1}$ be the \emph{H\"older
    conjugate} of $p$ and we focus on the case $a=0$, $0<b=T\le +\infty$.\medskip

We begin by recalling the following definition taken
from~\cite{DeGiorgietalt} (see also~\cite{AGS-ZH}).

\begin{definition}[{\bfseries \emph{$p$-curve of maximal slope}}]
  \label{def:p-max-slope}
  For a proper functional $\E : \mathfrak{M}\to (-\infty,\infty]$ with
  strong upper gradient $g$, a curve $v : [0,+\infty)\to \mathfrak{M}$
  is called a \emph{$p$-curve of maximal slope of $\E$}, ($0<T\le +\infty$), if $v\in
  AC_{loc}(0,T;\mathfrak{M})$, $\E\circ v : (0,T)\to \R$ is non-increasing, and
\begin{equation}
 \label{eq:6}
    \fdt\E(v(t))\le -
      \tfrac{1}{p^{\mbox{}_{\prime}}}g^{p^{\mbox{}_{\prime}}}(v(t))-\tfrac{1}{p}\abs{v'}^{p}(t)
      \qquad\text{for almost every $t\in (0,T)$.}
 \end{equation}
\end{definition}

Some authors (cf~\cite{AGS-invent}) call
Definition~\ref{def:p-max-slope} the \emph{metric formulation of a
  gradient flow}. Since we agree with this idea, we rather use the
following terminology.

\begin{definition}[{\bfseries $p$-gradient flows in $\mathfrak{M}$}]
  \label{def:pgradientflow}
  For a proper functional $\E : \mathfrak{M}\to (-\infty,\infty]$ with
  strong upper gradient $g$, a curve $v : [0,T)\to\mathfrak{M}$
  is called a \emph{$p$-gradient flow of $\E$} or also a
  \emph{$p$-gradient flow in $\mathfrak{M}$}, ($0<T\le +\infty$), if
  $v\in AC_{loc}(0,T;\mathfrak{M})$ and $v$ is a $p$-curve of maximal
  slope of $\E$ with respect to $g$. For the sake of convenience, we
  call each $2$-gradient flow in $\mathfrak{M}$, simply,
  \emph{gradient flow}. Further, a $p$-gradient flow $v$ of $\E$ has
  \emph{initial value} $v_0 \in \mathfrak{M}$ if $v(0)=v_0$.
\end{definition}

For the sake of completeness, we give the following characterization of
\emph{$p$-gradient flows}, which is scattered in the
literature. But we omit its easy proof.

\begin{proposition}
  \label{propo:chara-p-curves}
  Let $\E : \mathfrak{M}\to (-\infty,\infty]$ be proper functional
  with strong upper gradient $g$, and let
  $v\in AC_{loc}(0,T;\mathfrak{M})$, $(0<T\le +\infty)$. Then the following
  statements are equivalent.
 \begin{enumerate}
 \item\label{propo:chara-p-curves-claim-1} $v$ is a $p$-gradient flow
   of $\E$.

   \item\label{propo:chara-p-curves-claim-2} $\E\circ v$
     is non-increasing on $(0,T)$ and for all $0<s<t<T$,
     \begin{equation}
       \label{eq:8}
        \E(v(s))-\E(v(t))\ge \tfrac{1}{p}\int_{s}^{t}\abs{v'}^{p}(r)\,\dr
        +\tfrac{1}{p^{\mbox{}_{\prime}}}\int_{s}^{t}g^{p^{\mbox{}_{\prime}}}(v(r))\, dr.
     \end{equation}

  \item\label{propo:chara-p-curves-claim-3} $\E\circ v$
    is non-increasing on $(0,T)$ and the ``energy dissipation equality''
    \begin{equation}
      \label{EDE}
        \E(v(s))-\E(v(t))=\tfrac{1}{p}\int_{s}^{t}\abs{v'}^{p}(r)\,\dr
        +\tfrac{1}{p^{\mbox{}_{\prime}}}\int_{s}^{t}g^{p^{\mbox{}_{\prime}}}(v(r))\, dr
     \end{equation}
     holds for all $0<s<t<T$.
 \end{enumerate}
 In particular, for every $p$-gradient flow $v$ of $\E$, one has that
 \begin{displaymath}
   v\in AC_{loc}^{p}(0,T;\mathfrak{M}),\qquad
   g\circ v \in L^{p^{\mbox{}_{\prime}}}_{loc}(0,T),
 \end{displaymath}
 and for almost every $t\in (0,T)$,
 \begin{equation}
   \label{eq:11}
    \fdt\E(v(t))=-\abs{v'}^{p}(t)=-g^{p^{\mbox{}_{\prime}}}(v(t))=-(g\circ
    v)(t)\cdot\abs{v'}(t).
 \end{equation}
 Moreover, if $v(0+)$ exists in $\mathfrak{M}$ with $v(0+)\in D(\E)$,
 then $v\in AC^{p}_{loc}([0,T);\mathfrak{M})$ with
 $g\circ v\in L^{p^{\mbox{}_{\prime}}}_{loc}([0,T))$,
 and~\eqref{EDE} holds for all $0\le s<t<T$.
\end{proposition}

 \begin{remark}
   Note, for every $p$-gradient flow $v$ of $\E$, one has that
    \begin{equation}
      \label{EDENew}
        \E(v(s))-\E(v(t))=\int_{s}^{t}\abs{v'}^{p}(r)\,\dr =
        \int_{s}^{t}g^{p^{\mbox{}_{\prime}}}(v(r))\, dr
     \end{equation}
     for all $0<s<t<+\infty$ due to \eqref{EDE} and \eqref{eq:11}.
 \end{remark}

We will also need the following more general version of $p$-gradient
flows of $\E$ (cf~\cite[Definition~15]{Bolte2010} in the Hilbert space framework).

\begin{definition}[{\bfseries piecewise $p$-gradient flows in $\mathfrak{M}$}]
 \label{def:piceswise-p-gradientflows}
Given a proper functional $\E : \mathfrak{M}\to (-\infty,+\infty]$
with strong upper gradient $g$, then a curve $v : [0,T)\to
\mathfrak{M}$ is called a \emph{piecewise
   $p$-gradient flow} in $\mathfrak{M}$, ($0<T\le+\infty$), if there
 is a 
 countable partition $\mathcal{P}=\{0=t_{0}, t_{1},
 t_{2},\dots\}$ of $[0,T)$ such that 
 \begin{enumerate}[itemsep=2pt,leftmargin=21pt]
   \item[(i)]  $v_{1}:=v_{\vert (0,t_{1}]}\in AC_{loc}((0,t_{1}];\mathfrak{M})$
     and $v_{1}$ is a $p$-gradient flow of $\E$ on $(0,t_{1})$, 
   \item[(ii)]  for every $i\ge 2$,
   $v_{\vert [t_{i-1},t_{i}]}\in AC(t_{i-1},t_{i};\mathfrak{M})$ and
   $v_{i} : [0,t_{i}]\to \mathfrak{M}$ defined by
   $v_{i}(t)=v(t_{i-1}+t)$, $(t\in [0,t_{i}])$, is a $p$-gradient flow of $\E$ on $(0,t_{i})$, 
 \item[(iii)] for every $i$, $j\ge 1$ with $i\neq j$, the two
   intervals $\E(v_{i}([0,t_{i}]))$ and $\E(v_{j}([0,t_{j}]))$ have at
   most one common point, where for $i=1$ or $j=1$,
   these sets are replaced by
   $\E(v_{1}((0,t_{1}]))$.
 \end{enumerate}
\end{definition}

\begin{remark}
  Trivially, every $p$-gradient flow $v : [0,T)\to \mathfrak{M}$ of
  $\E$ is a piecewise $p$-gradient flow of $\E$. But we emphasize that a piecewise $p$-gradient flow
  $v : (0,T)\to \mathfrak{M}$ does in general not belong to
  $AC_{loc}^{p}(0,T;\mathfrak{M})$. In fact, for a piecewise $p$-gradient
  flow $v$ on $(0,T)$, there is a countable partition $\mathcal{P}=\{0=t_{0}, t_{1},
 t_{2},\dots\}$ of $[0,T)$ such that for every $i\ge 1$, $v(t_{i}-)$
 and $v(t_{i}+)$ exist in $\mathfrak{M}$, but $v$ does not need to be continuous at $t_{i}$.
\end{remark}

 The existence of $p$-gradient flows in metric spaces is obtained
 variationally by minimizing the functional $\Phi_{p}(\tau,v_{0};\cdot) :
\mathfrak{M}\to (-\infty,\infty]$ given by
\begin{equation}\label{Eulerp}
  \Phi_p(\tau, v_{0}; U)= \frac{1}{p\tau^{p-1}}
  d^p(v_{0}, U) + \E(U),\qquad\text{($U\in \mathfrak{M}$),}
\end{equation}
 recursively on a  given partition of of the time interval
 $[0,+\infty)$, $v_{0}\in D(\E)$, and $\tau>0$. This was well elaborated in the
 book~\cite{AGS-ZH} and the existence of such curves was proved by
 using the concept of {\it minimizing movements} (due to De
 Giorgi~\cite{DeGiorgi}, see also \cite{Ambrosio}). Sufficient
conditions for the existence of minimizers of $\Phi_p$ are as
 follows:
 \begin{enumerate}
 \item[(i)] let $\sigma$ be a topology on $\mathfrak{M}$ satisfying
   \begin{equation}\label{eq:17}
     \begin{cases}
       &\text{$\sigma$ is weaker than the induced topology by the
         metric $d$ of $\mathfrak{M}$}\\
       &\text{and $d$ is sequentially $\sigma$-lower semicontinuous;}
     \end{cases}
 \end{equation}

 \item[(ii)] given $\tau>0$ and $v_{0}\in D(\E)$, then for every $c\in
   \R$, the sublevel set
   \begin{displaymath}
     \{U\in \mathfrak{M}\,\vert\,
     \Phi_p(\tau,v_{0};V)\le c\}\qquad\text{is $\sigma$-relatively compact;}
   \end{displaymath}
 \item[(iii)] for given $\tau>0$ and $v_{0}\in D(\E)$, the function $U\mapsto
   \Phi_p(\tau,v_{0};V)$ is $\sigma$-lower semicontinuous.
 \end{enumerate} 
 \medskip 

 For stating the first existence theorem of $p$-gradient flows, we
 also need to introduce the notion of the \emph{relaxed slope} of a
 functional $\E$, (or, also called the \emph{sequentially lower
   semicontinuous envelope} of $\abs{D^{-}\E}$).

 \begin{definition}[{\bfseries \emph{Relaxed Slope}}]
  Given a topology $\sigma$ on $\mathfrak{M}$
  satisfying~\eqref{eq:17}. Then, for a proper functional $\E : \mathfrak{M}\to
(-\infty,+\infty]$, the \emph{relaxed
   slope} $\abs{\partial^{-}\E} : \mathfrak{M}\to [0,+\infty]$ of $\E$ is defined by
   \begin{displaymath}
     \abs{\partial^{-}\E}(u) = \inf \left\{ \liminf_{n \to \infty}
       \abs{D^{-}\E} (u_n)
       \;\Big\vert \; \  u_n \stackrel{\sigma}{\rightharpoonup} u, \ \sup_{n \in \N} \{ d(u_n,u), \E(u_n) \} < \infty\right\}
   \end{displaymath}
   for every $u\in D(\E)$.
 \end{definition}

The following generation result of $p$-gradient flows in metric
 spaces is sufficient for our purpose, where we choose the weak
 topology $\sigma$ to be the topology induced by the metric $d$ of
 $\mathfrak{M}$.

\begin{theorem}[{\cite[Corollary 2.4.12]{AGS-ZH}}]\label{exist1}
 Let $\E : \mathfrak{M}\to (-\infty,+\infty]$ be a proper lower
  semicontinuous function on $\mathfrak{M}$ satisfying
  \begin{enumerate}
  \item[(i)] the functional $\Phi_{p}$ given by~\eqref{Eulerp}
    satisfies the ``convexity condition'': there is a $\lambda\in \R$
    such that for every $v_{0}$, $v_{1}\in D(\E)$, there is a curve
    $\gamma$ with $\gamma(0)=v_{0}$, $\gamma(1)=v_{1}$ and for all $0<\tau<\frac{1}{\lambda^{-}}$,
    \begin{displaymath}
      v\mapsto \Phi_{p}(\tau,v_{0},v)\qquad\text{is
        $(\tau^{-1}+\lambda)$-convex on $\gamma$,}
    \end{displaymath}
  \item[(ii)] for every $\alpha\in \R$, every sequence $(v_{n})_{n\ge 1}\subseteq
    E_{\alpha}:=\{v\in \mathfrak{M}\,\vert\,\E(v)\le \alpha\}$ with
    $\sup_{n,m}d(v_{n},v_{m})<+\infty$ admits a convergent subsequence
    in $\mathfrak{M}$,
      \item[(iii)] for every $\alpha\in \R$, the descending slope
    $\abs{D^{-}\E}$ is lower semicontinuous on the sublevel set $E_{\alpha}$,
  \end{enumerate}
  where $\lambda^{-}=\max\{0,-\lambda\}$ and
  $\frac{1}{\lambda^{-}}:=+\infty$ if $\lambda^{-}=0$.  Then, for
  every $v_{0}\in D(\E)$ there is a $p$-gradient flow of $\E$ with
  $v(0+)=v_{0}$.
\end{theorem}

\subsection{Geodesically convex functionals}
\label{subsec:geo-convex}

An important class of functionals
$\E : \mathfrak{M}\to (-\infty,\infty]$ satisfying sufficient
conditions to guarantee the existence of $p$-gradient flows of $\E$ is given by
the ones that are \emph{convex along constant speed geodesics} (see
Theorem~\ref{ExistsGC} below). To be more precise, we recall the following
definition (cf~\cite{AGS-ZH}).

\begin{definition}[{\bfseries \emph{$\lambda$-convexity}}]
  \label{def:lambdaconvex}
  For $\lambda\in \R$, a functional
  $\E : \mathfrak{M}\to (-\infty,\infty]$ is called
  \emph{$\lambda$-convex along a curve
    $\gamma : [0,1] \to \mathfrak{M}$} (we write
  $\gamma\subseteq \mathfrak{M}$) if
\begin{equation}
  \label{eq:27}
\E(\gamma(t)) \le (1-t)\E(\gamma(0)) + t \E(\gamma(1))-
\tfrac{\lambda}{2} t(1-t)d^{2}(\gamma(0), \gamma(1))
\end{equation}
for all $t \in [0,1]$, and $\E$ is called \emph{convex along
  a curve $\gamma \subseteq \mathfrak{M}$} if $\E$ is $\lambda$-convex
along $\gamma$ for $\lambda=0$.
\end{definition}

\begin{definition}[{\bfseries constant speed geodesics and
    $\lambda$-geodesic convexity}]
  \label{def:lambdageoconvex}
  A curve $\gamma : [0,1] \rightarrow \mathfrak{M}$ is said to be a
  \emph{(constant speed) geodesic} connecting two points
  $v_0, v_1 \in \mathfrak{M}$ if $\gamma(0) = v_0$, $\gamma(1) = v_1$
  and
  \begin{displaymath}
    d(\gamma(s), \gamma(t)) = (t-s) d(v_0, v_1) \quad
    \text{for all $s$, $t \in [0,1]$ with $s\le t$.}
  \end{displaymath}
  A metric space $(\mathfrak{M}, d)$ with the property that for every
  two elements $v_{0}$, $v_{1} \in \mathfrak{M}$, there is at least
  one constant speed geodesic $\gamma\subseteq\mathfrak{M}$ connecting
  $v_{0}$ and $v_{1}$ is called a \emph{length space}. Given
  $\lambda\in \R$, a functional
  $\E : \mathfrak{M}\to (-\infty,\infty]$ is called
  \emph{$\lambda$-geodesically convex} if for every $v_{0}$,
  $v_{1} \in D(\E)$, there is a constant speed geodesic
  $\gamma\subseteq\mathfrak{M}$ connecting $v_{0}$ and $v_{1}$ such
  that $\E$ is $\lambda$-convex along $\gamma$ and $\E$ is called
  \emph{geodesically convex} if for every $v_{0}$,
  $v_{1} \in \mathfrak{M}$, there is a constant speed geodesic
  $\gamma\subseteq\mathfrak{M}$ connecting $v_{0}$ and $v_{1}$ and
  $\E$ is convex along $\gamma$.
\end{definition}

The descending slope $\abs{D^{-}\E}$ of a $\lambda$-geodesically
convex functional $\E$ admits several important properties, which we
recall now for later use.

\begin{proposition}[{\cite[Theorem~2.4.9 \& Corollary~2.410]{AGS-ZH}}]
  \label{propo:lconvex-slope}
  For $\lambda\in \R$, let $\E : \mathfrak{M}\to (-\infty,\infty]$ be a proper
  $\lambda$-geodesically convex functional on a length space $\mathfrak{M}$.
  Then, the following statements hold.
  \begin{enumerate}

  \item The descending slope $\abs{D^{-}\E}$ of $\E$ has the representation
    \begin{equation}\label{messi00}
      \abs{D^{-}\E}(u) =\sup_{v \not= u} \left[ \frac {\E(u) -
          \E(v)}{d(u,v)} + \frac{\lambda}{2} d(u,v)\right]^+
      \quad\text{for every $u\in D(\E)$.}
    \end{equation}

  \item For every $u\in D(\E)$, $\abs{D^{-}\E}(u)$ is the smallest
    $L \ge 0$ satisfying
    \begin{displaymath}
      \E(u) - \E(v) \le  L\, d(v,u) -
      \frac{\lambda}{2}d(v,u)^2 \qquad \text{for every } v \in \mathfrak{M}.
    \end{displaymath}
    In particular,
    \begin{equation}
      \label{mes2}
      \E(u) - \E(v) \le  \abs{D^{-}\E}(u)\, d(v,u) -
      \frac{\lambda}{2}d^{2}(v,u) \qquad \text{for every } v \in \mathfrak{M}.
    \end{equation}

  \item \label{propo:lconvex-slope-claim3} If $\E$ is lower
    semicontinuous, then the descending slope $\abs{D^{-}\E}$ of $\E$
    is a strong upper gradient of $\E$, and the relaxed slope
    $\abs{\partial^{-}\E}$ fulfills
    $\abs{\partial^{-}\E}=\abs{D^{-}\E}$. In particular,
    $\abs{D^{-}\E}$ is lower semicontinuous along sequences with
    bounded energy.
  \end{enumerate}
\end{proposition}

Throughout this paper, we use the following notion.

\begin{definition}[{\bfseries \emph{Relative Entropy}}]
 For a proper functional $\E  : \mathfrak{M}\to (-\infty,\infty]$ and
  $\varphi\in D(\E)$, one calls the (shifted) functional $\E(\cdot\vert\varphi)  :
  \mathfrak{M}\to (-\infty,\infty]$ given by
  \begin{equation}
    \label{eq:32}
    \E(v\vert\varphi)=\E(v)-\E(\varphi)\qquad\text{for every
      $v\in \mathfrak{M}$}
  \end{equation}
  the \emph{relative entropy} or \emph{relative energy} of $\E$ with
  respect to $\varphi$.
\end{definition}

Due to Proposition~\ref{propo:lconvex-slope}, we have that the
following result holds for $\lambda$-geodesically functionals $\E$ on
a length space $\mathfrak{M}$
(cf~\cite[Proposition~42]{Bolte2010} in the Hilbert space
framework). In the next proposition, we denote by
$[\E=c]$ the level set $\{v\in
\mathfrak{M}\,\vert\,\E(v)=c\}$ of a proper functional $\E : \mathfrak{M}\to
(-\infty,\infty]$ at level $c\in \R$ (see also
Notation~\ref{notation:1} below).

\begin{proposition}
  \label{propo:inflsc}
  For $\lambda\in \R$, let $\E : \mathfrak{M}\to (-\infty,\infty]$ be
  a proper, lower semicontinuous and $\lambda$-geodesically convex
  functional on a length space $\mathfrak{M}$ and for
  $\varphi\in D(\E)$, $R>0$, let
  $\mathcal{D}\subseteq [\E(\cdot,\vert\varphi)]^{-1}((-\infty,R])$ a
  nonempty compact set. Then, the function $s_{\mathcal{D}} :
  (-\infty,R] \to [0,+\infty]$ given by
  \begin{equation}
    \label{eq:29bis}
    s_{\mathcal{D}}(r)= \inf_{\hat{v}\in \mathcal{D}\cap
      [\E=r+\E(\varphi)]} \abs{D^{-}\E}(\hat{v})
    \qquad\text{for every $r\in (-\infty,R]$}
  \end{equation}
  is lower semicontinuous.
\end{proposition}

\begin{proof}
  For $\alpha\ge 0$, let $(r_{n})_{n\ge 1}\subseteq (-\infty,R]$ be a
  sequence converging to some $r\le R$ such that
  $s_{D}(r_{n})\le \alpha$ for every $n\ge 1$. Further, let
  $\varepsilon>0$. Then, there is a sequences $(v_{n})_{n\ge
    1}\subseteq \mathcal{D}\cap \E^{-1}(\{r_{n}+\E(\varphi)\})$ satisfying
  \begin{displaymath}
    \abs{D^{-}\E}(v_{n})<s_{\mathcal{D}}(r_{n})-\varepsilon\qquad
    \text{for every $n\ge 1$.}
  \end{displaymath}
  Since $\mathcal{D}$ is compact, there is an
  element $v\in \mathcal{D}$ such that after possibly passing to a
  subsequence $v_{n}\to v$ in $\mathfrak{M}$. Since
  $\lim_{n\to +\infty}\E(v_{n})=\lim_{n\to +\infty} r_{n}+\E(\varphi)=
  r+\E(\varphi)$ and since $\E$ is lower
  semicontinuous,
  \begin{math}
    \E(v)\le r+\E(\varphi).
  \end{math}
  Moreover, since $\abs{D^{-}\E}$ is lower semicontinuous
  (Proposition~\ref{propo:lconvex-slope}), $v\in D(\abs{D^{-}\E})$
  and
  \begin{displaymath}
    \abs{D^{-}\E}(v)\le \liminf_{n\to\infty}\abs{D^{-}\E}(v_{n}) \le \alpha-\varepsilon.
  \end{displaymath}
  On the other hand, by~\eqref{mes2},
  \begin{displaymath}
    \E(v_{n}) \le \E(v)+ \abs{D^{-}\E}(v)\, d(v_{n},v) -
      \frac{\lambda}{2}d^{2}(v_{n},v)\qquad\text{for every $n\ge 1$.}
  \end{displaymath}
  Sending $n\to +\infty$ in this inequality and using that $\lim_{n\to +\infty}\E(v_{n})=
  r+\E(\varphi)$, it follows that
  \begin{displaymath}
    r + \E(\varphi)-\E(v)= 0.
  \end{displaymath}
  Therefore, for every $\varepsilon>0$, there is at least one
  \begin{displaymath}
    v\in \mathcal{D}\subseteq
    \E^{-1}(\{r+\E(\varphi)\})\quad\text{ such that }\quad
    \abs{D^{-}\E}(v)\le\alpha-\varepsilon,
  \end{displaymath}
  implying that $s_{\mathcal{D}}(r)\le \alpha$.
\end{proof}

\begin{notation}
 For $\varphi\in \mathfrak{M}$ and $r>0$, we denote by
  \begin{displaymath}
    B(\varphi,r):=\Big\{v\in \mathfrak{M}\,\Big\vert\,d(v,\varphi)<r\Big\}
 \end{displaymath}
 the \emph{open ball in $\mathfrak{M}$ of radius $r$ and centered at
   $\varphi$}.
\end{notation}

\begin{definition}[{\bfseries \emph{Local \& global Minimum}}]
  Given a proper functional $\E : \mathfrak{M}\to(-\infty,+\infty]$,
  an element $\varphi\in D(\E)$ is called a \emph{local minimum} of
  $\E$ if there is an $r>0$ such that
  \begin{equation}
    \label{eq:42}
    \E(\varphi)\le \E(v)\qquad \text{for all $v\in B(\varphi,r)$,}
  \end{equation}
  and $\varphi$ is said to be a \emph{global minimum} of $\E$
  if~\eqref{eq:42} holds for all $v\in \mathfrak{M}$. Moreover, we
  denote by $\textrm{argmin}(\E)$ the set of all global minimizers
  $\varphi$ of $\E$.
\end{definition}

Our next proposition shows that for $\lambda$-geodesically convex
functionals $\E$ with $\lambda\ge0$, every \emph{local} minimum is a
global one and the set all \emph{points of equilibrium} of $\E$
(see~Definition~\ref{def:equilibrium-points} below) coincides with the set of
global minimizers $\textrm{argmin}(\E)$ of $\E$.

\begin{proposition}[{\bfseries Fermat's rule}]
   \label{propo:charcter-local-min}
    Let $\E : \mathfrak{M}\to(-\infty,+\infty]$ be a proper
    functional. Then the following statements hold.
    \begin{enumerate}
    \item \label{propo:charcter-local-min-claim-1} If $\varphi \in
      D(\E)$ is local minimum of $\E$, then $\abs{D^{-}\E}(\varphi)=0$.

    \item \label{propo:charcter-local-min-claim-2} If for $\lambda\ge 0$,
      $\E$ is $\lambda$-geodesically convex, then
        \begin{displaymath}
          \textrm{argmin}(\E) = \Big\{ \varphi \in D(\abs{D^{-}\E}) \;\Big\vert\;
          \abs{D^{-}\E}(\varphi) = 0 \Big\}.
        \end{displaymath}
      \item \label{propo:charcter-local-min-claim-3} If $\E$ is
        $\lambda$-geodesically convex for $\lambda>0$, then $\E$
        admits at most one minimum.
\end{enumerate}
\end{proposition}

\begin{proof}
  Claim~\eqref{propo:charcter-local-min-claim-1} is a direct
  consequence of the definition of the descending slope
  $\abs{D^{-}\E}$ of $\E$ (cf~Definition~\ref{def:descding-slope}). To
  see that claim~\eqref{propo:charcter-local-min-claim-2} holds, it
  remains to show that every $\varphi\in D(\abs{D^{-}\E})$ satisfying
  $\abs{D^{-}\E}(\varphi) = 0$ is a global minimum of $\E$. But this
  follows from the characterization~\eqref{messi00} of $\abs{D^{-}\E}$
  in
  Proposition~\ref{propo:lconvex-slope}. Claim~\eqref{propo:charcter-local-min-claim-3}
  is a consequence of inequality~\cite[(2.4.20) in
  Lemma~2.4.13]{AGS-ZH}, but for the sake of completeness, we shell
  briefly sketch the proof here. To see this, let $\varphi \in D(\E)$
  be a minimizer of $\E$ and suppose that $\psi \in D(\E)$ is another
  minimizer of $\E$. By hypothesis, there is a constant speed geodesic
  $\gamma : [0,1]\to \mathfrak{M}$ such that $\gamma(0)=\varphi$ and
  $\gamma(1)=\psi$ and $\E$ is $\lambda$-convex along $\gamma$. Thus,
   \begin{displaymath}
     0\le \frac{\E(\gamma(t))-\E(\gamma(0))}{t}\le
     \E(\gamma(1))-\E(\gamma(0))
     -\frac{\lambda}{2}(1-t)d^{2}(\gamma(0),\gamma(1))
   \end{displaymath}
   for every $t\in (0,1)$. We fix $t\in (0,1)$. Then, since $\lambda>0$ and $\varphi=\gamma(0)$
   is a minimizer of $\E$, the last inequality implies that
   \begin{displaymath}
     \frac{\lambda}{2}(1-t)d^{2}(\gamma(0),\gamma(1))\le
     \E(\gamma(1))-\E(\gamma(0)) \le 0
   \end{displaymath}
   implying that $\varphi = \gamma(0) = \gamma(1) = \psi$.
\end{proof}

For proper functionals $\E$ which are $\lambda$-geodesically convex for
$\lambda>0$ and admit a (global) minimizer $\varphi\in \mathfrak{M}$, every
gradient flow $v$ of $\E$ trends exponentially to $\varphi$ in the
\emph{metric sense} (see Corollary~\ref{Gconv},
cf~\cite[Theorem~2.4.14]{AGS-ZH} and~\cite{CMVIbero}). This
stability result is due to
inequality~\eqref{messi1} in our next proposition.

\begin{proposition}[{\cite[Lemma 2.4.8 \& Lemma 2.4.13]{AGS-ZH}}]\label{lambdaconvex1}
  For $\lambda>0$, let $\E : \mathfrak{M}\to (-\infty,\infty]$ be a proper, lower
  semicontinuous, $\lambda$-geodesically convex functional on a length
  space $(\mathfrak{M},d)$. Then,
  \begin{equation}
    \label{messi1}
    \left(\E(v)-\inf_{\hat{v}\in \mathfrak{M}}\E(\hat{v})\right)\le
    \frac{1}{2\lambda}\,\abs{D^{-}\E}^{2}(v)
    \qquad\text{for all $v\in D(\E)$.}
  \end{equation}
  In addition, if $\E$ satisfies the coercivity condition:
  \begin{equation}
    \label{eq:77}
    \text{there is $u_{\ast}\in D(\E)$, $r_{\ast}>0$ s.t. }
    m_{\ast}:=\inf_{u\in \mathfrak{M} : d(u,u_{\ast})\le r_{\ast}}\E(v)>-\infty,
  \end{equation}
  then there is a unique minimizer $\varphi\in D(\E)$ of $\E$ and
  \begin{equation}
    \label{messi1entropy}
    \frac{\lambda}{2}d^{2}(v,\varphi)\le \E(v\vert\varphi)\le
    \frac{1}{2\lambda}\,\abs{D^{-}\E}^{2}(v)
    \qquad\text{for all $v\in D(\E)$.}
  \end{equation}
\end{proposition}

\begin{remark}
  Under the hypotheses of Proposition~\ref{lambdaconvex1}, and if $\E$
  has a global minimizer $\varphi$, then inequality~\eqref{messi1} is,
  in fact, the \emph{entropy-entropy production
    inequality~\eqref{eq:13}} for linear $\Phi(s)=2\lambda s$
  (cf~\cite{VillaniCompMath2004,Villani2}). On the other hand,
  inequality~\eqref{messi1} can also be seen as the \emph{
    Kurdyka-\L{}ojasiewicz inequality~\eqref{eq:KL32}} with
  \begin{math}
   \theta(s)=\tfrac{\sqrt{2}}{\sqrt{\lambda}}\abs{s}^{-\frac{1}{2}}s
 \end{math}
 or as the \emph{L{}ojasiewicz-Simon
  inequality~\eqref{eq:Loj-Rd}} with exponent $\alpha=\frac{1}{2}$
(see Definition~\ref{def:LS-inequality}, and compare with
\cite{MR1644089,Si83,MR1609269,MR2019030,MR1986700,Bolte2010,MR3834660}).
\end{remark}

For $\lambda$-geodesically convex functionals, we have the following
existence result.

   \begin{theorem}[{\cite[Corollary 2.4.11]{AGS-ZH}}]\label{ExistsGC}
     Let $1<p<+\infty$ and for $\lambda \in \R$
     (respectively, for $\lambda = 0$ if
     $p\not= 2$), let $\E: \mathfrak{M} \rightarrow (-\infty, + \infty]$ be a proper,
     lower semicontinuous, $\lambda$-geodesically convex functional
     on a length space $(\mathfrak{M}, d)$. Further, suppose for every $c\in\R$, the sublevel set
     \begin{displaymath}
      E_{c}:=\Big\{ v \in  \mathfrak{M} \; \Big\vert  \; \E(v) \le c \Big\}
       \qquad\text{is compact in $\mathfrak{M}$,}
     \end{displaymath}
     Then, for every $v_0 \in D(\E)$, there is a $p$-gradient flow
     $v\in AC^{p}_{loc}([0,+\infty);\mathfrak{M})$ of $\E$ with
     respect to the strong upper gradient $\abs{D^{-}\E}$, such that $v(0+)=v_{0}$.
  \end{theorem}

%
%
%
%
%
%

\subsection{Some notions from the theory of dynamical systems}
\label{subsec:dynamicalsystem}

Here, we summarize some important notions and results
from the theory of dynamical systems in metric spaces \cite{Ha91},
which we adapt to our more general framework.

\begin{notation}
  For a curve $v : [0,+\infty)\to \mathfrak{M}$ and $t_{0}\ge 0$, we
   denote by
   \begin{displaymath}
       \mathcal{I}_{t_{0}}(v) =\big\{v(t)\,\big\vert\,t\ge t_{0}\big\}
   \end{displaymath}
   the \emph{image} in $\mathfrak{M}$ of the restriction
   $v_{\vert [t_{0},+\infty)}$ of $v$ on $[t_{0},+\infty)$; and by
   $\overline{\mathcal{I}}_{t_{0}}(v)$ the closure of
   $\mathcal{I}_{t_{0}}(v)$ in $\mathfrak{M}$.
\end{notation}

We start by recalling the classical notion of $\omega$-limit sets.

\begin{definition}[{\bfseries \emph{$\omega$-limit set}}]
 For a curve $v\in C((0,\infty);\mathfrak{M})$, the set
  \begin{displaymath}
  \omega(v):=\Big\{\varphi\in \mathfrak{M}\,\Big\vert\,\text{ there is
  $t_{n}\uparrow +\infty$
  s.t. }\lim_{n\to\infty}v(t_{n})=\varphi\text{ in $\mathfrak{M}$}\Big\}
 \end{displaymath}
is called the \emph{$\omega$-limit set} of $v$.
\end{definition}

The $\omega$-limit set $\omega(v)$ of continuous curves in
$\mathfrak{M}$ has the following properties.

\begin{proposition}
  \label{propo:omega-limit-set}
  For a curve $v\in C([0,\infty);\mathfrak{M})$ the following holds.
  \begin{enumerate}[topsep=3pt,itemsep=3pt,partopsep=1ex,parsep=1ex]
   \item If for some $t_{0}\ge 0$, the curve $v$ has a relative compact image
     \begin{math}
       \mathcal{I}_{t_{0}}(v)
     \end{math}
     in $\mathfrak{M}$, then the $\omega$-limit set
     $\omega(v)$ is non-empty.

   \item \label{propo:omega-limit-set-claim2} If there is a $\varphi \in
     \mathfrak{M}$ such that $\lim_{t\to\infty}v(t)=\varphi$ in $\mathfrak{M}$,

       then $\omega(v)=\{\varphi\}$.
     \item If for $t_{0}\ge 0$, $v$ has relative compact image
       $\mathcal{I}_{t_{0}}$ in $\mathfrak{M}$ and
       $\varphi \in \mathfrak{M}$, then $\omega(v)=\{\varphi\}$ if and
       only if $\lim_{t\to\infty}v(t)=\varphi$ in $\mathfrak{M}$.
        \end{enumerate}
\end{proposition}

We omit the proof of these statements since they are standard
(cf~\cite{chill-fasangova:isem}).

\medskip

With our next definition we generalize the classical notion of \emph{Lyapunov
functions} (cf~\cite{Ha91}).

\begin{definition}[{\bfseries \emph{Lyapunov
function}}]
  \label{def:Lyapunov}
  For a curve $v\in AC_{loc}(0,\infty;\mathfrak{M})$, a proper
  functional $\E : \mathfrak{M}\to (-\infty,\infty]$ is called a
  \emph{Lyapunov function} of $v$ if $\E$ is non-increasing along the
  trajectory $\{v(t)\,\vert\,t>0\}$. Moreover, a Lyapunov function
  $\E$ of $v$ is called a \emph{strict Lyapunov function} of $v$ if
  $\E\circ v\equiv C$ on $[t_{0},\infty)$ for some $t_{0}\ge 0$
  implies that $v$ is constant on $[t_{0},\infty)$.
\end{definition}

In addition, we need to introduce the notion of equilibrium points.

\begin{definition}[{\bfseries \emph{Point of equilibrium}}]
  \label{def:equilibrium-points}
  An element $\varphi\in \mathfrak{M}$ is called an \emph{equilibrium
    point} (or also \emph{critical point}) of a proper functional
  $\E : \mathfrak{M}\to (-\infty,\infty]$ with strong upper
  gradient $g$ if $\varphi\in D(g)$ and $g(\varphi)=0$. We denote by
  $\mathbb{E}_{g}=g^{-1}(\{0\})$ the set of all equilibrium points of
  $\E$ with respect to strong upper gradient $g$.
\end{definition}

\begin{remark}
  By Proposition~\ref{propo:charcter-local-min}, if $\E$ is a
  $\lambda$-geodesically convex energy functionals $\E$ with
  $\lambda\ge 0$, then the set of equilibrium points of $\E$ for the strong upper
    gradient $g=\abs{D^{-}\E}$ fulfills
  \begin{math}
    \mathbb{E}_{g}=\textrm{argmin}(\E).
  \end{math}
\end{remark}

Our next proposition describes the \emph{standard
Lyapunov entropy method} for $p$-gradient flows $v$ of a proper functional $\E$ on
a metric space $\mathfrak{M}$.

\begin{proposition}
  \label{propo:omega-limit-gradientflow}
  Let $\E : \mathfrak{M}\to (-\infty,+\infty]$ be a proper functional
  with strong upper gradient $g$, and $v : [0,+\infty)\to \mathfrak{M}$ a $p$-gradient flow of
  $\E$. Then, the following hold.
  \begin{enumerate}[topsep=3pt,itemsep=1pt,partopsep=1ex,parsep=1ex]
  \item\label{propo:omega-limit-gradientflow-claim-1} $\E$ is a
    strict Lyapunov function of $v$.

  \item\label{propo:omega-limit-gradientflow-claim-3} {\bfseries
      (Trend to equilibrium in the entropy sense)} If for
    $t_{0}\ge 0$, $\E$ restricted on the set
    $\overline{\mathcal{I}}_{t_{0}}(v)$ is lower semicontinuous, then
    for every $\varphi\in \omega(v)$, one has $\varphi\in D(\E)$ and
    \begin{equation}
      \label{eq:9}
      \lim_{t\to \infty}\E(v(t))=\E(\varphi)=\inf_{\xi\in \overline{\mathcal{I}}_{t_{0}}(v)}\E(\xi).
    \end{equation}
    In particular, if $\omega(v)$ is non-empty,
    then $\E_{\vert \overline{\mathcal{I}}_{t_{0}}(v)}$ is bounded from below.

  \item\label{propo:omega-limit-gradientflow-claim-5} {\bfseries
      ($\omega$-limit points are equilibrium points of $\E$)} Suppose
    for $t_{0}\ge 0$, $\E$ restricted on the set
    $\overline{\mathcal{I}}_{t_{0}}(v)$ is bounded from below and $g$
    restricted on the set $\overline{\mathcal{I}}_{t_{0}}(v)$ is lower
    semicontinuous. Then the $\omega$-limit set $\omega(v)$ of $v$ is
    contained in the set $\mathbb{E}_{g}$ of equilibrium points of
    $\E$.
\end{enumerate}
\end{proposition}

\begin{proof}
  By inequality~\eqref{eq:8}, the function $\E(v(\cdot))$ is
  non-increasing along $(0,+\infty)$. Thus, $\E$ is a Lyapunov
  function. Now, suppose that there are $t_{0}\ge 0$ and $C\in \R$ such that
  $\E(v(t))\equiv C$ for every $t\ge t_{0}$. By dissipation energy
  equality~ \eqref{EDENew}, the metric derivative $\abs{v'}(t)=0$ for
  all $t\ge t_{0}$ and so by Proposition~\ref{propo:1}, $v$ is
  constant on $[t_{0},+\infty)$,
  proving~\eqref{propo:omega-limit-gradientflow-claim-1}.

  Next, we assume there is a $t_{0}\ge 0$ such that $\E$ restricted on the set
  $\overline{\mathcal{I}}_{t_{0}}(v)$ is lower
  semicontinuous, and let $\varphi\in
  \omega(v)$. Then, there is a sequence $t_{n}\uparrow +\infty$ such
  that $v(t_{n})\to \varphi$ in $\mathfrak{M}$ and so, the lower
  semicontinuity of $\E$ yields $\varphi\in D(\E)$ and
  \begin{displaymath}
    \liminf_{t\to\infty}\E(v(t_{n})) \ge \E(\varphi).
  \end{displaymath}
  By monotonicity of $\E(v(\cdot))$ along $(0,+\infty)$ and since $t_{n}\uparrow +\infty$, for
  every $t>t_{0}$ there is an $t_{n}>t$ satisfying
  \begin{displaymath}
    \E(v(t))\ge \E(v(t_{n}))\ge \E(\varphi).
  \end{displaymath}
  Thus, $\E$ is bounded from below on
  $\overline{\mathcal{I}}_{t_{0}}(v)$ and by using again the
  monotonicity of $\E(v(\cdot))$, we see that limit~\eqref{eq:9} holds,
  establishing
  statement~\eqref{propo:omega-limit-gradientflow-claim-3}.

  To see that
  statement~\eqref{propo:omega-limit-gradientflow-claim-5} hold, we
  note first that since $\E(v(\cdot))$ is
  bounded from below on $(t_{0},+\infty)$ for some $t_{0}>0$ and
  since $v$ is a $p$-gradient flow
  of $\E$ with strong upper gradient $g$, we can infer from energy dissipation
  equality~\eqref{EDE} that the metric derivative
  $\abs{v'}\in L^{p}(t_{0},\infty)$. Now, let $\varphi\in
  \omega(v)$. Then there is sequence
  $(t_{n})_{n\ge 1}\subseteq (t_{0},+\infty)$ such that
  $t_{n}\uparrow +\infty$ and $v(t_{n})\to \varphi$ in $\mathfrak{M}$
  as $n\to \infty$. Since $v(t_{n})\in D(g)$ for every $n$ and $g$ is
  lower semicontinuous, we have that $\varphi\in D(g)$.
  Further, for every $s\in [0,1]$, H\"older's inequality gives
  \allowdisplaybreaks
  \begin{align*}
    \limsup_{n\to+\infty}d(v(t_{n}+s),\varphi)
    & \le \limsup_{n\to+\infty} \int_{t_{n}}^{t_{n}+s}\abs{v'}(r)\,\dr
      + \limsup_{n\to+\infty}d(v(t_{n}),\varphi)\\
    & \le \limsup_{n\to+\infty}
      \left(\int_{t_{n}}^{t_{n}+s}\abs{v'}^{p}(r)\,\dr\right)^{1/p}\\
    &\le \limsup_{n\to+\infty}
      \left(\int_{t_{n}}^{+\infty}\abs{v'}^{p}(r)\,\dr\right)^{1/p}=0,
  \end{align*}
  showing that the sequence $(v_{n})_{n\ge 1}$ of curves $v_{n} :
  [0,1]\to \mathfrak{M}$ given by
  \begin{displaymath}
    v_{n}(s)=v(t_{n}+s)\qquad\text{for every $s\in [0,1]$}
  \end{displaymath}
  satisfies
  \begin{equation}
    \label{eq:31}
    \lim_{n\to+\infty}\sup_{s\in [0,1]}d(v_{n}(s),\varphi)=0.
  \end{equation}
  Since $v$ is a $p$-gradient flow
  of $\E$, also $v_{n}$ is a $p$-gradient flow
  of $\E$ with strong upper gradient $g$. By assumption,
  $g$ is lower semicontinuous and since $v_{n}$ is continuous on
  $[0,1]$, $g(v_{n})$ is lower semicontinuous on $[0,1]$ and so, measurable. Thus,
  by Fatou's lemma applied to~\eqref{eq:31} and by~\eqref{eq:11},
  \begin{align*}
   0\le  g^{p^{\mbox{}_{\prime}}}(\varphi)&\le
                     \liminf_{n\to+\infty}\int_{0}^{1}g^{p^{\mbox{}_{\prime}}}(v_{n}(s))\,\ds\\ & \le
      \limsup_{n\to+\infty}\int_{t_{n}}^{t_{n}+1}\abs{v'}^{p}(r)\,\dr
     \le \limsup_{n\to+\infty}\int_{t_{n}}^{+\infty}\abs{v'}^{p}(r)\,\dr=0.
  \end{align*}
  Therefore, $\varphi\in \mathbb{E}_{g}$. This completes the proof of this
  proposition.
\end{proof}

\begin{remark}[\emph{The $\omega$-limit set $\omega(v)$ for
    $\lambda$-geodesically convex $\E$}]
  If for $\lambda\in \R$, the functional
  $\E : \mathfrak{M}\to (-\infty,+\infty]$ is proper, lower
  semicontinuous and $\lambda$-geodesically convex on a length space
  $\mathfrak{M}$, then by Proposition~\ref{propo:lconvex-slope}, the
  descending slope $\abs{D^{-}\E}$ of $\E$ is lower
  semicontinuous. Thus by
  Proposition~\ref{propo:omega-limit-gradientflow}, for every
  $p$-gradient flow $v : [0,+\infty)\to \mathfrak{M}$ of $\E$, the $\omega$-limit set
  $\omega(v)\subseteq \mathbb{E}_{\abs{D^{-}\E}}$.
\end{remark}

%
%
%
%
%
%

\section{Kurdyka-\L{}ojasiewicz-Simon inequalities in metric spaces}
\label{sec:KL-inequality}

Here, we adapt the classical Kurdyka-\L{}ojasiewicz inequality
(cf~\cite{MR1644089,Lo63,Bolte2010,MR3834660}) to a metric space
framework. Throughout this section, let $(\mathfrak{M},d)$ be a
complete metric space and $1<p<+\infty$.

\subsection{Preliminaries to K\L{} and \L{}S-inequalities}
\label{subsec:KL-inequality-prliminary}
We begin by fixing the following notation.

\begin{notation}\label{notation:1}
  For a proper function
  $\mathcal{F} : \mathfrak{M}\to (-\infty,+\infty]$, we write
  $[\abs{\mathcal{F}}>0]$ and $[\mathcal{F}>0]$ for the pre-image sets
  $\big\{v\in \mathfrak{M}\,\big\vert\,\abs{\mathcal{F}(v)}>0\big\}$
  and $\big\{v\in \mathfrak{M}\,\big\vert\,\mathcal{F}(v)>0\big\}$.
  For $R$, $r>0$, we write $[0<\mathcal{F}<R]$,
  $[0<\mathcal{F}\le R]$, $[\mathcal{F}=r]$, and
  $[0<\abs{\mathcal{F}}\le R]$ to denote the sets
  $\big\{v\in \mathfrak{M}\,\big\vert\,0<\mathcal{F}(v)< R\big\}$,
  $\big\{v\in \mathfrak{M}\,\big\vert\,0<\mathcal{F}(v)\le R\big\}$,
  $\big\{v\in \mathfrak{M}\,\big\vert\,\mathcal{F}(v)=r\big\}$,
  $\big\{v\in \mathfrak{M}\,\big\vert\,0<\mathcal{F}(v)\le R\big\}$,
  and
  $\big\{v\in
  \mathfrak{M}\,\big\vert\,0<\abs{\mathcal{F}(v)}<r\big\}$.
\end{notation}

The following inequality plays the key role of this paper.
Our definition extends the ones
in~\cite{Bolte2010,MR3834660} to the metric space framework.

\begin{definition}[{\bfseries \emph{K\L{}- inequality
      in a metric spaces setting}}]
  \label{def:KL-inequality}
  A proper functional $\E : \mathfrak{M}\to (-\infty,+\infty]$ with
  strong upper gradient $g$ and equilibrium point
  $\varphi\in \mathbb{E}_{g}$ is said to satisfy a
  \emph{Kurdyka-\L{}ojasiewicz inequality on the set
    $\U\subseteq [g>0]\cap [\theta'(\E(\cdot\vert \varphi))>0]$} if
  there is a strictly increasing function
  $\theta\in W^{1,1}_{loc}(\R)$ satisfying $\theta(0)=0$ and
  \begin{equation}
    \label{eq:16}
     \theta'(\E(v\vert \varphi))\, g(v)\ge 1\qquad\text{for all $v\in
       \U$}
   \end{equation}
   where $\E(v\vert \varphi))$ is the relative energy of $\E$
   with respect to $\varphi$ (see~\eqref{eq:32}).
\end{definition}

\begin{remark}
  (a) We recall from~\cite[Theorem~8.2]{Brezis}, that every $u\in W^{1,1}(I)$ for
  $I\subseteq \R$ bounded or unbounded open interval, admits a
  continuous representative. Therefore, we always assume that the
  function $\theta$ in~\eqref{eq:16} fulfills
  \begin{math}
    \theta\in C(\R).
  \end{math}
  (b) In Section~\ref{sec:sufKL} and Section~\ref{sec:chara-KL-talweg},
  one sees that an equilibrium point $\varphi\in \mathbb{E}_{g}$
  necessarily belongs to $\overline{\U}$ if $\E$ satisfies a
  Kurdyka-\L{}ojasiewicz inequality~\eqref{eq:16} on the set
  $\U\subseteq [g>0]\cap [\theta'(\E(\cdot\vert \varphi))>0]$.
\end{remark}

In the particular case that for an exponent $\alpha\in (0,1]$ and $c>0$,
\begin{equation}
    \label{eq:38bis}
    \theta(s):=\tfrac{c}{\alpha}\,\abs{s}^{\alpha-1}s\qquad\text{ for
      every $s\in \R$,}
\end{equation}
Kurdyka-\L{}ojasiewicz inequality~\eqref{eq:16} reduces to the
gradient inequality due to \L{}ojasiewicz~\cite{Lo63,Lo65} in $\R^{N}$
and Simon~\cite{Si83} in infinite dimensional Hilbert spaces.

\begin{definition}[{\bfseries \emph{K\L{}S- inequality
      in a metric spaces setting}}]
  \label{def:LS-inequality}
  A proper functional $\E : \mathfrak{M}\to (-\infty,+\infty]$ with
  strong upper gradient $g$ and equilibrium point
  $\varphi\in \mathbb{E}_{g}$ is said to satisfy a
  \emph{\L{}ojasiewicz-Simon inequality with exponent
    $\alpha\in (0,1]$ near $\varphi$} if there
  are $c>0$ and a set $\U\subseteq D(\E)$ with $\varphi\in \U$
  such that
  \begin{equation}
    \label{eq:38}
    \abs{\E(v\vert\varphi)}^{1-\alpha}\le c\,g(v)
    \quad \quad \hbox{for every $v\in \U$.}
  \end{equation}
\end{definition}

\subsection{$p$-Gradient flows of finite length}
\label{subsec:pflows-of-finite-length}
In this part, we show that the validity of a Kurdyka-\L{}ojasiewicz
inequality~\eqref{eq:16} on a set $\U$ yields finite length of
$p$-gradient flows in $\U$. Our next theorem generalizes the result
\cite[$(i)\Rightarrow(ii)$ of Theorem~18]{Bolte2010}) for proper, lower
semicontinuous and semi-convex functionals $\E$ on Hilbert
space.

\begin{theorem}[{\bfseries Finite length of $p$-gradient flows}]
  \label{thm:finite-length}
  Let $\E : \mathfrak{M}\to (-\infty,+\infty]$ be a proper functional
  with strong upper gradient $g$, and for
  $\varphi\in \mathbb{E}_{g}$, suppose $\E$ satisfies a
  Kurdyka-\L{}ojasiewicz inequality~\eqref{eq:16} on the set
  $\U\subseteq [g>0]\cap [\theta'(\E(\cdot\vert \varphi))>0]$. Then,
  the following statements hold.
  \begin{enumerate}[topsep=3pt,itemsep=2pt,partopsep=1ex,parsep=1ex]
  \item Let $v : [0,T)\to \mathfrak{M}$ of $\E$, $(0<T\le +\infty)$,
    be a piecewise $p$-gradient flow and suppose, there is a $0\le t_{0}< T$ such that
    \begin{equation}
      \label{eq:30}
     v(t)\in  \U\quad\text{for almost every $t\in [t_{0},T)$,}
    \end{equation}
    and $\E \circ v_{\vert [t_{0},T)}$ is bounded
    from below.  Then, the curve $v_{\vert [t_{0},T)}$ has finite length $\gamma(v_{\vert [t_{0},T)})$.
  \item There is a strictly increasing function
  $\theta\in W^{1,1}_{loc}(\R)$ satisfying $\theta(0)=0$ such that for
  every $p$-gradient flow $v : [0,T)\to \mathfrak{M}$ of $\E$
  satisfying~\eqref{eq:30} for some $0\le t_{0}<T\le +\infty$,
  one has that
   \begin{equation}
     \label{eq:26} d(v(s), v(t)) \le
     \int_{s}^{t}\abs{v'}(r)\,\dr\le
     \theta(\E(v(s)\vert\varphi))-\theta(\E(v(t)\vert\varphi))
   \end{equation}
   for every $t_{0}\le s<t<T$.
  \end{enumerate}
\end{theorem}

\begin{remark}[\emph{Trend to equilibrium in the metric sense}]
  The main task in proving trend to equilibrium in the metric sense of
  $p$-gradient flows $v$ is to show that the curve $v$, \emph{once}
  entered into the set $\U$ for which~\eqref{eq:16} holds, can not
  escape from $\U$ any more. In other words, $v$
  satisfies~\eqref{eq:30} for some minimal $t_{0}\ge 0$ such that
  $v(t_{0})\in \U$ and hence $v$ has finite length (see
  Section~\ref{secbehaviour}).
\end{remark}

\begin{proof}[Proof of Theorem~\ref{thm:finite-length}]
  Let $v : [0,T)\to \mathfrak{M}$ be a piecewise $p$-gradient flow of
  $\E$ satisfying~\eqref{eq:30} for some $0\le t_{0}\le T<+\infty$. Then,
  there is a partition
  $\mathcal{P}=\{0=\hat{t}_{0}, \hat{t}_{1}, \hat{t}_{2},\dots,\}$ of $[0,T)$ such
  that $v_{\vert (0,\hat{t}_{1}]}\in
  AC_{loc}((0,\hat{t}_{1}];\mathfrak{M})$ and
  $v_{\vert (0,\hat{t}_{1})}$ is a$p$-gradient flow of $\E$ on
  $(0,\hat{t}_{1})$. Moreover, for every $i\ge 2$,
  $v_{\vert [\hat{t}_{i-1},\hat{t}_{i}]}\in
  AC(\hat{t}_{i-1},\hat{t}_{i};\mathfrak{M})$ and
  $v_{\vert [\hat{t}_{i-1},\hat{t}_{i}]}$ is a $p$-gradient flow of
  $\E$ on $(\hat{t}_{i-1},\hat{t}_{i})$. Now, if $t_{0}=0$, then
  $v(0)\in D(\E)$ and so, Proposition~\ref{propo:chara-p-curves}
  implies that
  $v_{\vert [0,\hat{t}_{1}]}\in
  AC(0,\hat{t}_{1};\mathfrak{M})$. Therefore, and also in the case $t_{0}>0$,
  one can always construct a countable partition
  $\mathcal{P}_{t_{0}}:=\{t_{0}, t_{1}, t_{2},\dots\}$ of $[t_{0},T)$ from $\mathcal{P}$ such
  that for every $i\ge 1$, the curve $v_{\vert [t_{i-1},t_{i}]}\in  AC(t_{i-1},t_{i};\mathfrak{M})$ and
  $v_{\vert [t_{i-1},t_{i}]}$ is a $p$-gradient flow of
  $\E$ on $(t_{i-1},t_{i})$.

 Now, let $i\ge 1$. Recall, $\E\circ v$ is non-increasing 
 and locally absolutely continuous on $(t_{i-1},t_{i})$. By hypothesis,
 $\theta$ is strictly increasing and belongs to
 $W^{1,1}_{loc}(\R)$. Thus, the composed function
 \begin{equation}
   \label{eq:19}
   \mathcal{H}(\cdot):=\theta(\E(v(\cdot)\vert\varphi))
 \end{equation}
 is non-increasing and differentiable almost everywhere on
 $(t_{i-1},t_{n})$ satisfying the ``chain rule''
 \begin{equation}
    \label{eq:36}
    \fdt\mathcal{H}(t)=\theta'(\E(v(t)\vert\varphi))\,\fdt\E(v(t)\vert\varphi)
\end{equation}
for almost every $t\in (t_{i-1},t_{i})$
(cf~\cite[Corollary~3.50]{leoni}). Since $\mathcal{P}_{t_{0}}$ is a countable
partition of $[t_{0},T)$, $\mathcal{H}$ is differentiable almost
everywhere on $(t_{0},T)$ and~\eqref{eq:36} holds for almost every
$t\in (t_{0},T)$. By~\eqref{eq:30} and since
$\U\subseteq [g>0]\cap [\theta'(\E(\cdot\vert \varphi))>0]$, we can
apply the Kurdyka-\L{}ojasiewicz inequality~\eqref{eq:16} to $v=v(t)$ for
almost every $t\in (t_{0},T)$. Thus and by~\eqref{eq:11}, 
  \begin{equation}
    \label{eq:20}
      -\fdt\mathcal{H}(t)
      \ge \theta'(\E(u(t)\vert\varphi))\, g(v(t))\,\abs{v'}(t)
       \ge \,\abs{v'}(t)
  \end{equation}
  for a.e. $t\in (t_{0},T)$. Now, 
  integrating~\eqref{eq:20} over $(t_{i-1},t_{i})$ gives 
  \begin{displaymath}
    \int_{t_{i-1}}^{t_{i}}\abs{v'}(s)\,\ds
    \le \theta(\E(v(t_{i-1}+)\vert\varphi))-\theta(\E(v(t_{i}-)\vert\varphi)).
  \end{displaymath}
  Since $\E\circ v$ is non-increasing along
  $(t_{i-1},t_{i})$ and for every integer $k$, $l\ge 1$ with $k\neq l$, the two
   intervals $\E(v_{k}([0,t_{k}]))$ and $\E(v_{l}([0,t_{l}]))$ have at
   most one common point, it follows that
  \begin{math}
  \theta(\E(v(t_{i}-)\vert\varphi))\ge\theta(\E(v(t_{i}+)\vert\varphi)).
 \end{math}
 Thus, for every integer $j>1$,
  \begin{align*}
   \int_{t_{0}}^{t_j}\abs{v'}(s)\,\ds
    &\le \sum_{i=1}^{j}\Big(\theta(\E(v(t_{i-1}+)\vert\varphi))-\theta(\E(v(t_{i}-)\vert\varphi))\Big)\\
    &\le \theta(\E(v(t_{0}+) \vert\varphi))-\theta(\E(v(t_{j}-)\vert\varphi)).
  \end{align*}
  Since $\E \circ v_{\vert [t_{0},T)}$ is bounded from below and
  $\theta\in C(\R)$, sending $j\to \infty$ in the last estimate
  yields
  \begin{displaymath}
    \gamma(v_{\vert [t_{0},T)})=\int_{t_{0}}^{T}\abs{v'}(r)\,\dr\le
    \theta(\E(v(t_{0})\vert\varphi)) + C,
  \end{displaymath}
  for some constant $C\in \R$.
  Therefore, $v_{\vert [t_{0},T)}$ has finite length. 

  To complete the proof of this theorem, we need to show that
  inequality~\eqref{eq:26} holds for every $p$-gradient flow $v : [0,T)\to
  \mathfrak{M}$ of $\E$ satisfying~\eqref{eq:30} for some $t_{0}\in [0,T)$,
  ($T\le +\infty$). Recall that for every $p$-gradient flow $v$ of
  $\E$, one has that $v\in AC_{loc}((0,T);\mathfrak{M})$ and~\eqref{eq:1}
  holds for $m=\abs{v'}$ and every $0<s<t<T$. Thus, we can integrate~\eqref{eq:36} over
  $(s,t)$ for any $t_{0}\le s<t<T$ and apply~\eqref{eq:1}. From this, one
  sees that \eqref{eq:26} holds.
\end{proof}

%
%
%
%
%
%

\subsection{Talweg implies K\L{}-inequality}
\label{sec:sufKL}
We start this section by introducing the following assumption on the
functional $\E$ and the strong upper gradient $g$ of $\E$
(cf~\cite{Bolte2010} in the Hilbert space framework).

\begin{assumption}[{\bfseries \emph{Sard-type~condition}}]
  \label{ass:H1}
  Let $\E : \mathfrak{M}\to (-\infty,+\infty]$ be proper with strong upper gradient
  $g$ and equilibrium point $\varphi\in \mathbb{E}_{g}$. Suppose
  \begin{equation}
    \label{eq:15}
    \tag{H1}
     \text{$\U\subseteq \mathfrak{M}$ satisfies }\; \U\cap [\E(\cdot\vert\varphi)>0]\subseteq [g>0].
  \end{equation}
\end{assumption}

\begin{remark}
 The following statements are worth mentioning.
  \begin{enumerate}[topsep=3pt,itemsep=1ex,partopsep=1ex,parsep=1ex]
  \item\label{rem:strongform} If for $\varphi\in \mathbb{E}_{g}$ and
    $\U\subseteq [\theta'(\E(\cdot\vert \varphi))>0]$, the functional
    $\E$ satisfies the following \emph{stronger form} of
    Kurdyka-\L{}ojasiewicz inequality:
    \begin{equation}\label{eq:73}
      g(v)\ge 1/\theta'(\E(v\vert \varphi))\qquad\text{for all
        $v\in\U$,}
    \end{equation}
    where $\theta\in W^{1,1}_{loc}(\R)$ is a strictly increasing
    function satisfying $\theta(0)=0$, then for the set $\U$
    hypothesis~\eqref{eq:15} holds and $\E$ satisfies a
    Kurdyka-\L{}ojasiewicz inequality~\eqref{eq:16} on $\U$. Note,
    inequality~\eqref{eq:73} is, indeed, \emph{stronger}
    than~\eqref{eq:16}.  Since given a set
    $\U\subseteq [\theta'(\E(\cdot\vert \varphi))>0]$,
    then~\eqref{eq:16} can be deduced from~\eqref{eq:73}, while
    if~\eqref{eq:16} holds on $\U$, then one cannot conclude
    that~\eqref{eq:73} holds.

  \item Examples to~\eqref{rem:strongform}: if for
    $\varphi\in \mathbb{E}_{g}$, $\E$ satisfies a \L{}ojasiewicz-Simon
    inequality~\eqref{eq:38} with exponent $\alpha\in (0,1]$ on the set
    $\U\subseteq D(\E)$ then $\E$ satisfies~\eqref{eq:15} on $\U$.

  \item If for the functional $\E$ and $\varphi\in \mathbb{E}_{g}$, the
    set $\U\subseteq [\E(\cdot\vert\varphi)\ge 0]$ satisfies
    hypothesis~\eqref{eq:15}, then for every
    $\tilde{\varphi}\in \U\cap [g=0]$, one has
    $\E(\tilde{\varphi}\vert\varphi)=0$. In other words, \eqref{eq:15}
    is a \emph{Sard-type} condition.
  \end{enumerate}
\end{remark}

We come now to the definition of a \emph{talweg curve} in
$\mathfrak{M}$ introduced by Kurdyka~\cite{MR1644089} in $\R^{N}$ (see also
\cite{Bolte2010} for the Hilbert space framework).

\begin{definition}[{\bfseries \emph{Talweg curve in a metric spaces setting}}]\label{def:talweg}
  Let $\E : \mathfrak{M}\to (-\infty,+\infty]$ be a proper functional
  with strong upper gradient $g$. For $\varphi\in \mathbb{E}_{g}$
  and $R>0$, suppose the set $\U \subseteq [\E(\cdot\vert\varphi)\le R]$ satisfies
  hypothesis~\eqref{eq:15} and $\varphi\in \overline{\U}$. We set
  \begin{displaymath}
    s_{\U}(r)=\inf_{\hat{v}\in \U\cap
      [\E=r+\E(\varphi)]}g(\hat{v})\qquad\text{for every $r\in (0,R]$.}
  \end{displaymath}
 Then, for given $C>1$ and $\delta>0$, we call a piecewise
  continuous curve $x : (0,\delta]\to \mathfrak{M}$ a \emph{talweg curve}
  through the $C$-valley
  \begin{equation}
    \label{eq:45}
    \mathcal{V}_{C,\U}(\varphi):=\Big\{v\in \U\cap
      [\E(\cdot\vert\varphi)>0]
    \;\Big\vert\;  g(v)\le C\,s_{\U}(\E(v\vert \varphi))\Big\}
  \end{equation}
  if $x$ satisfies
  \begin{equation}\label{eq:25bis}
    x(t)\in \mathcal{V}_{C,\U}(\varphi)\quad \text{for every $t\in
      (0,\delta]$}\quad\text{and}\quad\lim_{t\to 0+}\E(x(t)\vert\varphi)=0.
  \end{equation}
\end{definition}

\begin{remark}[\emph{Interpretation of Definition~\ref{def:talweg}}]
  Since $s_{\U}(r)$ denotes the minimal slope of $\E$ along the level
  curve $\U\cap [\E=r+\E(\varphi)]$, a curve $x$ satisfying
  condition~\eqref{eq:25bis} describes a path of minimal slope
  (up to a multiplicative constant $C>1$) through a neighborhood $\U$
  of an equilibrium point $\varphi$ of $\E$.
\end{remark}

Further, we need that a talweg has the following \emph{regularity}.

\begin{definition}
   For $\delta>0$, a curve $x : (0,\delta]\to \mathfrak{M}$ is called
  \emph{piecewise $AC$} if there is a countable partition
  $\{I_{n}\}_{n\ge 1}$ of $(0,\delta]$ into nontrivial intervals
  $I_{n}=(a_{n},b_{n}]\subseteq (0,\delta]$ such that
  $x\in AC(I_{n},\mathfrak{M})$ for all $n\ge 1$.
\end{definition}

\begin{remark}[\emph{Existence of a talweg curve}]
  If for $\varphi\in \mathbb{E}_{g}$ and $R>0$, a functional $\E$
  satisfies the stronger inequality~\eqref{eq:73} on a set
  \begin{displaymath}
  \U\subseteq [\E(\cdot\vert\varphi)\le R]\cap
  [\theta'(\E(\cdot\vert\varphi))>0],
\end{displaymath}
then one has that
  \begin{displaymath}
    s_{\U}(r)\ge 1/\theta'(r)\qquad\text{for
    every $r\in (0,R]\cap [0<\theta'(r)<\infty]$.}
  \end{displaymath}
  We emphasize that the existence of a talweg curve depends on the
  controllability of $s_{\U}(r)$ from below with respect to $r$.
\end{remark}

Our first theorem of this subsection provides sufficient conditions on
the talweg curve implying that the functional $\E$ satisfies a
Kurdyka-\L{}ojasiewicz inequality~\eqref{eq:16} with \emph{optimal}
growth function $\theta$. The optimality on the growth function
$\theta$ is shown in Example~\eqref{ex:lojforE} in
Section~\ref{sec:Hilbertspacesframework} by investigating the
\emph{smooth case} (see, for instance~\cite{HaJenBook} or the examples
in~\cite{MR1986700}), for which it is known that $\E$ satisfies a
\L{}ojasiewicz-Simon inequality with exponent $\alpha=1/2$.

\begin{theorem}[{\bfseries Talweg implies K\L{}-inequality}]
  \label{thm:3talweg-kl}
  Let $\E : \mathfrak{M}\to (-\infty,+\infty]$ be a proper functional
  with strong upper gradient $g$, and for
  $\varphi\in \mathbb{E}_{g}$ and $R>0$, suppose that the set
  $\U\subseteq [0\le \E(\cdot\vert\varphi)\le R]$ satisfies
  hypothesis~\eqref{eq:15}. If there are $C>1$, $\delta>0$, and a
  piecewise $AC$ talweg $x : (0,\delta]\to \mathfrak{M}$ through the
  $C$-valley $\mathcal{V}_{C,\U}(\varphi)$ such~that
  \begin{equation}
    \label{eq:46}
    h(t):=\E(x(t)\vert\varphi)\quad\text{for every $t\in (0,\delta]$,}\quad
    h(0):=\lim_{t\to0+}\E(x(t)\vert\varphi),
  \end{equation}
  is a strictly increasing function $h : [0,\delta]\to [0,R]$
  satisfying $h\in W^{1,1}_{loc}(0,\delta)$, $h(0)=0$ and $\abs{\{h'=0\}}=0$,
  then there is a subset $\tilde{\U}\subseteq \U\cap
  [0<\E(\cdot\vert\varphi)\le R]$ such that for every $v\in \tilde{\U}$,
  one has
  \begin{equation}
    \label{eq:47}
    (h^{-1})'(\E(v\vert\varphi))g(v)\ge
    \frac{1}{C\,\abs{x'}(t)}\qquad\text{with $t=h^{-1}(\E(v\vert\varphi))$.}
  \end{equation}
  In addition, if $\abs{x'} \in L^{\infty}(0,\delta)$, then
  $\E$ satisfies a Kurdyka-\L{}ojasiewicz inequality~\eqref{eq:16} on
  $\tilde{\U}$.
\end{theorem}

\begin{remark}
  We note that in Theorem~\ref{thm:3talweg-kl}, the hypotheses imply
  that $h$ is a homeomorphism. In the smooth case (see, for instance,
  Example~\ref{ex:lojforE} in
  Section~\ref{sec:Hilbertspacesframework}), one can show that
  $h : (0,\delta)\to (0,R)$ is, in fact, at least a $C^{1}$-diffeomorphism and the talweg
  $x\in C^{1}((0,\delta);\mathfrak{M})\cap
  C([0,\delta];\mathfrak{M})$.
\end{remark}

\begin{proof}[Proof of Theorem~\ref{thm:3talweg-kl}]
  We begin by defining the strictly increasing homeomorphism $\theta :
  \R\to \R$ by setting
  \begin{displaymath}
    \theta(r)=
    \begin{cases}
      h^{-1}(R)(r-R)+h^{-1}(R) & \text{if $r>R$,}\\
      h^{-1}(r) & \text{if $r\in [0,R]$,}\\
      -h^{-1}(-r) & \text{if $r\in [-R,0]$,}\\
      -h^{-1}(R)(r+R)-h^{-1}(R) & \text{if $r<-R$.}
    \end{cases}
  \end{displaymath}
  By hypothesis, $\abs{\{h'=0\}}=0$. Thus, $\theta\in
  W^{1,1}_{loc}(\R)$ and
  \begin{displaymath}
    t= (\theta\circ h)(t)=\theta(\E(x(t)\vert \varphi))\qquad\text{for
     every $t\in [0,\delta]$.}
  \end{displaymath}
  Since $h$ is strictly
  increasing and since $\theta \in W^{1,1}_{loc}(\R)$, the chain
  rule (cf~\cite[Corollary~3.50]{leoni}) implies that
 \begin{displaymath}
    \fdt\theta(\E(x(t)\vert\varphi))=\theta'(\E(x(t)\vert\varphi))\,\fdt\E(x(t)\vert\varphi)
  \end{displaymath}
  for almost every $t\in (0,\delta)$. Since
  $\theta'(\E(x(\cdot )\vert \varphi))$ is positive on $(0,\delta)$,
  and since $g$ is a strong upper gradient of $\E$,
  inequality~\eqref{eq:4} yields that
  \begin{displaymath}
    1=\fdt(\theta\circ h)
    (t)=\theta'(\E(x(t)\vert\varphi))\fdt\E(x(t))
    \le \theta'(\E(x(t)\vert \varphi)) g(x(t))\;\abs{x'}(t)
 \end{displaymath}
 for almost every $t\in (0,\delta)$. Since
 $x(t)\in\mathcal{V}_{C,\U}(\varphi)$ for every $t\in
 (0,\delta]$ and by~hypothesis~\eqref{eq:15}, $g(x(t))>0$ for all
 $t\in (0,\delta]$. Thus, the previous inequality yields that
  \begin{equation}
    \label{eq:48}
    \theta'(\E(x(t)\vert \varphi)) g(x(t))\ge
    \frac{1}{\abs{x'}(t)}>0\qquad
    \text{for almost every $t\in (0,\delta)$.}
  \end{equation}

  Next, we set
  \begin{displaymath}
    \mathcal{P}=\Big\{t\in
  [0,\delta]\,\Big\vert\,0<h'(t)<+\infty,\;0<\abs{x'}(t)<+\infty\Big\}.
\end{displaymath}
Since $\abs{\{h'=0\}}=\abs{\{\abs{x'}=0\}}=0$, the set
$\mathcal{Z}:=[0,\delta]\setminus \mathcal{P}$ satisfies
$\abs{\mathcal{Z}}=0$ and since $h\in W^{1,1}_{loc}(0,\delta)$, the
set $\mathcal{L}:=h(\mathcal{Z})$ satisfies
$\abs{\mathcal{L}}=0$. Further, one has that the interval $[0,R]$
fulfills $[0,R]=h(\mathcal{P})\cup \mathcal{L}$ and $0<(h^{-1})'<\infty$ on
$h(\mathcal{P})$.

  Now, let
  $\tilde{U}:=[0<\E(\cdot\vert\varphi),\,\E(\cdot\vert\varphi)\in h(\mathcal{P})]\cap \U$
  and take $v\in
  \tilde{\U}$. Then $r:=\E(v\vert \varphi)\in (0,R]$ and $r\in
  h(\mathcal{P})$. Thus, there is a unique $t\in \mathcal{P}$ such that $h(t)=r$,
  $0<\abs{x'}(t)<+\infty$ exits and
  \begin{equation}
    \label{eq:49}
    \E(x(t)\vert\varphi)=h(t)=r=\E(v\vert\varphi).
 \end{equation}
 Since $x(t)\in\mathcal{V}_{C,\U}(\varphi)$ and by~\eqref{eq:49}, we
 see that
  \begin{displaymath}
    g(x(t))\le C \,  s_{\U}(\E(x(t)\vert\varphi)) = C\,\inf_{\hat{v}\in \U\cap
      [\E=(\E(x(t)\vert\varphi))+\E(\varphi)]} g(\hat{v}) \le C\,g(v).
  \end{displaymath}
  Since $t\in \mathcal{P}$, $x(t)$ also satisfies~\eqref{eq:48}. Thus,
  and by using again~\eqref{eq:49},
  \begin{displaymath}
    \frac{1}{\abs{x'}(t)}\le \theta'(\E(x(t)\vert\varphi))\,
    g(x(t))\le \theta'(\E(x(t)\vert\varphi)) C\,g(v)=\theta'(\E(v \vert\varphi))\,C\,g(v),
  \end{displaymath}
  which is inequality~\eqref{eq:47}. If, in addition,
  $\abs{x'}\in L^{\infty}(0,\delta)$, then
  $\E$ satisfies a Kurdyka-\L{}ojasiewicz inequality~\eqref{eq:16} on
  the set $\tilde{U}$ for the strictly increasing function
  $\tilde{\theta}:=
  \norm{\abs{x'}}_{L^{\infty}(0,\delta)}\,C\,\theta$. This completes
  the proof of this theorem.
\end{proof}

In our next result we replace the condition $\abs{x'}\in
L^{\infty}(0,T)$ by $\abs{x'}\in L^{1}(0,T)$.

\begin{theorem}[{\bfseries Talweg implies
    K\L{}-inequality, Version 2}]
  \label{thm:2-Talweg-KL}
  Let $\E : \mathfrak{M}\to (-\infty,+\infty]$ be a proper functional
  with strong upper gradient $g$, and for $\varphi\in \mathbb{E}_{g}$
  and $R>0$, suppose the set
  $\U\subseteq [0\le \E(\cdot\vert\varphi)\le R]$ satisfies
  hypothesis~\eqref{eq:15}. Further, suppose the set $\mathcal{D}\subseteq \U$ is
  nonempty and the function
  \begin{equation}
    \label{eq:29}
    r\mapsto s_{\mathcal{D}}(r):=\inf_{\hat{v}\in \mathcal{D}\cap [\E=r+\E(\varphi)]} g(\hat{v})
  \quad\text{is lower semicontinuous on $(0,R]$.}
  \end{equation}
  If for $C>1$, there is a piecewise $AC$ talweg
  $x : (0,R]\to \mathcal{V}_{C,\mathcal{D}}(\varphi)$ satisfying
  \begin{equation}
    \label{eq:85}
    \E(x(r)\vert \varphi)=r\qquad\text{for every $r\in (0,R]$,}
  \end{equation}
  and $\abs{x'}\in L^{1}(0,R)$, 
  then there is a set
  $\tilde{\U}\subseteq \U\cap [0< \E(\cdot\vert\varphi)\le R]$ such that $\E$
  satisfies a Kurdyka-\L{}ojasiewicz inequality~\eqref{eq:16} on $\tilde{\U}$.
\end{theorem}


Our proof adapts an idea from~\cite{Bolte2010} to the
metric space framework.

\begin{proof}[Proof of Theorem~\ref{thm:2-Talweg-KL}]
  Let $x : (0,R]\to \mathfrak{M}$ be a piecewise $AC$ talweg through
  the $C$-valley $\mathcal{V}_{C,\mathcal{D}}$ around $\varphi$
  satisfying~\eqref{eq:85} with $\abs{x'}\in L^{1}(0,R)$. If
  $\mathcal{D}\subseteq \U$ is a nonempty set such that~\eqref{eq:29}
  holds, then by hypothesis~\eqref{eq:15}, the function
  \begin{equation}
    \label{eq:87}
    u(r):=\frac{1}{\inf_{\hat{v}\in
        \mathcal{D}\cap [\E=r+\E(\varphi)]} g(\hat{v})}\qquad\text{for every
      $r\in (0,R]$}
  \end{equation}
  is measurable, strictly positive and upper semicontinuous on
  $(0,R]$. By~\eqref{eq:85}, \eqref{eq:4}, and since
  $x(r)\in \mathcal{V}_{C,\mathcal{D}}(\varphi)$ for every
  $r\in (0,\delta]$, we see that
 \begin{displaymath}
  1=\fdr \E(x(r)\vert\varphi)\le g(x(r))\;\abs{x'}(r)
    \le C\, \inf_{\hat{v}\in \U\cap [\E=r+\E(\varphi)] } g(y) \;\abs{x'}(r)
 \end{displaymath}
 for almost every $r\in (0,R)$. Since $\abs{x'}\in L^{1}(0,R)$, the
 function $u\in L^{1}(0,R)$. Now, due to~\cite[Lemma~45]{Bolte2010},
 for the measurable function $u$ given by~\eqref{eq:87}, there is a
 continuous function $\tilde{u} : (0,R]\to (0,+\infty)$ satisfying
 $\tilde{u}\ge u$ and $\tilde{u}\in L^{1}(0,R)$. Now, let
 $\theta : \R\to \R$ be defined by
\begin{equation}
  \label{eq:88}
  \theta(s)=
  \begin{cases}
   \displaystyle   \int_{0}^{R}\tilde{u}(r)\,\dr\,(s-R+1)& \text{if $s>R$,}\\[7pt]
   \displaystyle   \int_{0}^{s}\tilde{u}(r)\,\dr & \text{if $s\in [0,R]$,}\\[7pt]
   \displaystyle    -\int_{0}^{-s}\tilde{u}(r)\,\dr  & \text{if $s\in [-R,0)$,}\\[7pt]
   \displaystyle  -\int_{0}^{R}\tilde{u}(r)\,\dr\,(s+R+1) & \text{if $s<-R$}
    \end{cases}
\end{equation}
for every $s\in \R$. Then,
$\theta\in C^{1}(\R\setminus\{0\})\cap W^{1,1}_{loc}(\R)$ is strictly increasing, satisfies $\theta(0)=0$. In
particular, for every $v\in \U\cap [0<\E(\cdot\vert\varphi)]$, setting
$r=\E(v\vert\varphi)$ yields
 \begin{displaymath}
  \theta'(\E(v\vert\varphi))\;g(v)
  = u(r)\; g(v)
 \ge 1,
\end{displaymath}
proving that $\E$ satisfies a Kurdyka-\L{}ojasiewicz
inequality~\eqref{eq:16} near $\varphi$.
\end{proof}



%
%
%
%
%
%
\subsection{Existence of a talweg curve}
\label{sec:chara-KL-talweg}

The aim of this section is to establish existence of a talweg curve
through a $C$-valley $\mathcal{V}_{C,\U}(\varphi)$ near an equilibrium
point $\varphi\in \mathbb{E}_{g}$ of a given functional $\E$.\medskip 

We briefly summarize our approach of proving the existence of a talweg
curve and outline the assumption need for this. For the construction
of a talweg curve, we employ $p$-gradient curves of $\E$. We begin by
choosing a neighborhood $\U_{\varepsilon,R}$ of
$\varphi\in \mathbb{E}_{g}$ by
\begin{equation}
      \label{eq:65}
         \U_{\varepsilon,R}=B(\varphi,\varepsilon)\cap[0\le \E(\cdot\vert\varphi)<R]
   \end{equation}
   for given $\varepsilon$, $R>0$ and for $0<\varepsilon_{0}<\varepsilon$ and $0<R_{0}< R$, set
 \begin{equation}
    \label{eq:89}
    \mathcal{D}=\Bigg\{v\in \U_{\varepsilon_{0},R_{0}}\,\Bigg\vert\!
    \begin{array}{l}
      \textrm{there are $v^{0}\in \overline{\U}_{\frac{\varepsilon_{0}}{3},R_{0}}$, a
      $p$-gradient flow $\hat{v}$ of $\E$}\\
      \textrm{with $\hat{v}(0+)=v^{0}$ and $t_{0}\ge 0$ s.t. }\hat{v}(t_{0})=v
    \end{array}\!\!\!
    \Bigg\}.
  \end{equation}
  Then, roughly spoken, the set $\mathcal{D}$ is a region in
  $\U_{\varepsilon,R}$ of paths (the image) of a $p$-gradient flow
  (curve) $\hat{v}$ of $\E$ with initial value
  $v^{0}\in \overline{\U}_{\frac{\varepsilon_{0}}{3},R_{0}}$. Now,
  \emph{Step~1.} is to show that the following \emph{stability result}
  holds:
  \begin{equation}
    \label{eq:67}
    \begin{cases}
      \text{for every $v\in \mathcal{D}$, $p$-gradient flow
        $\hat{v}$ of $\E$ with $\hat{v}(0+)=v$} &\\
      \text{and every $t\ge 0$ satisfying $\E(\hat{v}(t)\vert\varphi)>0$, one has
        $\hat{v}(t)\in
        \mathcal{D}$.}&
    \end{cases}
  \end{equation}
  This property is crucial in our existence proof of a talweg curve,
  since it allows us to show that for every $C>1$, the $C$-valley
  $\mathcal{V}_{C,\mathcal{D}}(\varphi)$ is non-empty. For
  proving~\eqref{eq:67}, we employ that $p$-gradient flows in
  $\U_{\varepsilon_{0},R_{0}}$ are of length smaller than $\varepsilon_{0}/3$
  (see~\eqref{eq:29N} below), that for every initial value $v_{0}\in
  D(\E)$, there is a $p$-gradient flow, and continuous dependence of
  $p$-gradient curves on the initial data. 

Note, the last two properties together with the right-continuity of the strong
upper gradient $g$ along $p$-gradient curves are naturally given for
gradient flows in Hilbert spaces (cf~\cite{BrezisBook}
and~\cite{Bolte2010}) or, more generally, for gradient flows of
evolution variational inequalities (see Remark~\ref{rem:evi-conditions}
below). But for the general case, we need to make the following assumptions.


\begin{assumption}
  \label{ass:H2}
  Suppose, for the proper functional $\E : \mathfrak{M}\to (-\infty,+\infty]$
  with strong upper gradient $g$ the following hold.
  \begin{enumerate}[topsep=3pt,itemsep=-3pt,partopsep=1ex,parsep=1ex]
   \item[(i)]\label{H2:i} (\emph{Existence}) For every $v_{0}\in D(\E)$, there is a unique $p$-gradient flow $v$
    of $\E$ with initial value $v(0+)=v_{0}$;

   \item[(ii)]\label{H2:ii} (\emph{Continuous dependence}) The mapping $S : [0,+\infty)\times D(\E)\to D(\E)$
     defined by
    \begin{equation}
      \label{eq:66}
      S_{t}v_{0}=v(t)\qquad \text{for all $t\ge 0$, $v_{0}\in D(\E)$,}
    \end{equation}
     is continuous, where $v$ is the $p$-gradient flow of $\E$
     with $v(0+)=v_{0}$;

   \item[(iii)]\label{H2:iii} (\emph{Right-continuity of $g$}) For every $v_{0}\in D(\E)$ and $p$-gradient flow $v$
    of $\E$ with initial value $v(0+)=v_{0}$, the map $t\mapsto
    g(v(t))$ is right-continuous on $[0,+\infty)$.
  \end{enumerate}
\end{assumption}

\begin{remark}[\emph{Gradient flows of evolution variational
    inequalities}]
  \label{rem:evi-conditions}
  We emphasize that if $\E : \mathfrak{M}\to (-\infty,+\infty]$ is
  $\lambda$-geodesically convex, $(\lambda\in \R)$, and satisfies
  coercivity condition~\eqref{eq:77}, then the set of all gradient
  flows generated by $\E$ with strong upper gradient $g=\abs{D^{-}\E}$
  fulfills all three conditions in Assumption~\ref{ass:H2}
  (cf~\cite[Theorem~2.6]{DaneriSavare},
  \cite[Proposition~3.1]{DaneriSavare2}).
\end{remark}

For showing that the $C$-valleys
$\mathcal{V}_{C,\mathcal{D}}(\varphi)$ are non-empty, another crucial
property is needed, namely, compactness of the sublevel sets
of $\E$ (locally) near an equilibrium point $\varphi\in
\mathbb{E}_{g}$. Note, compactness of sublevel sets (in the global
sense) is a natural assumption for the existence of gradient flows
(cf~Theorem~\ref{ExistsGC} or~\cite{MR3465809}).

\begin{assumption}[\emph{Local coercivity of $\E$}]
  \label{ass:H3}
  Let $\E : \mathfrak{M}\to (-\infty,+\infty]$ be a proper functional
  with strong upper gradient $g$ and equilibrium point
  $\varphi\in \mathbb{E}_{g}$. Suppose, there are $\varepsilon>0$ and
  $R>0$ such that the set $\U_{\varepsilon,R}$ defined by~\eqref{eq:65}
   is relatively compact in $\mathfrak{M}$.
\end{assumption}

\begin{remark}
  Under the Assumption~\ref{ass:H3}, every closed subset $\mathcal{D}$
  of $\U_{\varepsilon,R}$ is compact in $\mathfrak{M}$. Thus, if, in
  addition, $\E$ is lower semicontinuous and $\lambda$-geodesically
  convex, $(\lambda\in \R)$, then for every closed and non-empty
  subset $\mathcal{D}$ of $\U_{\varepsilon,R}$, the function
  $s_{\mathcal{D}}$ defined by~\eqref{eq:29bis} is lower
  semicontinuous on $(0,+\infty)$ (see
  Proposition~\ref{propo:inflsc}).
\end{remark}

\emph{Step~2.} is to show that the set $\mathcal{D}$ is closed in
$[0<\E(\cdot\vert\varphi)\le R_{0}]$ and thereby compact. For proving
this, the Assumptions~\ref{ass:H1}--\ref{ass:H3} are employed together
with lower semicontinuity of $\E$ and \emph{energy dissipation
  equality}~\eqref{EDENew}.\medskip

With these preliminaries, we are now in the position (\emph{Step~3.})
to show that there is a $p$-gradient flows $v$ running
through a $C$-valley $\mathcal{V}_{C,\U_{\varepsilon,R}}(\varphi)$
(cf~\eqref{eq:45}). For this, we use the function $s_{\mathcal{D}}(r)$,
$(r>0)$, defined by~\eqref{eq:29bis}. It provides a control on the
steepest descent of the talweg curve since $s_{\mathcal{D}}(r)$ is the largest lower
bound of the strong upper gradient $g$ of $\E$ along the level curve
$[\E=r+\E(\varphi)]$. By assuming that $s_{\mathcal{D}}$ is lower
semicontinuous and that $g$ is right-continuous along $p$-gradient
flows (see~\eqref{H2:iii} in Assumption~\ref{ass:H2}), one finds
constants $C$, $0<R_{1}\le R$ and a $p$-gradient flows $v$ running
through the $C$-valley
$\mathcal{V}_{C,\U_{\varepsilon,R}}(\varphi)$.\medskip


We can now state our existence
result of talweg curves in metric spaces (cf~\cite[Theorem~18,
  $(ii)\Rightarrow (iv)$]{Bolte2010} in the Hilbert space setting).

\begin{theorem}[{\bfseries Existence of a talweg curve}]
  \label{thm:ex-talweg}
  Let $\E : \mathfrak{M}\to (-\infty,+\infty]$ be a proper, lower
  semicontinuous functional, $g : \mathfrak{M}\to [0,+\infty]$ a
  proper, lower semicontinuous strong upper gradient of $\E$, and
  $\varphi\in \mathbb{E}_{g}$. Suppose, for $\E$ and $g$
  Assumption~\ref{ass:H2} holds, and there are $\varepsilon$, $R>0$
  such that for $\U_{\varepsilon,R}$ given by~\eqref{eq:65}, the
  Assumptions~\ref{ass:H1} and~\ref{ass:H3} hold. Further, there are
  $0<\varepsilon_{0}<\varepsilon$ and $0<R_{0}< R$ such that $\mathcal{D}$ given
  by~\eqref{eq:89}
  is non-empty, $s_{\mathcal{D}}$ defined by~\eqref{eq:29bis} is lower semicontinuous,
  and for every $p$-gradient flow $v$ of $\E$ with initial value
  $v(0+)=v_{0}\in \overline{\U}_{\frac{\varepsilon_{0}}{3},R_{0}}$, one has that
  \begin{equation}
    \label{eq:29N}
    v([0,t))\subseteq \U_{\varepsilon_{0},R_{0}}\text{ for $t>0$,
      implies}\qquad
   \int_{0}^{t}\abs{v'}(s)\,\ds\le  \frac{\varepsilon_{0}}{3}.
  \end{equation}
  Then, there are $C>1$, $0<R_{1}\le R_{0}$, and a piecewise $AC$
  talweg $x : (0,R_{1}] \to \mathfrak{M}$ of finite length $\gamma(x)$
  through the $C$-valley $\mathcal{V}_{C,\mathcal{D}}(\varphi)$ of
  $\E$ satisfying~\eqref{eq:85} with $R$ replaced by $R_{1}$.
\end{theorem}

To conclude the sketch of the proof of Theorem~\ref{thm:ex-talweg}, we
outline its final \emph{Step~4.}\;: The construction of a
talweg curve $x$ of finite length. By the previous step, one has that
there are $C>1$, $0<R_{1}\le R_{0}$, $T>0$, and a $p$-gradient flow
$v$ of $\E$ such that
$v(t)\in\mathcal{V}_{C,\U_{\varepsilon,R}}(\varphi)$ for all
$t\in [0,T)$ and $\E(v(\cdot)\vert\varphi) : [0,T)\to [0,R_{1}]$ is a
strictly decreasing homeomorphism. If, in addition, 
$\lim_{t\to T-}\E(v(t)\vert\varphi)=0$ (\emph{Case 1.}), then
$x(r):=v(\hat{h}^{-1}(r))$, $(r\in [0,R_{1}])$ defines a locally
absolutely continuous talweg curve $x : [0,R_{1}]\to \mathfrak{M}$ of
finite length $\gamma(x)$. Here, $\hat{h}^{-1}$ denotes the inverse
function of $\hat{h} : [0,T]\to [0,R_{1}]$ given by
$\hat{h}(t):= \E(v(T-t)\vert\varphi)$, $(t\in [0,T])$.

Note, the fact that the talweg curve $x$ is locally
absolutely continuous on $[0,R_{1}]$ follows from the fact that $x=
v\circ \hat{h}^{-1}$ where $v$ is locally absolutely continuous and
$\hat{h}^{-1}$ satisfies the \emph{Lusin (N) property}. 

\begin{definition}[{\cite[p.77]{leoni}}]
  A function $u : I\to \R$ defined on
  an interval $I\subseteq \R$ satisfies the \emph{Lusin (N) property}
  if $u$ maps sets of Lebesgue measure zero into sets of Lebesgue
  measure zero.
\end{definition}

The next remark outlines why $\hat{h}^{-1}$ admits the Lusin (N) property.

\begin{remark}\label{rem:LusinNproperty}
  Each function $u\in W^{1,1}(I)$ defined on an interval
  $I\subseteq \R$ admits a continuous representative
  (see~\cite[Theorem~8.2]{Brezis}) and hence, $u$ admits the Lusin (N)
  property (see~\cite[Theorem~3.12]{leoni}). Further, if $u : [a,b]\to \R$
  is continuous and strictly increasing, then the inverse $u^{-1} :
  [u(a),u(b)]\to \R$ is absolutely continuous if $u$ satisfies the \emph{Lusin (N) property}.
\end{remark}

In \emph{Case 2.}, $\lim_{t\to T-}\E(v(t)\vert\varphi)=\E(v(T)\vert\varphi)>0$. Then, by
\emph{Step~1.}, $v(T)\in \mathcal{D}$ and so, there are $R'_{1}$,
$T'>0$, and a $p$-gradient flow $v' : [0,T']\to [0,R']$ with initial
value $v'(0+)=v(T)$ running through the $C$-valley
$\mathcal{V}_{C,\mathcal{D}}(\varphi)$ of $\E$. Iterating these
arguments, Zorn's lemma finally leads to the existence of a piecewise $p$-gradient
curve $\tilde{v}$ from which one can construct, as in (\emph{Case 1.}, a
talweg curve $x : (0,R_{1}]\to
\mathcal{V}_{C,\U_{\varepsilon,R}}(\varphi)$ of finite length.\medskip

Now, we give the details of the proof.

\begin{proof}[Proof of Theorem~\ref{thm:ex-talweg}]
  \emph{\underline{Step~1.\;}} We begin by showing that under the
  assumptions~\ref{ass:H1}--\ref{ass:H3} and by~\eqref{eq:89},
  the $p$-gradient curves of $\E$ satisfy \emph{stability
    result}~\eqref{eq:67}.

  Let $v\in \mathcal{D}$, $\hat{v}$ be a $p$-gradient
  flow of $\E$ with initial value $\hat{v}(0+)=v$, and let $t>0$ such
  that $\E(\hat{v}(t)\vert\varphi)>0$. Then, we need to show that
  $\hat{v}(t)\in \mathcal{D}$. For this, it is sufficient to prove
  that $\hat{v}(t)\in \U_{\varepsilon_{0},R_{0}}$ and there is a
  $p$-gradient flow $\hat{w}$ of $\E$ with
  $\hat{w}(0+)=w_{0}\in
  \overline{\U}_{\frac{\varepsilon_{0}}{3},R_{0}}$, and a
  $\hat{t}\ge 0$ such that $\hat{w}(\hat{t})=\hat{v}(t)$. Since
  $v\in \mathcal{D}$, there are
  $w_{0}\in \overline{\U}_{\frac{\varepsilon_{0}}{3},R_{0}}$ and a
  $p$-gradient flow $\hat{w}$ of $\E$ satisfying $\hat{w}(0+)=w_{0}$,
  and there is a $t_{0}\ge 0$ such that $\hat{w}(t_{0})=v$. By
  $\E(\hat{v}(t)\vert\varphi)>0$ and $\hat{v}(0+)=v=\hat{w}(t_{0})$,
  and since $\E(w_{0}\vert\varphi)\le R_{0}$, the monotonicity of $\E$
  yields that $0<\E(\hat{w}(s)\vert\varphi)\le R_{0}$ for every
  $s\in [0,t_{0}]$. Since
  $w_{0}\in \overline{\U}_{\frac{\varepsilon_{0}}{3},R_{0}}$ and
  $\hat{w}$ is continuous, there is a $0<\delta\le t_{0}$ such that
  $w(s)\in \U_{\varepsilon_{0},R_{0}}$ for every $s\in
  [0,\delta]$. Thus, by~\eqref{eq:29N},
  \begin{equation}
    \label{eq:28old}
      d(\hat{w}(s),\varphi)
      \le d(\hat{w}(s),w_{0})+d(w_{0},\varphi)
      \le\int_{0}^{s}\abs{\hat{w}'}(r)\,\dr
      + \frac{\varepsilon_{0}}{3}
      \le 2\,\frac{\varepsilon_{0}}{3}
  \end{equation}
  (firstly) for all $s\in [0,\delta]$, showing that
  $w(s)\in \overline{B}(\varphi, 2\,\frac{\varepsilon_{0}}{3})\cap[0<
  \E(\cdot\vert\varphi)\le R_{0}] $ for every $s\in [0,\delta]$. Since
  the right-hand side of~\eqref{eq:28old} is independent of $\delta$,
  and since $0<\E(\hat{w}(s)\vert\varphi)\le R_{0}$ for every
  $s\in [0,t_{0}]$, we can conclude that
  $\hat{w}(s)\in \overline{B}(\varphi,
  2\,\frac{\varepsilon_{0}}{3})\cap[0< \E(\cdot\vert\varphi)\le
  R_{0}]$ for all $s\in [0,t_{0}]$. Since
  $\hat{v}(0+)=v=\hat{w}(t_{0})$, it follows from~(ii) of Assumption~\ref{ass:H2} that
  $\hat{v}(s)=\hat{w}(s+t_{0})$ for every $s\in [0,t]$.  Thus,
  assumption $\E(\hat{v}(t)\vert\varphi)>0$ yields that
  $0<\E(\hat{w}(s)\vert\varphi)\le R_{0}$ for every $s\in [0,t]$ and so,
  $\hat{w}(s)\in \overline{B}(\varphi,
  2\,\frac{\varepsilon_{0}}{3})\cap[0< \E(\cdot\vert\varphi)\le
  R_{0}]$ for all $s\in [0,t_{0}+t]$, implying that
  $v(t)\in \U_{\varepsilon_{0},R_{0}}$. Moreover, we have shown that
  there is a
  $w_{0}\in \overline{\U}_{\frac{\varepsilon_{0}}{3},R_{0}}$, and
  $p$-gradient flow $\hat{w}$ of $\E$ satisfying $\hat{w}(0+)=w_{0}$
  and there is a $\hat{t}:=t_{0}+t\ge 0$ such that
  $\hat{w}(\hat{t})=\hat{v}(t)$.\medskip

  \emph{\underline{Step~2.\;}} We show that $\mathcal{D}$ is closed in
  $[0<\E(\cdot\vert\varphi)\le R_{0}]$. 

  Let $(v_{n})_{n\ge 1}\subseteq \mathcal{D}$ and
  $v\in [0<\E(\cdot\vert\varphi)\le R_{0}]$ such that $v_{n}\to v$ in
  $\mathfrak{M}$. By definition of $\mathcal{D}$, there are sequences
  $(t_{n}^{0})_{n\ge 1}\subseteq [0,+\infty)$,
  $(\hat{v}_{n}^{0})_{n\ge 1}\subseteq
  \overline{\U}_{\frac{\varepsilon_{0}}{3},R_{0}}$ and a sequence
  $(\hat{v}_{n})_{n\ge 1}$ of $p$-gradient flows of $\E$ with initial
  value $\hat{v}_{n}(0+)=\hat{v}_{n}^{0}$ satisfying
  $\hat{v}_{n}(t_{n}^{0})=v_{n}$ for every $n\ge 1$. By the lower
  semicontinuity of $\E$, there are $N\ge 1$ and $\hat{\varepsilon}>0$
  such that $\E(v_{n}\vert\varphi)\ge \hat{\varepsilon}$ for all
  $n\ge N$. Since $\E\circ \hat{v}_{n}$ is non-increasing on
  $[0,+\infty)$ and $\hat{v}_{n}(t_{n}^{0})=v_{n}$,
  \begin{equation}
    \label{eq:25}
    \E(\hat{v}^{0}_{n}\vert\varphi)\ge \hat{\varepsilon}\qquad\text{for all $n\ge N$}
  \end{equation}
  In addition,
  \begin{equation}
    \label{eq:60bis}
    \mathcal{I}:=\inf_{v\in
      \overline{\U}_{\varepsilon_{0},R_{0}}\cap
      [\E(\cdot\vert\varphi)\ge \hat{\varepsilon}]} g(v)>0,
  \end{equation}
  otherwise, there is a sequence $(\tilde{v}_{n\ge 1})$ in
  $\overline{\U}_{\varepsilon_{0},R_{0}}\cap
  [\E(\cdot\vert\varphi)\ge\hat{\varepsilon}]$ such that
  $g(\tilde{v}_{n})\to 0$ as $n\to \infty$. By
  Assumption~\ref{ass:H3}, there is a
  $\tilde{v}\in \overline{\U}_{\varepsilon_{0},R_{0}}\cap
  [\E(\cdot\vert\varphi)\ge \hat{\varepsilon}]$ such that up to a
  subsequence, $\tilde{v}_{n}\to \tilde{v}$ in $\mathfrak{M}$. Now, by
  the lower semicontinuity of $g$, $g(\tilde{v})=0$. On the other
  hand, $\E(\tilde{v})\ge \hat{\varepsilon}$ and so,
  Assumption~\ref{ass:H1} implies that $g(\tilde{v})>0$, showing that
  we arrived to a contradiction. Thus,~\eqref{eq:60bis} holds. Now,
  since each $\hat{v}_{n}$ is a $p$-gradient flow of $\E$ with initial
  value
  $\hat{v}_{n}(0+)=\hat{v}_{n}^{0}\in
  \overline{\U}_{\frac{\varepsilon_{0}}{3},R_{0}}$, \emph{energy
    dissipation equality}~\eqref{EDENew} gives that
  \begin{displaymath}
    0<t_{n}^{0}\,  \mathcal{I}^{p^{\mbox{}_{\prime}}}\le
    \int_{0}^{t_{n}^{0}}g ^{p^{\mbox{}_{\prime}}}(\hat{v}_{n}(s))\,\ds
    =\E(v_{n}^{0}\vert\varphi)-\E(\hat{v}_{n}(t_{n})\vert\varphi)
    \le \E(v_{n}^{0}\vert\varphi)\le R_{0},
  \end{displaymath}
  showing that the sequence $(t^{0}_{n})_{n\ge 1}$ is bounded. Thus,
  there is a $t_{0}\ge 0$ such that after possibly passing to a
  subsequence, $t_{n}^{0}\to t_{0}$ as $n\to +\infty$. Moreover,
  by~\eqref{eq:25},
  $(\hat{v}_{n}^{0})_{n\ge N}\subseteq
  \overline{\U}_{\frac{\varepsilon_{0}}{3},R_{0}}\cap
  [\E(\cdot\vert\varphi)\ge \hat{\varepsilon}]$. Hence,
  Assumption~\ref{ass:H3} implies that there is an
  $\hat{v}_{0}\in \overline{\U}_{\frac{\varepsilon_{0}}{3},R_{0}}\cap
  [\E(\cdot\vert\varphi)\ge \hat{\varepsilon}]$ such that after
  possibly passing to another subsequence,
  $\hat{v}_{n}^{0}\to \hat{v}_{0}$ in $\mathfrak{M}$. By~(i) of 
  Assumption~\ref{ass:H2}, there is a $p$-gradient flow $\hat{v}$ of
  $\E$ with initial value $\hat{v}(0+)=\hat{v}_{0}$ and by~(ii) of 
  Assumption~\ref{ass:H2}, we have $\hat{v}(t_{0})=v$ and so,
  $\hat{v}(t_{0})\in [0<\E(\cdot\vert\varphi)\le R_{0}]$. Thus and by
  stability property~\eqref{eq:67}, $v\in \mathcal{D}$.\medskip

  \emph{\underline{Step~3.\;}} We show that there are $T$, $C>1$, and a
  $p$-gradient flows $v : [0,T)\to \mathfrak{M}$ running through a $C$-valley
  $\mathcal{V}_{C,\U_{\varepsilon,R}}(\varphi)$.

  Let $s_{\mathcal{D}} : (0,R_{0}]\to (0,\infty)$ be defined
  by~\eqref{eq:29} on the interval $(0,R_{0}]$ and for $C>1$,
  \begin{equation}
    \label{eq:78}
    \mathcal{V}_{C}(r):=\Big\{v\in \mathcal{D}\cap
      [\E=r+\E(\varphi)]\;\Big\vert\;g(v)\le C\,
      s_{D}(r)\Big\}\subseteq V_{C,\mathcal{D}}(\varphi).
  \end{equation}
  By hypothesis, the set $\mathcal{D}\subseteq \U_{\varepsilon_{0},R_{0}}$ is
  non-empty. Thus, there are $0<R_{1}\le R_{0}$ and
  $v_{0}\in \overline{\U}_{\frac{\varepsilon_{0}}{3},R_{0}}$
  such that $\E(v_{0}\vert\varphi)=R_{1}$. By~(i) of Assumption~\ref{ass:H2}, there
  is a $p$-gradient flow $v$ of $\E$ with $v(0+)=v_{0}$,
  and by stability property~\eqref{eq:67}, there is a $0<T\le +\infty$ such that
  \begin{equation}
    \label{eq:60}
    v(t)\in \mathcal{D}\cap [\E(\cdot\vert\varphi)>0]
    \qquad\text{for all $t\in [0,T)$ 
      .}
  \end{equation}
  Let $T:=\inf\{t>0\,\vert\, \E(v(t)\vert\varphi)=0\}$. Then, since
  $\E(v_{0}\vert\varphi)=R_{1}>0$, the continuity of $v$ yields that
  $0<T\le +\infty$ and $\lim_{t\to T-}\E(v(t)\vert\varphi)=0$. By
  Proposition~\ref{propo:chara-p-curves}, $\E\circ v$ is strictly
  decreasing on $[0,T)$ and by
  Proposition~\ref{prop:Energy-abs-inequality}, $\E\circ v$ is locally
  absolutely continuous on $[0,T)$. Thus, by the intermediate value
  theorem, the mapping $\E\circ v : [0,T)\to (0,R_{1}]$ is a
  homeomorphism, and by~\eqref{eq:60}, one has that
  \begin{equation}
    \label{eq:70}
    \mathcal{V}_{C}(r)\neq \emptyset\qquad\text{ for all $r\in (0,R_{1}]$ and all
      $C>1$.}
  \end{equation}
  To see this, assume that the contrary holds. Then,
  by~\eqref{eq:60}, there are $r'\in (0,R_{1}]$ and $C'>1$ such
  that $g(v)> C'\, s_{D}(r')$ for all
  $v\in \mathcal{D}\cap [\E=r'+\E(\varphi)]$, implying that $1>C_{1}$,
  which obviously is a contradiction.

  Now, for $1<C_{1}<
  C$ and $0<R'_{1}\le R_{1}$, let $v_{0}\in \mathcal{V}_{C_{1}}(R'_{1})$ and $v$ be the
  $p$-gradient flow of $\E$ with initial value $v(0+)=v_{0}$. By~(iii)
  of Assumption~\ref{ass:H2} and since $C_{1}<\frac{C_{1}+C}{2}$,
  there is a $T_{0}>0$ such that
\begin{displaymath}
g(v(t))< \frac{C_{1}+C}{2}\, S_{\mathcal{D}}(R'_{1}) \qquad\text{for all
$t\in [0,T_{0})$.}
\end{displaymath}
By hypothesis, the function
$t\mapsto S_{\mathcal{D}}(\E(v(t)))$ is lower semicontinuous. Hence,
there is a $T_{1}\in (0,T_{0})$ such that
\begin{displaymath}
C\,S_{\mathcal{D}}(\E(v(t)\vert\varphi))>\frac{C_{1}+C}{2}\, s_{\mathcal{D}}(R'_{1})
\qquad\text{for all $t\in [0,T_{1})$.}
\end{displaymath}
Combining these two inequalities on the
interval $[0,T_{1})$ yields that
 \begin{equation}
   \label{eq:80}
    g(v(t))<\frac{C_{1}+C}{2}\, S_{\mathcal{D}}(r)
    <C\,S_{\mathcal{D}}(\E(v(t)\vert\varphi))\qquad\text{for all $t\in [0,T_{1})$.}
 \end{equation}
Since there might be some $t\in (0,T_{1})$ such that
 $v(t)\neq \mathcal{D}$, we need to apply the continuity of the
 $p$-gradient flow $v$ and stability property~\eqref{eq:67} to
 conclude that there is a $0<T_{\ast}\le T_{1}$ such that
 $v(t)\in \mathcal{D}$ and~\eqref{eq:80} holds for all $t\in
 [0,T_{\ast})$. Thereby, we have shown, for every $0<C_{1}<C$,
 $0<R'_{1}\le R_{1}$, and
 \begin{equation}
   \label{eq:79}
   \begin{cases}
    \!\! &\!\!\text{$v_{0}\in \mathcal{V}_{C_{1}}(R'_{1})$, there are $0<T_{\ast}\le T_{1}$
       and a $p$-gradient flow $v$}\\
    \!\! &\!\!\text{ such that  }v(t)\in \mathcal{V}_{C,\mathcal{D}}(\E(v(t)\vert\varphi))
     \text{ for all $t\in [0,T_{\ast})$.}
   \end{cases}
 \end{equation}

\emph{\underline{Step~4.\;}} We show that there is a talweg curve $x :
(0,R_{1}]\to V_{C,\mathcal{D}}(\varphi)$ of finite length $\gamma(x)$.

Let $v : [0,T_{\ast}]\to \mathfrak{M}$ be a $p$-gradient flow given
by~\eqref{eq:79} from some $v_{0}\in
\mathcal{V}_{C_{1}}(R_{1})$. Recall, $\E\circ v$ is strictly
decreasing and continuous on $[0,T_{\ast}]$. Here, we need to
consider two cases:\medskip

\emph{\underline{1. Case :}} Suppose
 \begin{equation}
   \label{eq:81}
   \lim_{t\to T_{\ast}-}\E(v(t)\vert\varphi)=0.
\end{equation}

Then, $\E(v(\cdot)\vert\varphi) : [0,T_{\ast}]\to [0,R_{1}]$ is a
strictly decreasing homeomorphism. By
Proposition~\ref{propo:chara-p-curves}, one has that
$\E(v(\cdot)\vert\varphi)\in W^{1,1}(0,T_{\ast})$ and by
Assumption~\ref{ass:H1}, $\fdt \E(v(t)\vert\varphi)<0$ for a.e.
$t\in (0,T_{\ast})$. Let $\hat{h} : [0,T_{\ast}]\to [0,R_{1}]$ be
defined by $\hat{h}(t):=\E(v(T_{\ast}-t)\vert\varphi)$ for every
$t\in [0,T_{\ast}]$. Then, $\hat{h}$ also belongs to
$W^{1,1}(0,T_{\ast})\cap C[0,T_{\ast}]$, $\hat{h}$ is strictly
increasing, $\abs{\{\hat{h}'(t)=0\}}=0$ and $\hat{h}(0)=0$.  Now, we
define the curve $x : [0,R_{1}]\to \mathfrak{M}$ by
 \begin{equation}
   \label{eq:90}
   x(r)=v(\hat{h}^{-1}(r))\qquad\text{for every $r\in [0,R_{1}]$.}
 \end{equation}
 Since $v(t)\in \mathcal{V}_{C,\mathcal{D}}(\E(v(t)\vert\varphi))$ for
 all $t\in [0,T_{\ast})$ and by the construction of $\hat{h}^{-1}$,
 $x$ is a talweg through the $C$-valley
 $\mathcal{V}_{C,\mathcal{D}}(\varphi)$ satisfying
 \begin{equation}
   \label{eq:79bis}
    x(r)\in \mathcal{V}_{C,\mathcal{D}}(\E(x(r)\vert\varphi))\qquad\text{for all $r\in (0,R_{1}]$.}
 \end{equation}
 and~\eqref{eq:85} with $R$ replaced by $R_{1}$. Since
 $\hat{h}\in W^{1,1,}(0,\delta)$, $\hat{h}'>0$ a.e. on $[0,\delta]$,
 and the inverse $\hat{h}^{-1}$ satisfies the Lusin~(N) property on
 $[0,R_{1}]$ (see Remark~\ref{rem:LusinNproperty}), one has that
 $\hat{h}^{-1}\in W^{1,1}(0,R_{1})$. Thus,
 $x\in AC_{loc}(0,R_{1};\mathfrak{M})$.
 Furthermore, by~\eqref{eq:90}, $x$ is a $p$-gradient curve of $\E$
 and by~\eqref{eq:79bis}, $x((0,R_{1}])\subseteq \mathcal{D}$. Since
 $\mathcal{D}\subseteq \U_{\varepsilon_{0},R_{0}}$,
 hypothesis~\eqref{eq:29N} yields that $x$ has finite length
 $\gamma(x)$. This concludes the proof of the theorem in this
 case.\medskip

\emph{\underline{2. Case :}} Suppose
\begin{equation}
  \label{eq:83}
  \lim_{t\to T_{\ast}-}\E(v(t)\vert\varphi)=\E(v(T_{\ast})\vert\varphi)>0.
\end{equation}
Then, by stability property~\eqref{eq:67},
$v(T_{\ast})\in \mathcal{V}_{C}(\E(v(T_{\ast})\vert\varphi))$.  We set
$r_{0}=R_{1}$, $v_{0}^{(0)}=v_{0}$, $t_{0}=T_{\ast}$, and
$r_{1}=\E(v(t_{0})\vert\varphi)$, and $v_{1}(t):=v(t)$,
$(t\in [0,t_{0}])$. Now, by~\eqref{eq:70} and~\eqref{eq:79}, we can
repeat the construction of $v_{1}$ iteratively for every integer
$n\ge 2$ as follows: with $r_{n-1}:=\E(v_{n-1}(t_{n-2})\vert\varphi)\in [0,r_{n-2})$, if
$r_{n-1}>0$, we choose a $v_{n-1}^{(0)}\in\mathcal{V}_{C_{1}}(r_{n-1})$
(by~\eqref{eq:70}, where one takes $0<C_{1}<C$), if $r_{n-1}=0$ then
one stops the iteration (cf~\emph{Case~1.}). By~\eqref{eq:79}, there are $t_{n-1}>0$ and a
$p$-gradient flow $v_{n}$ of $\E$ satisfying $v_{n}(0+)=v_{n-1}^{(0)}$
and
\begin{equation}
  \label{eq:82}
    v_{n}(t)\in \mathcal{V}_{C}(\E(v_{n}(t)\vert\varphi))\qquad\text{for all $t\in [0,t_{n-1})$.}
\end{equation}
Moreover, if $\E(v_{n}(t_{n-1})\vert\varphi)>0$, then by~\eqref{eq:67}, 
$v_{n}(t_{n-1}) \in
\mathcal{V}_{C}(\E(v_{n}(t_{n-1})\vert\varphi))$. Further, by setting
$r_{n}=\E(v_{n}(t_{n-1})\vert\varphi)$, and by
Assumption~\ref{ass:H1}, the function $\hat{h}_{n}: [0,t_{n-1}]\to [r_{n},r_{n-1}]$ defined by
\begin{displaymath}
\hat{h}_{n}(t):=\E(v_{n}(t_{n-1}-t)\vert\varphi)\qquad\text{for every
  $t\in [0,t_{n-1}]$,}
\end{displaymath}
is a homeomorphism, strictly increasing on $[0,t_{n-1})$, and
satisfying $\hat{h}_{n}(0)=r_{n-1}$. Moreover, $\hat{h}_{n}$ belongs
to $W^{1,1}(0,t_{n-1})\cap C[0,t_{n-1}]$ with
$\abs{\{\hat{h}'_{n}(t)=0\}}=0$
(cf. Remark~\ref{rem:LusinNproperty}). Now, we define the curve
$x_{n} : [r_{n+1},r_{n}]\to \mathfrak{M}$ by
 \begin{equation}
   \label{eq:90}
   x_{n}(r)=v(\hat{h}^{-1}_{n}(r))\qquad\text{for every $r\in [r_{n},r_{n-1}]$,}
 \end{equation}
where $\hat{h}^{-1}_{n}$ is the inverse of $\hat{h}^{-1}$.
Since $\hat{h}_{n}^{-1}\in W^{1,1}(r_{n},r_{n-1})$ with
$\abs{\{(\hat{h}_{n}^{-1})'=0\}}=0$, one has that
$x_{n}\in AC(r_{n},r_{n-1};\mathfrak{M})$. In addition, by
construction of $\hat{h}_{n}^{-1}$, for every $r\in [r_{n},r_{n-1}]$,
there is a unique $t\in [0,t_{n}]$ such that $\hat{h}_{n}^{-1}(r)=t$
and $\E(v_{n}(t)\vert\varphi)=r$.

Now, if there is an $N>1$ such that the curve $v_{N}$ is the first among
$\{v_{n}\}_{n=1}^{N-1}$ satisfying the limit~\eqref{eq:81} where
$T_{\ast}$ is replaced by $t_{N}$, then $r_{N}=0$ and there is a \emph{finite}
partition
\begin{displaymath}
\mathcal{P}\quad:\quad 0=r_{N}<r_{N-1}<\cdots<r_{1}<r_{0}=R_{1}
\end{displaymath}
of the interval $[0,R_{1}]$ and a curve $x : [0,R_{1}]\to \mathfrak{M}$ defined by
\begin{equation}
  \label{eq:91}
  x(t)=\sum_{i=0}^{N}v_{n}(\hat{h}_{n}^{-1}(t))
  \;\mathds{1}_{(r_{n},r_{n-1}]}(t)\qquad
\text{for every $t\in [0,R_{1}]$}
\end{equation}
which is a \emph{piecewise} $AC$ talweg through the $C$-valley
$\mathcal{V}_{C,\mathcal{D}}(\varphi)$ satisfying~\eqref{eq:85} for
$R$ replaced by $R_{1}$ and
$x_{\vert (r_{n},r_{n-1}]}=x_{n}\in
AC(r_{n},r_{n-1};\mathfrak{M})$. Furthermore, by~\eqref{eq:91}, $x$ is
a piecewise $p$-gradient curve of $\E$ and by~\eqref{eq:82}, $x((0,R_{1}])\subseteq
 \mathcal{D}$. Since $\mathcal{D}\subseteq
 \U_{\varepsilon_{0},R_{0}}$, hypothesis~\eqref{eq:29N} yields that $x$ has finite
 length $\gamma(x)$.\medskip

If for every integer $n\ge 1$, $v_{n}$ satisfies~\eqref{eq:83} with
$T_{\ast}$ replaced by $t_{n}$, then stability property~\eqref{eq:67}
yields that each $v_{n}$ satisfies~\eqref{eq:82} with
$t=t_{n}$. There is an $\alpha\in [0,R_{1})$ such that the
family $\{(r_{n},r_{n-1}]\}_{n\ge 1}$ defines a countable partition of the interval
$(\alpha,R_{1}]$ and the function
$x_{\alpha} : (\alpha,R_{1}]\to \mathcal{V}_{C,\mathcal{D}}(\varphi)$ defined by
\begin{equation}
  \label{eq:86}
  x_{\alpha}(t)=\sum_{i=0}^{\infty}v_{n}(\hat{h}_{n}^{-1}(t))
  \;\mathds{1}_{(r_{n},r_{n-1}]}(t)
\end{equation}
for every $t\in (\alpha,R_{1}]$ has the properties that
$x_{\alpha}\in AC((r_{n},r_{n-1}],\mathfrak{M})$ for all $n\ge 1$, and
\begin{equation}
  \label{eq:84}
  x_{\alpha}(t)\in
  \mathcal{V}_{C,\mathcal{D}}(\E(x_{n}(t)\vert\varphi))
\qquad\text{for all $t\in (\alpha,R_{1}]$}
\end{equation}
and
\begin{equation}
  \label{eq:76}
  \E(x_{\alpha}(r)\vert\varphi)=r\qquad\text{ for every $r\in
(\alpha,R_{1}]$.}
\end{equation}
Moreover, by~\eqref{eq:86}, $x$ is a piecewise $p$-gradient curve of
$\E$ and by~\eqref{eq:84}, $x((0,R_{1}])\subseteq \mathcal{D}$. Since
$\mathcal{D}\subseteq \U_{\varepsilon_{0},R_{0}}$,
hypothesis~\eqref{eq:29N} yields that $x$ has finite length
$\gamma(x)$.

Let $\mathcal{T}$ be the set of all pairs $(\alpha,x_{\alpha})$
for every $\alpha\in [0,R_{1})$ and curves
$x_{\alpha} : (\alpha,R_{1}]\to \mathcal{V}_{C,\mathcal{D}}(\varphi)$
of finite length $\gamma(x_{\alpha})$
satisfying~\eqref{eq:84}, \eqref{eq:76}, and there is countable partition
$\{I_{n}\}_{n\ge 1}$ of $(\alpha,R_{1}]$ of nontrivial intervals
$I_{n}\subseteq (0,R_{1}]$, for which
$x_{\alpha}\in AC(I_{n},\mathfrak{M})$ for all $n\ge 1$. Then, due to
the function $x_{\alpha}$ constructed in~\eqref{eq:86}, the set
$\mathcal{T}$ is non-empty. We can define a \emph{partial ordering}
``$\le$'' on $\mathcal{T}$ by setting that for all $(\alpha, x_{\alpha})$,
$(\hat{\alpha}, x_{\hat{\alpha}})\in \mathcal{T}$, one has
\begin{displaymath}
  (\alpha, x_{\alpha})\le (\hat{\alpha}, x_{\hat{\alpha}})\qquad
  \text{if $\hat{\alpha}\le \alpha$ and $x_{\hat{\alpha}\vert (\alpha,R_{1}]}=x_{\alpha}$.}
\end{displaymath}
Then, by Zorn's Lemma, there is a maximal element
  $(\alpha_{0},x_{\alpha_{0}})\in \mathcal{T}$. If we assume that
  $\alpha_{0}>0$, then by stability property~\eqref{eq:67},
 \begin{displaymath}
  x(\alpha_{0})\in \mathcal{V}_{C,\mathcal{D}}(\E(x(\alpha_{0})\vert\varphi))
\end{displaymath}
and so by using the same arguments as given at the beginning of
\emph{Case 2.}, we can construct an element
$(\hat{\alpha}, x_{\hat{\alpha}})\in \mathcal{T}$ satisfying
$(\alpha_{0}, x_{\alpha_{0}})\le (\hat{\alpha}, x_{\hat{\alpha}})$,
which contradicts the fact that $(\alpha_{0}, x_{\alpha_{0}})$ is the
maximal element of $\mathcal{T}$. Therefore, $\alpha_{0}=0$, which
shows that there is a \emph{piecewise} $AC$ talweg
$x_{\alpha_{0}} : (0,R_{1}]\to \mathcal{V}_{C,\mathcal{D}}(\varphi)$
through the $C$-valley $\mathcal{V}_{C,\mathcal{D}}(\varphi)$ of
finite length $\gamma(x_{\alpha_{0}})$ satisfying~\eqref{eq:85} with
$R$ replaced by $R_{1}$. This complete the proof of this theorem.
\end{proof}

Due to Theorem~\ref{thm:ex-talweg}, we can characterize the validity
of the Kurdyka-\L{}ojasie\-wicz inequality~\eqref{eq:16} for functionals $\E$ defined
on a metric space. 

\begin{theorem}[{\bfseries Characterization of K\L{} inequality}]
  \label{thm:chara-KL-talweg}
  Let $\E : \mathfrak{M}\to (-\infty,+\infty]$ be a proper lower
  semicontinuous functional and $g : \mathfrak{M}\to [0,+\infty]$ a
  proper lower semicontinuous strong upper gradient of $\E$ satisfying
  Assumption~\ref{ass:H2}. Suppose, for $\varphi\in \mathbb{E}_{g}$,
  there are $\varepsilon>0$ and $R>0$ such that the set
  $\U_{\varepsilon,R}$ given by~\eqref{eq:65} satisfies
  Assumption~\ref{ass:H3} and hypothesis~\eqref{eq:15} holds. Further,
  suppose, there are $0<\varepsilon_{0}<\varepsilon$ and $0<R_{0}< R$
  such that the set $\mathcal{D}\subseteq \U_{\varepsilon_{0},R_{0}}$
  given by~\eqref{eq:89} is non-empty and~\eqref{eq:29} holds.

  Then the following statements are equivalent.
  \begin{enumerate}[topsep=3pt,itemsep=1ex,partopsep=1ex,parsep=1ex]

  \item[(1)] ({\bfseries $\E$ satisfies a K\L{} inequality}) There is an
    $0<R_{1}\le R_{0}$ such that $\E$ satisfies
    Kurdyka-\L{}ojasiewicz inequality~\eqref{eq:16} on
    \begin{displaymath}
      \tilde{U}:=\U_{\varepsilon,R} \cap [0<\E(\cdot\vert\varphi)\le R_{1}].
    \end{displaymath}

  \item[(2)] ({\bfseries $p$-gradient flows of finite length}) There is an
    $0<R_{1}\le R_{0}$ such that for  $0<T\le \infty$, every piecewise $p$-gradient flow
    $v : [0,T)\to \mathfrak{M}$ of $\E$ satisfying~\eqref{eq:30} for
    some $0\le t_{0}<T$ and with $\U$ replaced by $\tilde{\U}$, has
    finite length $\gamma(v)$ given by~\eqref{eq:52}. In particular,
    there is a continuous, strictly increasing function
    $\theta\in W^{1,1}_{loc}(\R)$ satisfying $\theta(0)=0$ such that
    for every $p$-gradient flow $v : [0,T)\to \mathfrak{M}$ of $\E$
    satisfying~\eqref{eq:30}, one has that~\eqref{eq:26} holds.

  \item[(3)] ({\bfseries Existence of a piecewise AC-talweg}) There
    are $C>1$, $0<R_{1}\le R_{0}$ and a piecewise $AC$ talweg
    $x : (0,R_{1}] \to \mathfrak{M}$ of finite length $\gamma(x)$
    through the $C$-valley $\mathcal{V}_{C,\mathcal{D}}(\varphi)$
    satisfying~\eqref{eq:85} with $R$ replaced by $R_{1}$.
  \end{enumerate}
\end{theorem}

\begin{proof}
We only need to note that the implication $(1)\Rightarrow (2)$ is a
  consequence of Theorem~\ref{thm:finite-length}, $(2)\Rightarrow (3)$
  holds by Theorem \ref{thm:ex-talweg}, and $(3)\Rightarrow (1)$
  follows from Theorem~\ref{thm:2-Talweg-KL}.
\end{proof}

%
%
%
%
%
%

\subsection{Trend to equilibrium of $p$-gradient flows in the metric sense}
\label{secbehaviour}

This subsection is dedicated to establishing the trend to equilibrium in
the metric sense of $p$-gradient flows in metric spaces. The following
theorem is the main result.

\begin{theorem}[{\bfseries Trend to equilibrium in the metric sense}]
  \label{thm:convergence}
  Let $\E : \mathfrak{M}\to (-\infty,+\infty]$ be a proper functional
  with strong upper gradient $g$, and
  $v : [0,+\infty)\to \mathfrak{M}$ be a $p$-gradient flow of $\E$
  with non-empty $\omega$-limit set $\omega(v)$. Suppose, $\E$ is
  lower semicontinuous on $\overline{\mathcal{I}}_{\overline{t}}(v)$
  for some $\overline{t}\ge 0$ and for
  $\varphi\in \omega(v)\cap \mathbb{E}_{g}$, there is an
  $\varepsilon>0$ such that the set $B(\varphi,\varepsilon)$ satisfies
  hypothesis~\eqref{eq:15}.

  If there is a strictly increasing function
  $\theta\in W^{1,1}_{loc}(\R)$ satisfying $\theta(0)=0$ and
  $\abs{[\theta>0,\theta'=0]}=0$ such that $\E$ satisfies the
  Kurdyka-\L{}ojasiewicz inequality~\eqref{eq:16} on
  \begin{equation}
    \label{eq:96}
 \U_{\varepsilon}=B(\varphi,\varepsilon)\cap[\E(\cdot\vert\varphi)>0]\cap
   [\theta'(\E(\cdot\vert\varphi))>0],
  \end{equation}
  then $v$ has finite length and
   \begin{equation}
         \label{eq:17N}
         \lim_{t\to\infty}v(t)=\varphi\qquad\text{in $\mathfrak{M}$}.
       \end{equation}
\end{theorem}

\begin{remark}[\emph{$\omega$-limit point and points of equilibrium}]
  Note, due to statement~\eqref{propo:omega-limit-gradientflow-claim-5} of
  Proposition~\ref{propo:omega-limit-gradientflow}, if the strong upper gradient
  $g$ of $\E$ is lower semicontinuous on
  $\mathfrak{M}$, then one has that $\omega(v)\subseteq \mathbb{E}_{g}$.
\end{remark}

\begin{remark}[\emph{Consistency with the \L{}ojasiewicz-Simon inequality}]
  The function $\theta$ given by~\eqref{eq:38bis} satisfies the condition
  \begin{math}
    \abs{[\theta>0,\theta'=0]}=0
  \end{math}
  in Theorem~\ref{thm:convergence}.
\end{remark}

\begin{proof}[Proof of Theorem~\ref{thm:convergence}]
  Let $\varphi\in \omega(v)\cap \mathbb{E}_{g}$. Then, there is a
  sequence $(t_{n})_{n\ge 1}$ such that $t_{n}\uparrow \infty$ and
  \begin{equation}
    \label{eq:22}
    \lim_{n\to\infty}v(t_{n})=\varphi\qquad\text{ in $\mathfrak{M}$.}
  \end{equation}
  Since $\E$ is lower semicontinuous on
  $\overline{\mathcal{I}}_{\overline{t}}(v)$ for some $\overline{t}\ge
  0$, Proposition~\ref{propo:omega-limit-gradientflow} implies that
  limit~\eqref{eq:9} holds. Moreover, $\E$ is a strict Lyapunov
  function of $v$. 
  Thus, to show that the limit~\eqref{eq:17N} holds, it is sufficient to consider the
  case $\E(v(t)\vert\varphi)>0$ on $[0,+\infty)$.

  By hypothesis, for $\varphi\in \omega(v)\cap \mathbb{E}_{g}$, there
  is $\varepsilon>0$ such that the set $B(\varphi,\varepsilon)$
  satisfies hypothesis~\eqref{eq:15} and $\E$ satisfies a
  Kurdyka-\L{}ojasiewicz inequality on the set $\U_{\varepsilon}$
  given by~\eqref{eq:96}. Thus, we intend to apply
  Theorem~\ref{thm:finite-length} to $v$. For this, we need to show
  that $v$ satisfies condition~\eqref{eq:30} with $\U$ replaced by
  $\U_{\varepsilon}$ and for some $0\le t_{0}<T=+\infty$.

  By~\eqref{eq:22}, there is a $n_{0}\ge 1$ such that
   \begin{equation}
     \label{eq:37}
    v(t_{n})\in B(\varphi,\varepsilon)\qquad\text{for all $n\ge n_{0}$.}
  \end{equation}
  The, for every $n\ge n_{0}$, we define the \emph{first exit time
    with respect to $t_{n}$} by
  \begin{equation}
    \label{eq:12}
    t_{n}^{(1)}:=\inf\big\{t\ge
    t_{n}\,\big\vert\,d(v(t),\varphi)=\varepsilon\,\big\}.
  \end{equation}
  Since $v$ is continuous on $(0,\infty)$ with values in $\mathfrak{M}$, it follows
  that $t_{n}^{(1)}>t_{n}$ for every $n\ge n_{0}$. To see that $v$
  satisfies ~\eqref{eq:30} with $\U$ replaced by $\U_{\varepsilon}$
  and $T=+\infty$, we need to show that there is an $n_{1}\ge n_{0}$ such that
  $t_{n_{1}}^{(1)}=+\infty$.

  To prove this claim, we assume that the
  contrary is true and we shall arrive at a contradiction. Then,
  \begin{displaymath}
    0<t_{n}<t_{n}^{(1)}<\infty\qquad\text{ for all $n\ge n_{0}$}
  \end{displaymath}
  and by the continuity of $v$, we see that
  \begin{equation}
    \label{eq:23}
    d(v(t_{n}^{(1)}),\varphi)=\varepsilon\qquad\text{for all $n\ge n_{0}$.}
  \end{equation}
  By hypothesis, there is a strictly increasing function
  $\theta\in W^{1,1}_{loc}(\R)$ satisfying $\theta(0)=0$ and
  $\abs{[\theta>0,\theta'=0]}=0$. We use the auxiliary function $\mathcal{H}$
  from~\eqref{eq:19} on the interval $(0,+\infty)$. Since
  $\theta\in AC_{loc}(\R)$ and $t\mapsto \E(v(t)\vert\varphi)$ is
  decreasing, $\mathcal{H}$ is differentiable a.e. on $(0,+\infty)$
  and the chain rule~\eqref{eq:36} holds for a.e. $t\in
  (0,+\infty)$. Since $\E(v(t)\vert\varphi)>0$ on $(0,+\infty)$ and
  $\abs{[\theta>0,\theta'=0]}=0$, we have that
  $\theta'(\E(v(t))\vert\varphi)>0$ for a.e. $t\in (0,+\infty)$. Thus,
  for every $n\ge n_{0}$, $v(t)\in \U_{\varepsilon}$ for
  a.e. $t\in (t_{n},t_{n}^{1})$ and so by Kurdyka-\L{}ojasiewicz
  inequality~\eqref{eq:16}, we can conclude that
  inequality~\eqref{eq:20} holds for a.e.  $t\in
  (t_{n},t_{n}^{1})$. Integrating inequality~\eqref{eq:20} over
  $(t_{n},t)$ for $t\in (t_{n},t_{n}^{(1)}]$ and using
  Proposition~\ref{MDer} together with the fact that
  $\mathcal{H}(t)>0$ for every $t>0$, we get
  \begin{equation}
    \label{eq:21}
    \begin{split}
      d(v(t),\varphi) &\le d(v(t),v(t_n))+ d(v(t_n),\varphi) \\
      &\le \int_{t_{n}}^{t} \abs{v'}(r)\,\textrm{d}r
      + d(v(t_{n}),\varphi)
      \le \mathcal{H}(t_{n})+d(v(t_{n}),\varphi)
    \end{split}
  \end{equation}
  for every $t\in (t_{n},t_{n}^{(1)}]$ and $n\ge n_{0}$. In particular,
  \begin{equation}
    \label{eq:68}
    d(v(t_{n}^{(1)}),\varphi)
    \le\mathcal{H}(t_{n})+d(v(t_{n}),\varphi)\qquad\text{for all $n\ge
      n_{0}$}
  \end{equation}
  By limit~\eqref{eq:9}, the continuity of $\theta$, and since
  $\theta(0)=0$, we have that
  \begin{displaymath}
    \lim_{t\to\infty}\mathcal{H}(t)=0.
  \end{displaymath}
  Thus, by~\eqref{eq:22} and~\eqref{eq:68}, we can conclude that
  \begin{displaymath}
    \lim_{n\to\infty}d(v(t_{n}^{(1)}),\varphi)=0,
  \end{displaymath}
  which contradicts~\eqref{eq:23}. Therefore, our assumption is false
  and our claim that there is a $n_{1}\ge n_{0}$ satisfying
  $t_{n_{1}}^{1}=+\infty$ holds, proving condition~\eqref{eq:30} for
  some $0\le t_{0}<T=+\infty$ where $\U$ is replaced by
  $\U_{\varepsilon}$. Thus, Theorem~\ref{thm:finite-length} yields
  that the $p$-gradient flow $v$ has finite length. In other words,
  the metric derivative $\abs{v'}$ of $v$ belongs to $L^{1}(0,\infty)$
  and by~\eqref{eq:1} with $m=\abs{v'}$, we can apply the Cauchy
  criterion to conclude that $\lim_{t\to+\infty}v(t)$ exists in
  $\mathfrak{M}$. By~\eqref{eq:22} this limit needs to coincide
  with limit~\eqref{eq:17N} and therefore the statement of this
  theorem holds.
\end{proof}

\subsection{Decay rates and finite time of extinction}
\label{subsect:decayrates-finite-extinction}
In contrast to the general Kurdyka-\L{}ojasiewicz
inequality~\eqref{eq:16}, the \L{}ojasiewicz-Simon
inequality~\eqref{eq:38} has the advantage to derive \emph{decay
  estimates} of the trend to equilibrium in the metric sense and to
provide upper bounds on the \emph{extinction time}. We emphasize that
$p$-gradient flows trend with \emph{polynomial rate} to an equilibrium
in the metric sense if the \L{}ojasiewicz-exponent $0<\alpha<1/p$,
with \emph{exponential rate} if $\alpha=1/p$ and $p$-gradient flows
\emph{extinguish in finite time} if $1/p<\alpha\le 1$.\medskip

Our next result generalizes (partially) the ones in~\cite{MR2019030}, \cite[Theorem~2.7 \&
  Remark~2.8]{MR2289546} in the Hilbert space setting,
  and~\cite{MR3832005} in the $2$-Wasserstein setting.

\begin{theorem}[{\bfseries Decay estimates and finite time of extinction}]
  \label{thm:decayrates}
  Let $\E : \mathfrak{M}\to (-\infty,+\infty]$ be a proper functional
  with strong upper gradient $g$, and
  $v : [0,+\infty)\to \mathfrak{M}$ be a $p$-gradient flow of $\E$
  with non-empty $\omega$-limit set $\omega(v)$. Suppose, for some
  $\overline{t}>0$, $\E$ is lower semicontinuous on
  $\overline{\mathcal{I}}_{\overline{t}}(v)$, and for
  $\varphi\in \omega(v)\cap \mathbb{E}_{g}$, there are $\varepsilon$,
  $c>0$, and $\alpha\in (0,1]$ such that $\E$ satisfies
  a~\L{}oja\-sie\-wicz-Simon inequality~\eqref{eq:38} with exponent
  $\alpha$ on $B(\varphi,\varepsilon)\cap D(\E)$. Then,
  \begin{align*}
    d(v(t),\varphi) &\le \tfrac{c}{\alpha}\left(\E(v(t)\vert\varphi)\right)^{\alpha} =
     \; \mathcal{O}\left(t^{^{-\frac{\alpha(p-1)}{1-p\alpha}}}\right)\hspace{2.5cm}\text{if
       $0<\alpha<\tfrac{1}{p}$}\\
     d(v(t),\varphi)  &\le c\,p\left(\E(v(t)\vert\varphi)\right)^{\frac{1}{p}}
       \le   c\,p\,\left(\E(v(t_{0})\vert\varphi)\right)^{\frac{1}{p}}\;
                        e^{-\tfrac{t}{pc^{p^{\prime}}}}\hspace{0.75cm}\text{if
      $\alpha=\tfrac{1}{p}$}\\
     d(v(t),\varphi) &\le
    \begin{cases}
   \tilde{c}\,(\hat{t}-t)^{\frac{\alpha(p-1)}{p\alpha-1}}
      & \quad
    \text{if\quad $t_{0}\le t\le \hat{t}$,}\\
    0 & \quad
    \text{if $t>\hat{t}$,}\\
    \end{cases}
    \mbox{}\hspace{2cm}\text{if $\tfrac{1}{p}<\alpha\le 1$,}
  \end{align*}
  where,
  \begin{align*}
    & \tilde{c}:=\left[\left[\tfrac{1}{\alpha^{\alpha -1}c}\right]^{\frac{p^{\mbox{}_{\prime}}-1}{\alpha}}\,
      \tfrac{p\alpha-1}{\alpha(p-1)}\right]^{\frac{\alpha(p-1)}{p\alpha-1}},\\
   &\hat{t}:=t_{0}+ \, \alpha^{\frac{\alpha -1}{\alpha (p -1)}}\,
   c^{\frac{1}{\alpha (p -1)}}\,\tfrac{\alpha(p-1)}{p\alpha-1}\,
  (\E(v(t_{0})\vert \varphi))^{\frac{p\alpha-1}{\alpha(p-1)}},
\end{align*}
and $t_{0}\ge 0$ can be chosen to be the ``first entry time'', that is,
  $t_{0}\ge 0$ is the smallest time
  $\hat{t}_{0}\in [0,+\infty)$ such that $v([\hat{t}_{0},+\infty))\subseteq
  B(\varphi,\varepsilon)$.
\end{theorem}

\begin{proof}
  As in the proof of Theorem~\ref{thm:convergence}, it remains to
  consider the situation, when $\E(v(t))>\E(\varphi)$ for all
  $t\ge0$. In addition, we assume that $\E$ satisfies a
  \L{}ojasiewicz-Simon inequality with exponent $\alpha \in (0,1]$ on
  $B(\varphi,\varepsilon)\cap D(g)$. In this case, the function
  $\theta$ is given by~\eqref{eq:38bis} and so, $\mathcal{H}$ defined
  in~\eqref{eq:19} reduces to
  \begin{equation}
    \label{eq:50}
    \mathcal{H}(t)=\frac{c}{\alpha}\,(\E(v(t)\vert \varphi))^{\alpha}
  \end{equation}
  for every $t\ge0$. Let $t_{0}\ge0$ be the \emph{first entry time} of $v$ in
  $B(\varphi,\varepsilon)$. By Theorem~\ref{thm:convergence} and since
  $v(t)\in D(g)$ for a.e. $t\in (0,+\infty)$, $v$ satisfies
  condition~\eqref{eq:30} for $0\le t_{0}<T=+\infty$ where $\U$ is
  replaced by $B(\varphi,\varepsilon)\cap D(g)$ and so, we can apply the
  \L{}ojasiewicz-Simon inequality~\eqref{eq:38} to $v=v(t)$ for a.e.
  $t\in [t_{0},+\infty)$. Moreover, the function
  $\mathcal{H}$ is differentiable a.e. on $(0,+\infty)$ and
  chain rule~\eqref{eq:36} holds for a.e. $t\in (0,+\infty)$. Thus and since
  $v$ is a $p$-gradient flow 
  of $\E$ with respect to strong upper gradient $g$, we can conclude
  by~\eqref{eq:11} that
  \begin{align*}
    -\fdt\mathcal{H}(t)& =
   c\,(\E(v(t)\vert \varphi))^{\alpha-1}\,
    \left(-\fdt\E(v(t))\right)\\
    & = c\,(\E(v(t)\vert \varphi))^{\alpha-1}\,
      g(v(t))^{p^{\mbox{}_{\prime}}}\\
     &\ge \left[\tfrac{1}{c}\right]^{p^{\mbox{}_{\prime}}-1}\,
     (\E(v(t)\vert \varphi))^{\frac{1-\alpha}{p-1}}\\
     &=  \left[\tfrac{1}{\alpha^{\alpha -1}c}\right]^{\frac{p^{\mbox{}_{\prime}}-1}{\alpha}}\,
     \mathcal{H}^{\frac{1-\alpha}{\alpha(p-1)}}(t)
  \end{align*}
  for a.e. $t\in (t_{0},+\infty)$.
  Therefore, for a.e. $t\in (t_{0},+\infty)$, one has
  \begin{align*}
    \fdt\mathcal{H}^{-\frac{1-p\alpha}{\alpha(p-1)}}(t)
    &\ge
      \left[\tfrac{1}{\alpha^{\alpha -1}c}\right]^{\frac{p^{\mbox{}_{\prime}}-1}{\alpha}}\,
      \tfrac{1-p\alpha}{\alpha(p-1)}
    &&\qquad\text{if\quad
      $0<\alpha<\tfrac{1}{p}$}\\
    \fdt\log \mathcal{H}(t)
    &\le - \left[\tfrac{1}{\alpha^{\alpha
      -1}c}\right]^{\frac{p^{\mbox{}_{\prime}}-1}{\alpha}}\,
      = -   \tfrac{1}{pc^{p^{\prime}}}
    &&\qquad\text{if\quad
      $\alpha=\tfrac{1}{p}$}\\
    \fdt\mathcal{H}^{\frac{p\alpha-1}{\alpha(p-1)}}(t)
    &\le
      -  \left[\tfrac{1}{\alpha^{\alpha -1}c}\right]^{\frac{p^{\mbox{}_{\prime}}-1}{\alpha}}\,
      \tfrac{p\alpha-1}{\alpha(p-1)}
   &&\qquad\text{if\quad
      $\tfrac{1}{p}<\alpha\le 1$.}
  \end{align*}
  Integrating these inequalities over $(t_{0},t)$ for any $t>t_{0}$ and
  rearranging the resulting inequalities yields
  \begin{align*}
    \mathcal{H}(t) &\le \left[
     \left[\tfrac{1}{\alpha^{\alpha -1}c}\right]^{\frac{p^{\mbox{}_{\prime}}-1}{\alpha}}\,
      \tfrac{1-p\alpha}{\alpha(p-1)}
                     \,(t-t_{0}) +
    \mathcal{H}^{-\frac{1-p\alpha}{\alpha(p-1)}}(t_{0})\right]^{-\frac{\alpha(p-1)}{1-p\alpha}}
  &&\text{if $0<\alpha<\tfrac{1}{p}$,}\\
    \mathcal{H}(t)
     &\le  \mathcal{H}(t_0) e^{-\tfrac{t}{pc^{p^{\prime}}}}
  && \text{if $\alpha=\tfrac{1}{p}$,}
 \end{align*}
  and in the case $\tfrac{1}{p}<\alpha\le 1$,
 \begin{equation}
   \label{eq:69}
  \mathcal{H}^{\frac{p\alpha-1}{\alpha(p-1)}}(t) \le
    \left[\tfrac{1}{\alpha^{\alpha -1}c}\right]^{\frac{p^{\mbox{}_{\prime}}-1}{\alpha}}\,
      \tfrac{p\alpha-1}{\alpha(p-1)} (t_{0} - t) + \mathcal{H}^{\frac{p\alpha-1}{\alpha(p-1)}}(t_0)
 \end{equation}
for every $t > t_0$. Now, for
\begin{displaymath}
\hat{t} = t_0 + \tfrac{\alpha(p-1)}{\alpha p -1}\,
\alpha^{\frac{\alpha -1}{\alpha (p -1)}}\, c^{\frac{1}{\alpha (p -1)}}(\E(v(t_0)\vert\varphi))^{\frac{\alpha p -1}{p-1}},
\end{displaymath}
if $t=\hat{t}$, then the right-hand side in inequality~\eqref{eq:69}
becomes $0$ and hence $\mathcal{H}(\hat{t})=0$. By~\eqref{eq:50} and
since $\E$ is a strict Lyapunov function of $v$, this implies that
$v(t)\equiv \varphi$ for all $t\ge \hat{t}$. Finally, by \eqref{eq:7},
\eqref{eq:20} and Theorem \ref{thm:convergence}, we have
  \begin{displaymath}
    d(v(t),\varphi)\le
    \int_{t}^{\infty}\vert v'\vert(s)\,\textrm{d}s \le
    \mathcal{H}(t)
  \end{displaymath}
  for every $t\ge t_{0}$. Therefore, the previous three inequalities
  yield the claim of this theorem.
\end{proof}

Recall, by Proposition~\ref{lambdaconvex1}, if $\E$ is a proper,
lower-semicontinuous, $\lambda$-geodesically functional with
$\lambda>0$, then $\E$ satisfies a \L{}ojasiewicz-Simon
inequality~\eqref{eq:38} with exponent $\alpha=\tfrac{1}{2}$ and
constant $c=1/\sqrt{2\lambda}$ at its unique minimizer $\varphi$ (if
it exists). Due to Theorem~\ref{thm:decayrates}, we obtain then the
same exponential convergence result as
in~\cite[Theorem~2.4.14]{AGS-ZH}). Thus, our next corollary highlights
that the approach using the \L{}ojasiewicz-Simon
inequality~\eqref{eq:38} provides the same rate of convergence and
henceforth is consistent with the classical theory.

\begin{corollary}
  \label{Gconv}
  Let $(\mathfrak{M},d)$ be a length space and $\E : \mathfrak{M}\to
  \R\cup\{+\infty\}$ a proper, lower
  semicontinuous functional that is $\lambda$-geodesically convex for
  $\lambda > 0$ and admit a global minimizer $\varphi\in D(\E$).
  Then, every gradient flow $v$ of $\E$ satisfies
  \begin{displaymath}
    d(v(t), \varphi) =
    \mathcal{O}\left(e^{-\lambda t}\right)\qquad\text{as
      $t\to \infty$.}
  \end{displaymath}
\end{corollary}

\begin{remark}
  We note that a similar statement of Corollary~\ref{Gconv} can not hold
  for \emph{geodesically convex} (that is, with $\lambda=0$) functionals $\E$ since the class of
  \emph{convex}, proper and lower semicontinuous functionals on
  Hilbert spaces belong to this case. But in this class, the
  counter-example~\cite{MR496964} by Baillon is known.
\end{remark}

\subsection{Lyapunov stable equilibrium points}
\label{sec:Lyapunov-stability}

In this subsection, our aim is to show that if a functional $\E$ satisfies
a Kurdyka-\L{}ojasiewicz inequality~\eqref{eq:16} in a neighborhood of
an equilibrium point $\varphi\in \mathbb{E}_{g}$ of $\E$, then the \emph{Lyapunov
  stability} of $\varphi$ can be characterized with the
property that $\varphi$ is a local minimum of $\E$. This result
generalizes the main theorem in~\cite{MR2225367} for functionals $\E$
defined on the Euclidean space $\R^{N}$ and satisfying a \L{}ojasiewicz
inequality~\eqref{eq:Loj-Rd}.\medskip

We begin by recalling the notion of Lyapunov stable points of equilibrium.

\begin{definition}
  \label{def:Lyapunovstable}
  For a given proper functional
  $\E : \mathfrak{M}\to (-\infty,+\infty]$ with strong upper gradient
  $g$ a point of equilibrium $\varphi\in \mathbb{E}_{g}$ of $\E$ is called \emph{Lyapunov stable}
  if for every $\varepsilon>0$ there is a
  $\delta=\delta(\varepsilon)>0$ such that for every $v_{0}\in B(\varphi,\delta)\cap
  D(\E)$ and every $p$-gradient flow $v$ of $\E$
  with initial value $v(0+)=v_{0}$, one has
  \begin{equation}
    \label{eq:34}
    v(t)\in B(\varphi,\varepsilon)\qquad\text{for all $t\ge 0$.}
  \end{equation}
\end{definition}

The property that an equilibrium point $\varphi$ of $\E$ is Lyapunov
stable is a \emph{local} property. To characterize such point, we need
the following assumption on the existence of $p$-gradient curves with
initial values in a neighborhood of $\varphi$.

\begin{assumption}[\emph{Existence of $p$-gradient flows for small
    initial values}]
  \label{eq:existence-cond}%
  Suppose, for the proper energy functional
  $\E : \mathfrak{M}\to (-\infty,+\infty]$ with strong upper gradient
  $g$ and given $\varepsilon>0$ and $\varphi\in D(\E)$ the following
  holds:
     \begin{center}
      \textit{for all $v_{0}\in D(\E)\cap B(\varphi,\varepsilon)$, there is a $p$-gradient flow
        $v$ of $\E$ with $v(0+)=v_{0}$.}
    \end{center}
\end{assumption}

The next theorem is the main result of this section.

\begin{theorem}[{\bfseries Lyapunov stability and local Minima}]
  \label{thm:stability}
  Let $\E : \mathfrak{M}\to (-\infty,+\infty]$ be a proper, lower
  semicontinuous functional and $g$ be a with proper strong
  upper gradient of $\E$. Then the following statements hold.
  \begin{enumerate}[topsep=3pt,itemsep=1ex,partopsep=1ex,parsep=1ex]
  \item[(1)] Suppose $g$ is lower semicontinuous and for
    $\varphi\in \mathbb{E}_{g}$, there is a $\varepsilon>0$ such that
    $\E$ is bounded from below on $B(\varphi,\varepsilon)$, $\E$ and
    $g$ satisfy Assumption~\ref{eq:existence-cond}, the set
    \begin{equation}
      \label{eq:14}\tag{H$1^{\ast}$}
      \overline{B}(\varphi,\varepsilon)\cap [\E(\cdot\vert\varphi)\neq
      0]\quad\text{ is contained in }\quad [g>0],
    \end{equation}
    and there is a strictly increasing function
  $\theta\in W^{1,1}_{loc}(\R)$ satisfying $\theta(0)=0$ and
  $\abs{[\theta\neq0,\theta'=0]}=0$, for which $\E$ satisfies a
  Kurdyka-\L{}ojasiewicz inequality~\eqref{eq:16} on
  \begin{equation}
    \label{eq:55}
    \U_{\varepsilon}:=B(\varphi,\varepsilon)\cap[\E(\cdot\vert\varphi)\neq
    0]\cap [\theta'(\E(\cdot\vert\varphi))>0].
  \end{equation}
  Then, if $\varphi$ is Lyapunov stable, $\varphi$ is a local minimum of $\E$.

  \item[(2)] Suppose for $\varphi\in \mathbb{E}_{g}$, there is an
    $\varepsilon>0$ such that
    \begin{equation}
      \label{eq:56}
      \begin{cases}
       \text{for every $\eta>0$, there is a $0<\delta\le \varepsilon$ such that}&\\
       \quad\text{$\E(v\vert\varphi)<\eta$ for all $v\in B(\varphi,\delta)\cap D(\E)$,}&
      \end{cases}
   \end{equation}
   the set $\overline{B}(\varphi,\varepsilon)\cap [\E(\cdot\vert\varphi)\neq
      0]$ satisfies~\eqref{eq:14}, and there is a strictly increasing function
  $\theta\in W^{1,1}_{loc}(\R)$ satisfying $\theta(0)=0$ and
  $\abs{[\theta\neq0,\theta'=0]}=0$, for which $\E$ satisfies a
  Kurdyka-\L{}ojasiewicz inequality~\eqref{eq:16} on
  $\U_{\varepsilon}$.
 Then, if $\varphi$ is a local minimum of $\E$, $\varphi$ is Lyapunov stable.
  \end{enumerate}
\end{theorem}

\begin{remark}
  Concerning Theorem~\ref{thm:stability}, we note the following.
  \begin{enumerate}[topsep=3pt,itemsep=1ex,partopsep=1ex,parsep=1ex]
  \item The hypothesis that $\E$ is bounded from below on
    $B(\varphi,\varepsilon)$ is a necessary condition for $\varphi$
    being a local minimum of $\E$.
  \item If the functional $\E$ is continuous at the equilibrium point
    $\varphi\in \mathbb{E}_{g}$, then $\E$ is necessarily locally bounded and
    satisfies condition~\eqref{eq:56}.
  \end{enumerate}
\end{remark}

\begin{proof}[Proof of Theorem~\ref{thm:stability}]
  We begin by showing statement~(1). To do this, we argue by
  contradiction. Thus, suppose $\varphi$ is not a local minimum of
  $\E$. Then we shall show that there is $\varepsilon>0$ such that for
  every $\delta>0$ there is a
  $v_{\delta}^{0}\in D(\E)\cap B(\varphi,\delta)$ and a $p$-gradient
  flow $v_{\delta}$ of $\E$ with initial value
  $v_{\delta}(0+)=v_{\delta}^{0}$ satisfying
  \begin{equation}
    \label{eq:18}
    v_{\delta}([0,+\infty))\nsubseteq B(\varphi,\varepsilon).
  \end{equation}
  If $\varphi$ is not a local minimum of $\E$ then for every
  $\delta>0$ there is a $v_{\delta}^{0}\in D(\E)\cap B(\varphi, \delta)$ satisfying
  \begin{equation}
    \label{eq:40}
    \E(v_{\delta}^{0}\vert\varphi)<\E(\varphi\vert\varphi)=0.
  \end{equation}
  Now, by hypothesis, there is an $\varepsilon>0$ and a strictly
  increasing function $\theta\in W^{1,1}_{loc}(\R)$ satisfying
  $\theta(0)=0$ and $\abs{[\theta\neq0,\theta'=0]}=0$, for which $\E$
  satisfies a Kurdyka-\L{}ojasiewicz inequality~\eqref{eq:16} on
  $\U_{\varepsilon}$. For every
    $0<\delta<\varepsilon$,
  Assumption~\ref{eq:existence-cond} ensures that there is a
    $p$-gradient flow $v_{\delta}$ of $\E$ with initial value
    $v_{\delta}(0+)=v_{\delta}^{0}\in D(\E)\cap B(\varphi,\delta)$
    satisfying~\eqref{eq:40}. Since $\E$ is a Lyapunov function of
  $v_{\delta}$ (cf~Proposition~\ref{propo:omega-limit-gradientflow}),
  inequality~\eqref{eq:40} implies
  \begin{equation}
    \label{eq:51}
    \E(v_{\delta}(t)\vert\varphi)\le
    \E(v_{\delta}(0+)\vert\varphi)<\E(\varphi\vert\varphi)=0
    \qquad\text{for all $t\ge 0$.}
  \end{equation}
  If we assume that $v_{\delta}$ satisfies~\eqref{eq:34}, then
  by~\eqref{eq:51} and since $\abs{[\theta\neq0,\theta'=0]}=0$, the
  trajectory $v_{\delta}$ satisfies condition~\eqref{eq:30} in
  Theorem~\ref{thm:finite-length} with $0=t_{0}<T=+\infty$, where $\U$
  is replaced by $\U_{\varepsilon}$. Since $\E$ satisfies a
  Kurdyka-\L{}ojasiewicz inequality~\eqref{eq:16} on
  $\U_{\varepsilon}$ and since $\E$ is bounded from below on
  $\U_{\varepsilon}$, we can conclude that $v_{\delta}$ has finite
  length. Since $\mathfrak{M}$ is complete, there is an element
  $\hat{\varphi}\in \overline{B}(\varphi,\varepsilon)$ such that
  $v_{\delta}(t)\to \hat{\varphi}$ in $\mathfrak{M}$ as $t\to
  \infty$. By hypothesis, $\E$ and $g$ are lower semicontinuous, and
  $\E$ is bounded from below on $B(\varphi,\varepsilon)$. Thus,
  Proposition~\ref{propo:omega-limit-gradientflow} yields that
  $g(\hat{\varphi})=0$. But on the other hand, by~\eqref{eq:51},
  $\abs{\E(\hat{\varphi}\vert\varphi)}>0$ and so, by
  hypothesis~\eqref{eq:14}, $g(\hat{\varphi})>0$ which is a
  contradiction to $g(\hat{\varphi})=0$. Thus, our assumption is
  false, and therefore we have shown the existence of a $p$-gradient
  flow $v_{\delta}$ satisfying~\eqref{eq:18}. Since this holds for all
  $0<\delta<\frac{\varepsilon}{7}$, we have thereby proved that
  $\varphi$ is not Lyapunov stable.

Next, we prove statement~(2). For this, suppose $\varphi$ is
  a local minimum of $\E$. Then, there is an $r>0$ such that
   \begin{equation}
    \label{eq:53}
    \E(v\vert\varphi)\ge 0\qquad\text{for all $v\in
      B(\varphi,r)\cap D(\E)$}
  \end{equation}
  By hypothesis, there are $\varepsilon_{0}>0$ and a continuous strictly
  increasing function $\theta\in W^{1,1}_{loc}(\R)$ satisfying
  $\theta(0)=0$ and $\abs{[\theta\neq0,\theta'=0]}=0$, for which $\E$
  satisfies a Kurdyka-\L{}ojasiewicz inequality~\eqref{eq:16} on the
  set $\U_{\varepsilon_{0}}$. Now, let
  $0<\varepsilon\le\varepsilon_{0}$. Then by~\eqref{eq:56}, there is a
  $0<\delta<\min\{\frac{\varepsilon}{2},r\}$ such that
  \begin{equation}
    \label{eq:54}
    \E(v\vert\varphi)<\theta^{-1}(\tfrac{\varepsilon}{2})\qquad\text{ for every
    $v\in B(\varphi,\delta)\cap D(\E)$.}
  \end{equation}
  Now, let $v_{0}\in B(\varphi,\delta)\cap D(\E)$ and
  $v : [0,+\infty)\to \mathfrak{M}$ be a $p$-gradient flow of $\E$ with
  initial value $v(0+)=v_{0}$. Then by the continuity of $v$, there is a
  $0<T\le +\infty$ such that
  \begin{equation}
   \label{eq:92}
   v(t) \in B(\varphi,\varepsilon)\qquad\text{for all $0 \le t < T$.}
  \end{equation}
  Hence, by~\eqref{eq:54}, since $\abs{[\theta\neq0,\theta'=0]}=0$,
  and since $\delta$ and $\varepsilon\le \varepsilon_{0}$, we have
  that $v(t) \in \U_{\varepsilon_{0}}$ for a.e. $0 \le t < T$. Since
  $\E$ satisfies a Kurdyka-\L{}ojasiewicz inequality~\eqref{eq:16} on
  $\U_{\varepsilon}$, we can apply Theorem~\ref{thm:finite-length} to
  conclude that the restriction $v_{T}:=v_{\vert [0,T)}$ of $v$ on
  $[0,T)$ has finite arc-length $\gamma(v_{T})$. Let
  $T=sup\big\{T>0\,\vert\, v([0,T))\subseteq
  B(\varphi,\varepsilon)\big\}$. Then, to complete this proof, it
  remains to show that $T=+\infty$. Thus, assume that $T$ is finite
  and then we shall arrive to a contradiction. By
  Lemma~\ref{lem:arclengthreparametrisation}, we can parametrize the
  curve $v_{T}$ by its arc-length $\gamma(v_{T})$ on $(0,T)$. Let
  $\hat{v}_{T} : [0,\gamma(v_{T})]\to \mathfrak{M}$ be this
  reparametrization of $v_{T}$ by its arc-length. Then, $\hat{v}_{T}$
  is a $p$-gradient flow of $\E$ with metric derivative
  $\abs{\hat{v}'}=1$ a.e. on $(0,\gamma(v_{T}))$
  satisfying~\eqref{eq:11} and $\hat{v}_{T}(t)\in \U_{\varepsilon}$
  for all $t\in [0, \gamma(v_{T}))$. Moreover, $\E$ satisfies a
  Kurdyka-\L{}ojasiewicz inequality~\eqref{eq:16}. Thus, we can conclude
  that
\begin{align*}
  \fds\theta(\E(\hat{v}(s)\vert\varphi))
  &=\theta'(\E(\hat{v}(s)\vert\varphi))\,
    \fds\E(\hat{v}(s)\vert\varphi)\\
  & =\theta'(\E(\hat{v}(s)\vert\varphi))\,
     (-g(\hat{v}(s)))\,\abs{v'}(s)\\
  & = -\theta'(\E(\hat{v}(s)\vert\varphi))\,g(\hat{v}(s))\\
  &\le - 1
\end{align*}
for a.e. $t\in (0,\gamma(v_{T}))$. Integrating this inequality from
$(0, \gamma(v_{T}))$ leads to
\begin{displaymath}
  \left(\theta(\E(\hat{v}(\gamma_{T})\vert\varphi))-
    \theta(\E(v_{0}\vert\varphi))\right)\le -\gamma(v_{T}).
\end{displaymath}
Rearranging this inequality, then applying~\eqref{eq:53}
and~\eqref{eq:54} and using the monotonicity of $\theta$ shows that the
length $\gamma(v_{T})$ of $v$ on $[0,T)$ satisfies
\begin{displaymath}
  \gamma(v_{T})\le \theta(\E(v_{0}\vert\varphi))<\frac{\varepsilon}{2}.
\end{displaymath}
Therefore,
\begin{displaymath}
  d(v(t),\varphi)\le d(v(t),v_{0})+d(v_{0},\varphi)\le
 \gamma(v_{T})+\frac{\varepsilon}{2}<\varepsilon.
\end{displaymath}
Thus, by the continuity of $v$, and since we have assumed that $T$ is
finite, there is a $T'>T$ such that $v$ satisfies~\eqref{eq:92} for
$T$ replaced by $T'$.  But this is a contradiction to the fact that
$T$ is maximal such that~\eqref{eq:92} holds, implying that our
assumption is false. Therefore, $v$ satisfies~\eqref{eq:92} with
$T=+\infty$. Since $0<\varepsilon\le\varepsilon_{0}$ and the
$p$-gradient flow $v$ with initial value
$v(0+)=v_{0}\in B(\varphi,\delta)\cap D(\E)$ were arbitrary, we have
thereby shown that the local minimizer $\varphi$ of $\E$ is Lyapunov
stable.
\end{proof}

\subsection{Entropy-transportation inequality and K\L{} inequality}
\label{sec:character-global-KLandET}

In this last part of Section~\ref{sec:KL-inequality}, we show that a
generalized \emph{entropy-transportation
    inequality~\eqref{eq:13-transport0}} is equivalent to
  Kurdyka-\L{}ojasiewicz inequality~\eqref{eq:16}. We
  investigate this in two cases, namely, when $\E$
  satisfies these inequalities \emph{locally} and
  \emph{globally}.\medskip

We begin by introducing the notion of local and global
\emph{entropy-transportation inequality} (cf~\cite{VillaniCompMath2004}).

\begin{definition}
  \label{def:ETinequality}
  A proper functional $\E : \mathfrak{M}\to (-\infty,+\infty]$ with
  strong upper gradient $g$ is said to satisfy \emph{locally a
    generalized entropy-transportation (ET-) inequality at a point of
    equilibrium $\varphi\in \mathbb{E}_{g}$} if there are
  $\varepsilon>0$ and a strictly increasing function $\Psi\in C(\R)$
  satisfying $\Psi(0)=0$ and
  \begin{equation}
    \label{eq:ET-local}
    \inf_{\hat{\varphi}\in \mathbb{E}_{g}\cap
      B(\varphi,\varepsilon)}d(v,\hat{\varphi})\le \Psi(\E(v\vert\varphi))
  \end{equation}
  for every $v\in B(\varphi,\varepsilon)\cap D(\E)$. Further, a
  functional $\E$ is said to satisfy \emph{globally a generalized
    entropy-transportation inequality at $\varphi\in \mathbb{E}_{g}$}
  if $\E$ satisfies
   \begin{equation}
    \label{eq:ET-global}
    \inf_{\hat{\varphi}\in \mathbb{E}_{g}}d(v,\hat{\varphi})\le
    \Psi(\E(v\vert\varphi)) \quad\text{for every $v \in D(\E)$.}
  \end{equation}
  \end{definition}

\begin{remark}[\emph{isolated equilibrium points}]
 If a functional $\E$ with strong upper gradient $g$ admits
  a point of equilibrium $\varphi\in \mathbb{E}_{g}$ satisfying
  \begin{displaymath}
  \mathbb{E}_{g}\cap B(\varphi,\varepsilon)=\{\varphi\}
 \end{displaymath}
 for some $\varepsilon>0$, then the generalized entropy-transportation
 inequality~\eqref{eq:ET-local} reduces to
 inequality~\eqref{eq:13-transport0}. This is, for instance, the
 case when $\E$ is $\lambda$-geodesically convex with $\lambda>0$ (see
 Proposition~\ref{propo:charcter-local-min}).
\end{remark}

The following theorem is our first main result of this subsection.

\begin{theorem}[{\bfseries Equivalence between local K\L{}- \& ET-inequality}]
  \label{thm:charact-local-KLET}
 Let $\E : \mathfrak{M} \rightarrow (-\infty, + \infty]$ be a proper,
 lower semicontinuous functional with strong upper gradient $g$. 
Then, the following statements hold.
  \begin{enumerate}[topsep=3pt,itemsep=1ex,partopsep=1ex,parsep=1ex]
  \item\label{thm:local-KL} ({\bfseries K\L{}-inequality implies
      ET-inequality}) Let $g$ be lower semicontinuous and the
    equilibrium point $\varphi\in \mathbb{E}_{g}$ be Lyapunov stable.
    Suppose, there is an $\varepsilon>0$ such that $\E$ is bounded
    from below on $B(\varphi,\varepsilon)$, $\E$ and $g$ satisfy
    Assumption~\ref{eq:existence-cond}, and 
    $\overline{B}(\varphi,\varepsilon)\cap [\E(\cdot\vert\varphi)\neq
    0]$ satisfies~\eqref{eq:14}.

   If there is a strictly increasing function
    $\theta\in W^{1,1}_{loc}(\R)$ satisfying $\theta(0)=0$ and
    $\abs{[\theta\neq0,\theta'=0]}=0$ and such that $\E$ satisfies a
    Kurdyka-\L{}ojasiewicz inequality~\eqref{eq:16} on the set
    $\U_{\varepsilon}$ given by~\eqref{eq:55}, then $\E$
    satisfies locally a generalized entropy-transpor\-tation
    inequality~\eqref{eq:ET-local} at~$\varphi$.

  \item\label{thm:local-ET} ({\bfseries ET-inequality implies
      K\L{}-inequality}) Suppose $\varphi\in \mathbb{E}_{g}$ is a local
    minimum of $\E$ and there is an $\varepsilon>0$ such that $\E$ and
    $g$ satisfy
   \begin{equation}
       \label{eq:57}
       \E(v\vert\varphi)\le g(v)\,d(v,\varphi)\qquad\text{for all
       $v \in B(\varphi,\varepsilon)\cap D(\E)$.}
   \end{equation}

   If there is a strictly increasing function $\Psi\in C(\R)$
   satisfying $\Psi(0)=0$ and such that $s\mapsto \Psi(s)/s$ belongs
   to $L^{1}_{loc}(\R)$, and $\E$ satisfies a generalized
   entropy-transpor\-tation inequality~\eqref{eq:ET-local} on
   $B(\varphi,\varepsilon)\cap D(\E)$, then $\E$ satisfies a
   Kurdyka-\L{}ojasiewicz inequality~\eqref{eq:16} on
   $\U\cap [\E(\cdot\vert\varphi)>0]$.
  \end{enumerate}
\end{theorem}

\begin{remark}
  \label{rem:geod-growthcond}
 It is worth noting that every proper, $\lambda$-geodesically convex functional
  $\E$ with $\lambda\ge 0$ satisfies condition~\eqref{eq:57} with $g=\abs{D^{-}\E}$
  the descending slope of $\E$ (see~Proposition~\ref{propo:lconvex-slope}).
\end{remark}

\begin{remark}[\emph{The role of the set $D(\E)$ in
    Theorem~\ref{thm:charact-local-KLET} and ``invariant sets''}]\label{rem:setD}
  We note that Assumption~\ref{eq:existence-cond} in
  Theorem~\ref{thm:charact-local-KLET} is only needed to establish
  statement~\eqref{thm:local-KL}. Further, the proof of
    statement~\eqref{thm:local-KL} in
    Theorem~\ref{thm:charact-local-KLET} shows, that the set $D(\E)$
    in the two inequalities~\eqref{eq:16} and~\eqref{eq:ET-local}
    could be replaced by every subset $\mathcal{D}\subseteq D(\E)$
    which the flow map $S : [0,+\infty)\times D(\E)\to 2^{D(\E)}$
    defined by~\eqref{eq:66} leaves \emph{invariant}, that is,
    \begin{equation}
      \label{eq:133}
      \begin{cases}
        \text{for every $v_{0}\in \mathcal{D}$ and every $p$-gradient
          flow $v$ of $\E$ with}&\\
         \text{initial value $v(0+)=v_{0}$, one has $v(t)\in \mathcal{D}$ for all $t\ge 0$.}&
      \end{cases}
    \end{equation}
\end{remark}

After these remarks, we turn now to the proof.

\begin{proof}[Proof of Theorem~\ref{thm:charact-local-KLET}]
 We begin by showing statement~\eqref{thm:local-KL}. By
  hypothesis, for $\varphi\in \mathbb{E}_{g}$, there are
  $\varepsilon>0$ and a strictly increasing function
  $\theta\in W^{1,1}_{loc}(\R)$ such that $\E$ satisfies a
  Kurdyka-\L{}ojasiewicz inequality~\eqref{eq:16} on the set
  $\U_{\varepsilon}$ given by~\eqref{eq:55}. Now, let
  $v\in D(\E)\cap B(\varphi,\varepsilon)$. By
  Assumption~\ref{eq:existence-cond} and since $\varphi$ is Lyapunov
  stable, there is a $p$-gradient flow
  $\hat{v} : [0,+\infty)\to \mathfrak{M}$ of $\E$ with initial value
  $\hat{v}(0+)=v$ satisfying~\eqref{eq:34} and by Theorem~\ref{thm:stability},
  $\varphi$ is a local minimum of $\E$. Thus,
  $\E(\hat{v}(t)\vert\varphi)\ge 0$ for all $t\ge 0$.

  We set $T:=\sup\{t\ge 0\,\vert\,
  \E(\hat{v}(t)\vert\varphi)>0\}$. Then, we need to consider three
  cases. First, suppose $T=+\infty$. Then,
  $\E(\hat{v}(t)\vert\varphi)>0$ for all $t\ge 0$. Since
  $\abs{[\theta\neq0,\theta'=0]}=0$ and the set
  $\overline{B}(\varphi,\varepsilon)\cap [\E(\cdot\vert\varphi)>0]$
  satisfies~\eqref{eq:14}, it follows that
  $\hat{v}(t)\in \U_{\varepsilon}$ for almost every $t\ge 0$. Since
  $\E$ satisfies a Kurdyka-\L{}ojasiewicz inequality~\eqref{eq:16} on
  the set $\U_{\varepsilon}$ and since $\E$ is bounded from below on
  $B(\varphi,\varepsilon)$, Theorem~\ref{thm:finite-length} yields
  that $\hat{v}$ has finite length~\eqref{eq:52}. By the completeness
  of $\mathfrak{M}$ and by
  statement~\eqref{propo:omega-limit-gradientflow-claim-5} of
  Proposition~\ref{propo:omega-limit-gradientflow}, there is a
  $\hat{\varphi}\in \mathbb{E}_{g}\cap
  \overline{B}(\varphi,\varepsilon)$ such that
  $\hat{v}(t)\to \hat{\varphi}$ in $\mathfrak{M}$ as $t\to
  +\infty$. Moreover, the functions $\theta$ and $\hat{v}$ satisfy
  inequality~\eqref{eq:26} in Theorem~\ref{thm:finite-length}. Thus,
  \begin{equation}
    \label{eq:93}
    d(v,\hat{v}(t))\le \int_{0}^{t}\abs{\hat{v}'}(s)\,\ds\le
    \theta(\E(v\vert \varphi))-\theta(\E(\hat{v}(t)\vert \varphi))
  \end{equation}
  for all $t\ge 0$. Since
  $\hat{\varphi}\in \mathbb{E}_{g}\cap
  \overline{B}(\varphi,\varepsilon)$, hypothesis~\eqref{eq:14} implies
  that $\E(\hat{\varphi}\vert \varphi)=0$. Thus, sending $t\to+\infty$
  in~\eqref{eq:93} yields
  \begin{equation}
    \label{eq:45bis}
    d(v,\hat{\varphi})\le  \theta(\E(v\vert \varphi)),
  \end{equation}
  and taking the infimum of $d(v,\cdot)$ over all equilibrium points
  $\tilde{\varphi}\in \mathbb{\E}_{g}\cap B(\varphi,\varepsilon)$ of
  $\E$ on the left-hand side of~\eqref{eq:45bis} gives
  \begin{displaymath}
        \inf_{\tilde{\varphi}\in \mathbb{\E}_{g}\cap B(\varphi,\varepsilon)}
          d(v,\tilde{\varphi})\le
        d(v,\hat{\varphi})\le
        \theta(\E(v\vert\varphi)),
  \end{displaymath}
  which for $\Psi=\theta$, is an entropy-transpor\-tation
    inequality~\eqref{eq:ET-local} at~$\varphi$.

    Next, suppose $0<T<\infty$. Since $\E$ is a Lyapunov function of
    $\hat{v}$ (see~Proposition~\ref{propo:omega-limit-gradientflow}),
    $\hat{v}(t)=\hat{v}(T)$ for all $t\ge T$ and since $g$ is lower
    semicontinuous on $\mathfrak{M}$,
    statement~\eqref{propo:omega-limit-gradientflow-claim-5} of
    Proposition~\ref{propo:omega-limit-gradientflow} implies that
    $\hat{v}(T)\in \mathbb{E}_{g}\cap B(\varphi,\varepsilon)$ with
    $\E(\hat{v}(T))=\E(\varphi)$. Since
    $\hat{v}(t)\in \U_{\varepsilon}$ for almost every $t\in (0,T)$ and
    since $\E$ satisfies a Kurdyka-\L{}ojasiewicz
    inequality~\eqref{eq:16} on the set $\U_{\varepsilon}$, it follows
    from Theorem~\ref{thm:finite-length} that $\hat{v}$ satisfies
    inequality~\eqref{eq:93} for $t=T$ and so,
  \begin{displaymath}
    \inf_{\tilde{\varphi}\in \mathbb{\E}_{g}\cap B(\varphi,\varepsilon)}
          d(v,\tilde{\varphi})\le d(v, \hat{v}(T))\le
          \theta(\E(v\vert \varphi)).
  \end{displaymath}
  In the case $T=0$, the initial value
  $v \in \mathbb{E}_{g}\cap B(\varphi,\varepsilon)$ satisfies
  $\E(v\vert\varphi)=0$.  Since $\theta(0)=0$ and $d(v,v)=0$, a
  generalized entropy-transpor\-tation inequality~\eqref{eq:ET-local}
  at~$\varphi$ with $\Psi:=\theta$, trivially holds. Therefore and
  since $v\in D(\E)\cap B(\varphi,\varepsilon)$ was arbitrary, the
  proof of statement~(1) is complete.

  Next, we intend to show statement~\eqref{thm:local-ET}. To see this,
  we note that by condition~\eqref{eq:57} and since $\varphi$ is a
  local minimum,
  \begin{displaymath}
    \E(\tilde{\varphi})=\E(\varphi)\qquad\text{for all
      $\tilde{\varphi}\in \mathbb{E}_{g}\cap B(\varphi,\varepsilon)$,}
  \end{displaymath}
  or equivalently, for the set $B(\varphi,\varepsilon)$, $\E$ and $g$
  satisfy condition~\eqref{eq:15} in Section~\ref{sec:sufKL}. By using
  again that $\varphi$ is a local minimum, one sees that the two sets
  $B(\varphi,\varepsilon)\cap [\E(\cdot\vert\varphi)\neq0]$ and
  $B(\varphi,\varepsilon)\cap [\E(\cdot\vert\varphi)>0]$
  coincide. Thus, and by~\eqref{eq:57},
  \begin{displaymath}
    \frac{\E(v\vert \varphi)}{g(v)}\le
    \inf_{\tilde{\varphi}\in \mathbb{E}_{g}\cap
      B(\varphi,\varepsilon)}d(v,\tilde{\varphi})
  \end{displaymath}
  for every $v\in B(\varphi,\varepsilon)\cap [\E(\cdot\vert\varphi)>0]$, or equivalently,
  \begin{equation}
    \label{eq:94}
    \E(v\vert \varphi)\le g(v)\, \inf_{\tilde{\varphi}\in \mathbb{E}_{g}\cap
      B(\varphi,\varepsilon)}d(v,\tilde{\varphi}).
  \end{equation}
  By hypothesis, there is a strictly increasing function
  $\Psi\in C(\R)$ satisfying $\Psi(0)=0$ such that $\E$ satisfies a
  generalized entropy-transpor\-tation inequality~\eqref{eq:ET-local}
  on $B(\varphi,\varepsilon)\cap D(\E)$. Combining~\eqref{eq:ET-local}
  with~\eqref{eq:94}, one finds that
  \begin{equation}
    \label{eq:95}
    \E(v\vert \varphi)\le g(v)\, \inf_{\tilde{\varphi}\in
      \mathbb{E}_{g}\cap B(\varphi,\varepsilon)} d(v,\tilde{\varphi})\le
    g(v) \Psi(\E(v\vert \varphi))
  \end{equation}
  for every
  $v\in B(\varphi,\varepsilon) \cap [\E(\cdot\vert\varphi)> 0]$. By
  hypothesis, $\Psi$ is a continuous, strictly increasing function on
  $\R$ satisfying $\Psi(0)=0$, and $s\mapsto \Psi(s)/s$ belongs to
  $L^{1}_{loc}(\R)$. Thus, if $\theta : \R\to (0,+\infty)$ is defined by
  \begin{displaymath}
    \theta(s)=\int_{0}^{s}\frac{\Psi(r)}{r} \,\dr \qquad \text{for every $s\in \R$,}
  \end{displaymath}
  then $\theta\in W^{1,1}_{loc}(\R)$, $\theta$ is strictly increasing
  and satisfies $\theta(0)=0$ and $\theta'(s)=\frac{\Psi(s)}{s} >0$ for all $s\in
  \R\setminus\{0\}$. Note, that this implies that
  $\abs{\{\theta\neq 0,\theta'=0\}}=0$. Moreover, in terms of the function
  $\theta$, inequality~\eqref{eq:95} can be rewritten as
\begin{displaymath}
  1\le g(v)\, \frac{\Psi(\E(v\vert \varphi))}{\E(v\vert
    \varphi)}
  =g(v)\,\theta'(\E(v\vert \varphi))
\end{displaymath}
for every $v\in B(\varphi,\varepsilon) \cap [\E(\cdot\vert\varphi)>
0]$, proving that $\E$ satisfies a Kurdyka-\L{}ojasiewicz
inequality~\eqref{eq:16} on $\U=B(\varphi,\varepsilon) \cap [\E(\cdot\vert\varphi)>
0]$. This completes the proof of this theorem.
\end{proof}

To show that a \emph{global}
Kurdyka-\L{}ojasiewicz inequality~\eqref{eq:16} implies a
\emph{global} entropy-transportation
inequality~\eqref{eq:13-transport0}, we need the
following assumption instead of Assumption~\ref{eq:existence-cond}.

\begin{assumption}[\emph{Existence of $p$-gradient flows}]
  \label{eq:existence-cond-bis}%
  Suppose, for the proper energy functional
  $\E : \mathfrak{M}\to (-\infty,+\infty]$ with strong upper gradient
  $g$ holds:
     \begin{center}
      \textit{for all $v_{0}\in D(\E)$, there is a $p$-gradient flow
        $v$ of $\E$ with $v(0+)=v_{0}$.}
    \end{center}
\end{assumption}

We recall that by Fermat's rule
(Proposition~\ref{propo:charcter-local-min}), for
$\lambda$-geodesically convex functionals $\E$ with $\lambda\ge 0$,
every equilibrium point $\varphi\in \mathbb{E}_{\abs{D^{-}\E}}$ of
$\E$ is a \emph{global} minimum of $\E$. Thus, by following the idea
of the proof of Theorem~\ref{thm:charact-local-KLET} and using
statement~\eqref{propo:lconvex-slope-claim3} of
Proposition~\ref{propo:lconvex-slope} together with
Remark~\ref{rem:geod-growthcond}, one sees that the following result
on the equivalence of \emph{global} Kurdyka-\L{}ojasiewicz
inequality~\eqref{eq:16} and \emph{global} entropy-transportation
inequality~\eqref{eq:13-transport0} holds. We omit the proof of
Theorem~\ref{thm:charact-global} since it is repetitive to the
previous one.

\begin{theorem}[{\bfseries Equivalence between global K\L{}- \& ET-inequality}]
  \label{thm:charact-global}
  For $\lambda\ge 0$, let $\E : \mathfrak{M} \rightarrow (-\infty, + \infty]$ be a proper,
  lower semicontinuous, $\lambda$-geodesically convex functional on a
  length space $(\mathfrak{M},d)$. Suppose, $\E$
  and the descending slope $\abs{D^{-}\E}$ satisfy
  Assumption~\ref{eq:existence-cond-bis} and for
  $\varphi\in \mathbb{E}_{\abs{D^{-}\E}}$, the set
  $[\E(\cdot\vert\varphi)>0]$ satisfies
  hypothesis~\eqref{eq:14}. Then, the following statements are
  equivalent.
  \begin{enumerate}[topsep=3pt,itemsep=1ex,partopsep=1ex,parsep=1ex]
  \item\label{thm:local-KLbis} ({\bfseries K\L{}-inequality})
   There is a strictly increasing function
      $\theta\in W^{1,1}_{loc}(\R)$ satisfying $\theta(0)=0$ and
      $\abs{[\theta\neq0,\theta'=0]}=0$, and $\E$ satisfies a
      Kurdyka-\L{}ojasiewicz inequality~\eqref{eq:16} on
      $\U:=[\E(\cdot\vert\varphi)>0]\cap
      [\theta'(\E(\cdot\vert\varphi))>0]$.

\item\label{thm:local-ETbis} ({\bfseries ET-inequality})There is a strictly
    increasing function $\Psi\in C(\R)$ satisfying $\Psi(0)=0$ and
    $s\mapsto \Psi(s)/s$ belongs to $L^{1}_{loc}(\R)$ such that $\E$
    satisfies the generalized entropy-transpor\-tation inequality
   \begin{equation}
        \label{eq:13-transport}
        \inf_{\tilde{\varphi}\in
          \textrm{argmin}(\E)}d(v,\tilde{\varphi})\le
        \Psi(\E(v\vert\varphi))
        \qquad\text{for all $v\in D(\E)$.}
   \end{equation}
 \end{enumerate}
\end{theorem}

  Under the hypotheses
  of Theorem~\ref{thm:charact-global}, if $\E$ satisfies a global
  \L{}ojasiewicz-Simon inequality~\eqref{eq:38} with exponent
  $\alpha\in (0,1]$ at $\varphi\in \mathbb{E}_{\abs{D^{-}\E}}$, then
  the proof of statement~\eqref{thm:local-KL} of
  Theorem~\ref{thm:charact-local-KLET} shows that $\E$ satisfies the
  generalized entropy-transpor\-tation
  inequality~\eqref{eq:13-transport} for the function $\psi =\theta$
  given by~\eqref{eq:38bis}. Conversely, if $\E$ satisfies the the
  generalized entropy-transpor\-tation
  inequality~\eqref{eq:13-transport} for
  $\psi(s) = \frac{c}{\alpha}\abs{s}^{\alpha-1}s$, then the proof of
  statement~\eqref{thm:local-ET} of
  Theorem~\ref{thm:charact-local-KLET} yields that $\E$ satisfies the
  global \L{}ojasiewicz-Simon inequality~\eqref{eq:38} at $\varphi$
  for the function
  $\theta(s) = \frac{C}{\alpha^2}\abs{s}^{\alpha-1}s$. Summarizing, we
  state the following result.

\begin{corollary}[{\bfseries Equivalence between global \L{}S- \& ET-inequality}]
  \label{equivalent1}
  For $\lambda\ge 0$, let
  $\E : \mathfrak{M} \rightarrow (-\infty, + \infty]$ be a proper,
  lower semicontinuous, $\lambda$-geodesically convex functional on a
  length space $(\mathfrak{M},d)$. Suppose, $\E$
  and the descending slope $\abs{D^{-}\E}$ satisfy
  Assumption~\ref{eq:existence-cond-bis} and for
  $\varphi\in \mathbb{E}_{\abs{D^{-}\E}}$, the set
  $[\E(\cdot\vert\varphi)>0]$ satisfies
  hypothesis~\eqref{eq:14}. Then, for $\alpha\in (0,1]$, the following
  statements hold.

    \begin{enumerate}[topsep=3pt,itemsep=1ex,partopsep=1ex,parsep=1ex]
       \setlength{\itemsep}{7pt}
     \item ({\bfseries \L{}S-inequality implies
    ET-inequality}) If there is a $c>0$ such that $\E$ satisfies
     \begin{equation}
       \label{E1equivalent1}
        (\E(v\vert \varphi))^{1-\alpha} \le c \abs{D^- \E}(v)\qquad
        \text{for all $v\in D(\E)$}
     \end{equation}
     then $\E$ satisfies
     \begin{equation}
         \label{E2equivalent1}
          \inf_{\tilde{\varphi}\in
          \textrm{argmin}(\E)}d(v,\tilde{\varphi})
        \le \frac{c}{\alpha}(\E(v\vert \varphi))^{\alpha}
           \qquad\text{for all $v\in D(\E)$.}
     \end{equation}

     \item ({\bfseries ET-inequality implies
    \L{}S-inequality}) If there is a $c>0$ such that $\E$
     satisfies~\eqref{E2equivalent1},
       then $\E$ satisfies
     \begin{equation}
       \label{E1equivalent1alpha}
        (\E(v\vert \varphi))^{1-\alpha} \le \frac{c}{\alpha} \abs{D^- \E}(v) \qquad
        \text{for all $v\in D(\E)$.}
     \end{equation}
    \end{enumerate}
  \end{corollary}

\begin{remark}
  For Theorem~\ref{thm:charact-global} and
  Corollary~\ref{equivalent1}, the same comments hold as stated
  in Remark~\ref{rem:setD}.
\end{remark}

\begin{remark}[\emph{The classical case $\lambda>0$}]
  For $\lambda>0$, let $\E : \mathfrak{M}\to (-\infty,\infty]$ be a proper,
  lower semicontinuous, $\lambda$-geodesically convex functional on a
  length space $(\mathfrak{M}, d)$. Then by
  Proposition~\ref{lambdaconvex1}, if $\E$ admits a (unique) minimizer
  $\varphi\in D(\E)$ of $\E$, then~\eqref{messi1entropy}
  shows that the \L{}ojasiewicz-Simon inequality~\eqref{eq:38} with
  exponent $\alpha=\frac{1}{2}$ holds and immediately induces that $\E$
  satisfies an entropy-transpor\-tation
  inequality~\eqref{eq:13-transport0} with
  $\Psi(s)=\frac{2}{\lambda}\abs{s}^{-1/2}s$. On the other hand,
  Corollary~\ref{equivalent1} yields for $\lambda>0$ that
  entropy-transpor\-tation inequality~\eqref{eq:13-transport0} for
  $\Psi(s)=\frac{2}{\lambda}\abs{s}^{-1/2}s$ implies that $\E$
  satisfies a \L{}ojasiewicz-Simon inequality~\eqref{eq:38} with
  exponent $\alpha=\frac{1}{2}$.
\end{remark}

%
%
%
%
%
%

\section{Applications}
\label{sec:application}

\subsection{The classical Banach and Hilbert space case}
\label{sec:Hilbertspacesframework}

We begin by considering the metric derivative of Banach space-valued
curves.\medskip

Suppose that $\mathfrak{M}=X$ is a reflexive Banach space. Then, a
curve $v : (0,+\infty)\to \mathfrak{M}$ belongs to the class
$AC^{p}_{loc}(0,+\infty;X)$ if and only if $v$ is differentiable at
a.e. $t\in (0,+\infty)$, $v'\in L^{p}_{loc}(0,+\infty;X)$ and
\begin{displaymath}
  v(t)-v(s)=\int_{s}^{t}v'(r)\,dr\qquad\text{for all $0<s\le t<+\infty$.}
\end{displaymath}
In particular, the \emph{metric derivative} $\abs{v'}$ of $v$ is given by
\begin{displaymath}
  \abs{v'}(t)=\norm{v'(t)}_{X}\qquad\text{for a.e. $t\in (0,+\infty)$.}
\end{displaymath}

\begin{remark}[\emph{$\lambda$-convex functions on $H$}]
  \label{rem:geodesicconvexity-in-H}
  Due to the identity
   \begin{displaymath}
      t\,\norm{v_{1}}^{2}_{H}+(1-t)\norm{v_{0}}_{H}^{2}
      -t\,(1-t)\norm{v_{1}-v_{0}}_{H}^{2}=\norm{t v_{1}+(1-t)v_{0}}_{H}^{2}
    \end{displaymath}
    holding for every element $v_{0}$, $v_{1}$ of an Hilbert space $H$
    and $t\in [0,1]$, a functional $\E : H\to (-\infty,\infty]$ is
    $\lambda$-convex for some $\lambda\in \R$ along every \emph{line
      segment} $\gamma=\overline{v_{0}v_{1}} \subseteq H$ if and only
    if
    $v\mapsto
    \E_{-\lambda}(v):=\E(v)-\frac{\lambda}{2}\norm{v}_{H}^{2}$ is
    convex on $H$, or equivalently (following the notion
    in~\cite{MR3465809}), $\E$ is \emph{semi-convex} on $H$ for
    $\omega=-\lambda$.
\end{remark}

It is important to see that the theory in the \emph{non-smooth}
framework is consistent with the \emph{smooth} one. We begin with the
smooth setting. For this, we recall that a functional
$\E : \mathcal{U} \rightarrow \R$ is Fr\'echet differentiable on an
open set $\mathcal{U}\subseteq X$ if for every $v\in \mathcal{U}$,
there is a (unique) element $T$ of the dual space $X'$ of $X$ such
that
\begin{displaymath}
  \E(v+h)=\E(v)+T(h)+\mathit{o}(h)\qquad\text{for $h\to 0$ in $X$.}
\end{displaymath}
Then, one sets $\E'(v)=T$ and calls the mapping $\E' : \mathcal{U}\to
X'$ the Fr\'echet differential of $\E$.
It is outlined in the book~\cite{AGS-ZH} that a mapping $g : X\to
[0,+\infty]$ is an \emph{upper gradient of $\E$} if and only if
\begin{displaymath}
g(v) \ge \norm{\E'(v)}_{X'}\qquad \text{for all $v\in \mathcal{U}$,}
\end{displaymath}
where $\norm{\cdot}_{X'}$ denotes the norm of the dual space $X'$.
\medskip

Next, we revisit an interesting example of a smooth energy functional
$\E$ given in~\cite[Proposition~1.1]{MR2019030} (see
also~\cite[Corollary~3.13]{MR1986700}, \cite{chill-fasangova:isem}) to
demonstrate that the hypotheses on the talweg curve $x$ in
Theorem~\ref{thm:3talweg-kl} provides an \emph{optimal} function
$\theta$ in Kurdyka-\L{}ojasiewicz inequality~\eqref{eq:16}. To
measure the optimality of $\theta$, we focus on the class of functions
$\theta$ by~\eqref{eq:38bis} for some exponent $\alpha\in (0,1]$.

\begin{examples}
  \label{ex:lojforE}
  Let $\E : \mathcal{U}\to \R$ be a twice continuously differentiable
  function on an open neighborhood $\mathcal{U}\subseteq X$. Suppose
  $\varphi\in \mathcal{U}$ is a \emph{local minimum} of $\E$ and
  $\E''(\varphi) : X \to X'$ is \emph{invertible}. Then,
  by~\cite[Proposition~1.1]{MR2019030}, there is an $R_{L}>0$ such that
  $\E$ satisfies \L{}ojasiewicz-Simon inequality~\eqref{eq:38} on
  $B(\varphi,R_{L})$.

  Now, our aim is to construct a talweg $x : (0,\delta]\to X$ through the
  $C$-valley $\mathcal{V}_{C,\mathcal{D}}(\varphi)$ (for some $C>1$,
  $\delta>0$ and $\mathcal{D}\subseteq B(\varphi,R_{L})$) which satisfies the hypotheses of
  Theorem~\ref{thm:3talweg-kl} such that $\theta:=h^{-1}$ for
  $h(\cdot):=\E(x(\cdot)\vert\varphi)$ does not grow faster or slower as
  $s\mapsto\abs{s}^{1-\alpha}s$ with exponent
  $\alpha=\frac{1}{2}$.

  By assumption, there is an $r_{0}>0$ such that
    \begin{equation}
      \label{eq:5}
      \E(v\vert\varphi)\ge 0\qquad\text{for all $v\in B(\varphi,r_{0})$.}
    \end{equation}
 Applying Taylor's theorem gives that 
  \begin{equation}
    \label{eq:61}
    \E'(v)=\E''(\varphi)(v-\varphi) +\mathit{o}(\norm{v-\varphi}_{X})
  \end{equation}
  for every $v\in B(\varphi,r_{1})$, for some $0<r_{1}\le r_{0}$. By possibly
  choosing $r_{1}\in (0,r_{0}]$ a bit smaller,
  we can conclude from~\eqref{eq:61} that
  \begin{equation}
    \label{eq:24}
    \norm{\E'(v)}_{X'}\le C_{1}\,\norm{v-\varphi}_{X}
  \end{equation}
  for every $v\in B(\varphi,r_{1})$, where
  $C_{1}=(\norm{\E''(\varphi)}_{X''}+1)$. Since $\E''(\varphi)$ is
  invertible, there is a $C_{2}>0$ such that
  \begin{displaymath}
    \norm{v-\varphi}_{X}\le C_{2} \norm{\E''(\varphi)(v-\varphi)}_{X'}
  \end{displaymath}
  for all $v\in X$. Combining this inequality with~\eqref{eq:61},
  and by possibly choosing $r_{1}\in (0,r_{0}]$ again smaller, we get
  \begin{equation}
    \label{eq:58}
    \norm{\E'(v)}_{X'}\ge
     \frac{1}{2C_{2}} \norm{v-\varphi}_{X}
  \end{equation}
  for all $v\in B(\varphi,r_{1})$. Thus, functional $\E$
  satisfies Assumption~\ref{ass:H1} on $B(\varphi,r_{1})$. Further, Taylor's
  expansion gives
  \begin{equation}
    \label{eq:64}
    \E(v\vert \varphi)=\tfrac{1}{2}\langle
    \E''(\varphi)(v-\varphi),v-\varphi\rangle_{X',X}
    +\mathit{o}(\norm{v-\varphi}_{X}^{2})
  \end{equation}
 for every $v\to \varphi$. Thus, by possibly replacing $r_{1}$ by a
 smaller $\hat{r}_{1}>0$, we see that
  \begin{equation}
    \label{eq:63}
    \abs{\E(v\vert \varphi)}^{1/2}\le C_{3}\,\norm{v-\varphi}_{X}
  \end{equation}
  for every $v\in B(\varphi,r_{1})$. Now, fix an element
  $v_{0}\in \partial B(\varphi,r_{1})$ and let $x : [0,1]\to X$ be the
  \emph{straight line} from $\varphi$ to $v_{0}$ given by
  \begin{displaymath}
    x(r)=\varphi+r(v_{0}-\varphi)\qquad\text{for every $r\in
      [0,1]$}
  \end{displaymath}
  Then, by using again~\eqref{eq:64}, there is an
  $0<r_{2}<\min\{1,r_{1}\}$ such that
  \begin{equation}
    \label{eq:62}
    \abs{\E(x(r)\vert\varphi)}^{1/2}\ge C_{4}\norm{x(r)-\varphi}_{X}
  \end{equation}
  for every $r\in [0,r_{2}]$, where
  $C_{4}=\tfrac{1}{2}\tfrac{\abs{\langle
      \E''(\varphi)(v_{0}-\varphi),v_{0}-
      \varphi\rangle_{X',X}}^{1/2}}{\norm{v_{0}-\varphi}_{X}^{1/2}}>0$. In
  fact, inserting $v=x(r)$ into \eqref{eq:64} gives
  \begin{displaymath}
           \E(x(r)\vert \varphi)=\tfrac{r^{2}}{2}\langle
           \E''(\varphi)(v_{0}-\varphi),v_{0}-\varphi\rangle_{X',X}
           +\mathit{o}(r^{2})\quad\text{as $r\to 0+$.}
  \end{displaymath}
  Since $\frac{\mathit{o}(r^{2})}{r^{2}}\to 0$ as $r\to 0+$ and $\abs{\langle
    \E''(\varphi)(v_{0}-\varphi),v_{0}-\varphi\rangle_{X',X}}>0$,
  there is $0<\delta\le r_{2}<1$ such that
  \begin{displaymath}
    \frac{\mathit{o}(r^{2})}{r^{2}}<\frac{1}{4} \abs{\langle
    \E''(\varphi)(v_{0}-\varphi),v_{0}-\varphi\rangle_{X',X}}
  \end{displaymath}
  for every $0<r\le \delta$. From this and by using the triangle
  inequality, we see that
  \begin{align*}
    \abs{\E(x(r)\vert \varphi)}&=\abs{\tfrac{r^{2}}{2}\langle
    \E''(\varphi)(v_{0}-\varphi),v_{0}-\varphi\rangle_{X',X}
    +\mathit{o}(r^{2})}\\
    &\ge \abs{\tfrac{r^{2}}{4}\langle
    \E''(\varphi)(v_{0}-\varphi),v_{0}-\varphi\rangle_{X',X}}\\
    &= C_{4}^{2}\,\norm{x(r)-\varphi}_{X}^{2}
  \end{align*}
  for every $r\in [0,\delta]$. Now, set
  $\mathcal{D}=\overline{B}(\varphi,\delta)$ and let
  $\hat{v}\in \overline{B}(\varphi,\delta)\cap [\E=\E(x(r))]$ for
  $r\in (0,\delta]$. Note, $\delta>0$ can always be chosen smaller
  such that $\mathcal{D}\subseteq B(\varphi,R_{L})$ for the $R_{L}>0$
  given by~\cite{MR2019030}. Then, by \eqref{eq:24}, \eqref{eq:62},
  and by~\eqref{eq:63} combined with~\eqref{eq:58},
    \begin{align}\notag
      \norm{\E'(x(r))}_{X'}&\le C_{1}\,\norm{x(r)-\varphi}_{X}\\ \notag
      & \le \frac{C_{1}}{C_{4}} \abs{\E(x(r)\vert\varphi)}^{1/2}\\
      &= \frac{C_{1}}{C_{4}} \abs{\E(\hat{v}\vert\varphi)}^{1/2}
      \le \frac{C_{1}}{C_{4}}\frac{C_{3}}{2 C_{2}} \norm{\E'(\hat{v})}_{X'}. \label{ineq:lojhalf}
    \end{align}
    Taking the infimum over all
    $\hat{v}\in \mathcal{D}\cap
   [\E=\E(x(r)\vert\varphi)+\E(\varphi)]$, gives
    \begin{displaymath}
      \norm{\E'(x(r))}_{X'}\le \frac{C_{1}}{C_{4}}\frac{C_{3}}{2 C_{2}}\,s_{B(\varphi,R)}(\E(x(r)\vert\varphi)).
    \end{displaymath}
    Since $r\in (0,\delta]$ were arbitrary, we have thereby shown that the line curve $x$
    is a \emph{talweg curve} through the $C$-valley
    $\mathcal{V}_{C,\mathcal{D}}(\varphi)$ of $\E$ for
    $C=\frac{C_{1}}{C_{4}}\frac{C_{3}}{2 C_{2}}>1$. Moreover,
    $x'(r)=(v_{0}-\varphi)\neq 0$ for every $r\in [0,\delta]$. Hence, $x\in
    AC^{\infty}(0,\delta;X)$. Further, the function
    \begin{displaymath}
      h(r)=\E(x(r)\vert\varphi)\qquad\text{
    for every $r\in [0,\delta]$,}
  \end{displaymath}
  satisfies $h(0)=0$ and since
    $\E\in C^{2}$ and $\E''(\varphi)\neq 0$,
    \begin{align*}
     & h'(r)=\langle \E'(x(r)),v_{0}-\varphi\rangle_{X',X}\neq
       0\qquad\text{for all $t\in (0,\delta)$, and}\\
     & h''(r)=\langle \E''(x(r))(v_{0}-\varphi),v_{0}-\varphi\rangle_{X',X}\neq
       0\qquad\text{for all $r\in (0,\delta)$.}
    \end{align*}
    Thus, $h(r)$ is a homeomorphism from $[0,\delta]\to [0,R]$ for
    $R:=\E(x(\delta)\vert \varphi)$ and a
    diffeomorphism from $(0,\delta)\to (0,R)$. This shows that $\E$
    satisfies the hypotheses of Theorem~\ref{thm:3talweg-kl} and so
    $\E$ satisfies Kurdyka-\L{}ojasiewicz inequality~\eqref{eq:16} near
    $\varphi$. Moreover, by the
    inequalities~\eqref{eq:63} and~\eqref{eq:62},
    \begin{displaymath}
      \frac{1}{C_{3}\,\norm{v_{0}-\varphi}_{X}}\abs{h(r)}^{1/2}\le r \le
      \frac{1}{C_{4}\,\norm{v_{0}-\varphi}_{X}} \abs{h(r)}^{1/2}
    \end{displaymath}
    for every $r\in [0,\delta]$ and so, the function $\theta=h^{-1}$ satisfies
    \begin{displaymath}
      \frac{1}{C_{3}\,\norm{v_{0}-\varphi}_{X}}\abs{s}^{1/2} \le
      \theta(s)\le \frac{1}{C_{4}\,\norm{v_{0}-\varphi}_{X}} \abs{s}^{1/2}
    \end{displaymath}
    for every $s\in [0,R]$. In particular, by inequality~\eqref{ineq:lojhalf},
  \begin{equation}
    \label{eq:59}
    1 \le C\,\norm{\E'(v)}_{X'} \abs{\E(v\vert\varphi)}^{-1/2}
  \end{equation}
  for all $v\in B(\varphi,\delta)$, where $C=\frac{C_{3}}{2C_{2}}>0$. This
  shows that $\E$ satisfies Kurdyka-\L{}ojasiewicz inequality~\eqref{eq:16} near
  $\varphi$ with $\theta(s)=2 C\abs{s}^{\frac{1}{2}-1}s$, or, \L{}ojasiewicz-Simon
  inequality~\eqref{eq:38} near $\varphi$ with exponent
  $\alpha=\tfrac{1}{2}$, showing the optimality of $\theta$ for this
  example of an energy $\E$.
\end{examples}

Next, we consider the non-smooth case of energy functionals
$\E : X \rightarrow (-\infty, +\infty]$. Here, one assumes that $\E$
can be decomposed as $\E=\E_{1}+\E_{2}$, where $\E_{1}$ is a
proper, lower semicontinuous, convex functional and $\E_{2}$ a
continuously differentiable functional on $X$. Then the
\emph{subdifferential} $\partial \E$ of $\E$ is given by
\begin{displaymath}
\partial\E=\Big\{(v,x')\in X\times X'\,\Big\vert\; \liminf_{t\downarrow
0}\tfrac{\E(v+th)-\E(v)}{t}\ge \langle
x',h\rangle_{X',X}\;\text{ for all }\,h\in X\Big\}.
\end{displaymath}
By~\cite[Corollary 1.4.5]{AGS-ZH}, for $v\in D(\abs{D^{-}\E})$, the
\emph{descending slope} 
\begin{equation}
  \label{eq:98}
\abs{D^{-}\E} (v) =\min \left\{ \norm{x'}_{X'}\,\Big\vert\;x' \in \partial\E(v) \right\}
\end{equation}
and $\abs{D^{-}\E}$ is a strong upper gradient of $\E$. Since
$\partial\E^{\circ}(v)$ denotes the set of all \emph{subgradients}
$x'\in \partial\E(v)$ with minimal (dual) norm, the
relation~\eqref{eq:98} can be rewritten as
$\abs{D^{-}\E}(v) = \norm{\partial\E^{\circ}(v)}_{X'}$.

\medskip

In the case $X=(H, (.,.)_{H})$ is a real Hilbert space and
$\E : H \rightarrow (-\infty, +\infty]$ a proper, lower semicontinuous
and semi-convex functional. Then, the following well-known
\emph{generation theorem} \cite[Proposition~3.12]{Brezis} (see also~\cite{MR3465809}) holds for
solutions to gradient system~\eqref{ACP1}.
\begin{center}
  \begin{minipage}[c]{0.9\linewidth}
    \emph{For every $v_0 \in \overline{D(\E)}$, there is a unique
      solution $v$ of} 
      \begin{equation}\label{ACP1bis}\tag{\ref{ACP1}}
        \begin{cases}
          v'(t) + \partial \E (v(t)) \ni 0, & t\in (0, \infty),\\
          \mbox{}\hspace{1,95cm} v(0) = v_0. &\\
        \end{cases}
      \end{equation}
    \mbox{}
    \end{minipage}
 \end{center}

Since semi-convexity coincides with the notion of
 $(-\lambda)$-geodesically convexity for some $\lambda\in \R$ (see
 Remark~\ref{rem:geodesicconvexity-in-H}) and by~\eqref{eq:98}, solution $v$
  of~\eqref{ACP1} are the ($2$-)\emph{gradient flows} of $\E$ in the metric space
  $\mathfrak{M}:=\overline {D(\E)}$ equipped with the induced metric
  of $H$. In particular, $\E$ generates an \emph{evolution variational
    inequality} (see Remark~\ref{rem:evi-conditions} and also~\cite{DaneriSavare}
  and~\cite{AGS-ZH}).\medskip

  As a consequence of Theorem~\ref{thm:convergence}, we obtain the
  following stability result for the Hilbert space framework which
  generalizes~\cite[Theorem~2.6]{MR2289546} (also, compare with
  \cite[Theorem~12.2]{chill-fasangova:isem}, and \cite[Theorem~18,
  $(i)\Rightarrow (ii)$]{Bolte2010}).

\begin{corollary}
  Let $\E : H \to (-\infty,+\infty]$ be a proper, lower semicontinuous
  and semi-convex functional on a Hilbert space $H$. Suppose $v :
  [0,+\infty)\to H$ is a gradient flow of $\E$ and there are $\varepsilon>0$ and an
  equilibrium point $\varphi\in \omega(v)$ such that the set
  $B(\varphi,\varepsilon)$ satisfies hypothesis~\eqref{eq:15}.  If
  there is a strictly increasing function
  $\theta\in W^{1,1}_{loc}(\R)$ satisfying $\theta(0)=0$ and
  $\abs{[\theta>0,\theta'=0]}=0$, for which $\E$ satisfies a
  Kurdyka-\L{}ojasiewicz inequality on the set $\U_{\varepsilon}$
  given by~\eqref{eq:96}, then the following statements hold.
  \begin{enumerate}[topsep=3pt,itemsep=1ex,partopsep=1ex,parsep=1ex]
    \setlength{\itemsep}{7pt}
   \item[(1)] The gradient flow $v$ trends to the equilibrium point $\varphi$ of $\E$ in
     the metric sense of $H$.

   \item[(2)] If there is a Banach space $V$ which is continuously
     embedded into $H$ such that the gradient flow $v\in C((0,\infty);V)$
     and has relatively compact image in $V$, then $v$ trends to the
     equilibrium point $\varphi$ of $\E$ in the metric sense of $V$.
  \end{enumerate}
\end{corollary}

Further, from Theorem~\ref{thm:decayrates}, we can conclude the following
result of decay estimates in the Hilbert space framework.

\begin{corollary}\label{decahilbert}
  Let $\E : H \to (-\infty,+\infty]$ be a proper, lower semicontinuous
  and semi-convex functional on a Hilbert space $H$. Suppose
    $v : [0,+\infty)\to H$ is a gradient flow of $\E$ and there are $c$, $\varepsilon>0$ and
    an equilibrium point $\varphi\in \omega(v)$ such that $\E$
    satisfies a \L{}ojasiewicz-Simon inequality~\eqref{eq:38} with
    exponent $\alpha\in (0,1]$ on $B(\varphi,\varepsilon)$. Then,
    \allowdisplaybreaks
  \begin{align*}
    \norm{v(t)-\varphi}_{H} &\le \tfrac{c}{\alpha}\left(\E(v(t)\vert\varphi)\right)^{\alpha} =
     \; \mathcal{O}\left(t^{^{-\frac{\alpha}{1-2\alpha}}}\right)\hspace{2.25cm}\text{if
       $0<\alpha<\tfrac{1}{2}$}\\
     \norm{v(t)-\varphi}_{H} &\le c\,2\left(\E(v(t)\vert\varphi)\right)^{\frac{1}{2}}
       \le  c\,2\,\left(\E(v(t_{0})\vert\varphi)\right)^{\frac{1}{2}}\; e^{-\tfrac{t}{2c^{2}}}\hspace{0.4cm}\text{if
      $\alpha=\tfrac{1}{2}$}\\
     \norm{v(t)-\varphi}_{H} &\le
    \begin{cases}
   \tilde{c}\,(\hat{t}-t)^{\frac{\alpha}{2\alpha-1}}
      & \quad
    \text{if\quad $t_{0}\le t\le \hat{t}$,}\\
    0 & \quad
    \text{if $t>\hat{t}$,}\\
    \end{cases}
    \mbox{}\hspace{1.8cm}\text{if $\tfrac{1}{2}<\alpha\le 1$,}
  \end{align*}
  where,
  \begin{displaymath}
     \tilde{c}:=\left[\left[\tfrac{1}{\alpha^{\alpha -1}c}\right]^{\frac{1}{\alpha}}\,
      \tfrac{2\alpha-1}{\alpha}\right]^{\frac{\alpha}{2\alpha-1}},\quad
   \hat{t}:=t_{0}+ \, \alpha^{\frac{\alpha -1}{\alpha}}\,
   c^{\frac{1}{\alpha}}\,\tfrac{\alpha}{2\alpha-1}\,
  (\E(v(t_{0})\vert \varphi))^{\frac{2\alpha-1}{\alpha}},
\end{displaymath}
and $t_{0}\ge 0$ can be chosen to be the ``first entry time'', that
is, $t_{0}\ge 0$ is the smallest time $\hat{t}_{0}\in [0,+\infty)$ such that
$v([\hat{t}_{0},+\infty))\subseteq B(\varphi,\varepsilon)$.
\end{corollary}

\subsubsection{{\bfseries Strategy to derive global ET- and \L{}S-inequalities}}
\label{subsubsec:Derivation-LS-inequality}

In the next two examples, we want to illustrate how an abstract
  \emph{Poincar\'e-Sobolev inequality}
  (see~\eqref{eq:PoincareSobolevE} below) leads to an \emph{entropy-transportation
    inequality}~\eqref{eq:ET-global} and from this via the
  equivalence relation to \L{}ojasiewicz-Simon inequality~\eqref{eq:38}
  (Corollary~\ref{equivalent1}) to \emph{decay estimates} of the trend
  to equilibrium and \emph{finite time of extinction} (via
  Corollary~\ref{decahilbert}). In particular, we provide a
  \emph{simple} method to establish upper bounds on the
  \emph{arrival time to equilibrium} (also called \emph{extinction time})
\begin{displaymath}
  T^{\ast}(v_{0}):=\inf\Big\{t>0\,\Big\vert\;
   v(s)= \varphi \quad\text{for all $s\ge t$}\Big\}
\end{displaymath}
of solutions $v$ of parabolic boundary-value problems on a domain
$\Omega\subseteq \R^{N}$, ($N\ge 1$) with initial value $v(0)=v_{0}$
and given equilibrium point $\varphi$. To do this, we revisit two
classical examples on the total variational flow
from~\cite{ACMDirichlet} and~\cite{ACMBook} (see also~\cite{GigaKohn}
and \cite{MR2881044}). But we stress that this method can easily be
applied to any other nonlinear parabolic boundary-value problem that
can be realized as an abstract gradient system~\eqref{ACP1} for a
functional $\E$ defined on a Hilbert space $H$ and for which a \emph{polynomial}
entropy-transportation inequality
\begin{displaymath}
  \norm{v-\varphi}_{H}\le C\,\left(\E(v\vert \varphi)\right)^{\beta}\qquad \qquad\text{($v\in D(\E)$),}
\end{displaymath}
holds for a global minimizer $\varphi$ of $\E$ and given $C$,
$\beta>0$ (see also Remark~\ref{rem:gigakohn} below). Taking $\beta$th
root on both sides of the last inequality, then we call
\begin{equation}
  \label{eq:PoincareSobolevE}
  \norm{v-\varphi}_{H}^{1/\beta}\le \tilde{C}\,\E(v\vert \varphi)\qquad \qquad\text{($v\in D(\E)$),}
\end{equation}
an (abstract) \emph{Poincar\'e-Sobolev inequality}
since~\eqref{eq:PoincareSobolevE} is either obtained by only a known Sobolev inequality (see,
for instance,~\cite[Theorem 3.47]{AFPBook} or~\cite[Chapter 11, 13.5,
14.6]{leoni}) or by a Sobolev inequality combined with a Poincar\'e
inequality (see, for instance, \cite[Theorem 3.44]{AFPBook}).

We note that the idea to employ \L{}ojasiewicz-Simon
inequality~\eqref{eq:38} for deriving \emph{decay estimates} and
\emph{finite time of extinction} of solution of semi-linear and
quasilinear parabolic problems is well-known (see, for
instance,~\cite{MR2019030,MR2289546}). But, our approach to exploit
the equivalence relation between entropy-transportation
inequality~\eqref{eq:ET-global} and \L{}ojasiewicz-Simon
inequality~\eqref{eq:38} (Corollary~\ref{equivalent1}) seems to be
new. On the other hand, we also need to mention that
Corollary~\ref{decahilbert} does not lead to optimal bounds on the
extinction time $T^{\ast}(v_{0})$.

%
%

\subsubsection{{\bfseries Finite Extinction time of the Dirichlet-Total Variational Flow}}
\label{subsubsec:Extinction-Dirichlet}

In \cite{ACMDirichlet} (see also \cite{ACMBook}), the following parabolic initial
  boundary-value problem
  \begin{equation}
   \label{dirichlet1}
   \begin{cases}
     \mbox{}\hspace{11pt}  v_t = {\rm div} \left(\frac{Dv}{\abs{Dv}} \right) & \text{in
       $\Omega\times (0,+\infty)$,}\\
     \mbox{}\hspace{14pt}v=0 & \text{on
       $\partial\Omega\times (0,+\infty)$,}\\
     v(0)=v_{0} & \text{on $\Omega$,}
   \end{cases}
 \end{equation}
 related to the \emph{total variational flow with homogeneous
   Dirichlet boundary conditions} was studied on a bounded connected
 extension domain $\Omega \subseteq \R^N$. In particular, it
 was shown that problem~\eqref{dirichlet1} can be rewritten as an
 abstract initial value problem~\eqref{ACP1bis} in the Hilbert space
 $H=L^{2}(\Omega)$ for the energy functional
 $\E : L^2(\Omega) \rightarrow (-\infty, + \infty]$ given by
  \begin{equation}
    \label{eq:97}
    \E(v):=
    \begin{cases}
      \displaystyle\int_\Omega \abs{D v} + \int_{\partial \Omega}
      \abs{v}   & \text{if $v \in BV(\Omega) \cap
        L^2(\Omega)$,}\\
      +\infty & \text{if otherwise.}
    \end{cases}
  \end{equation}
  In fact (cf~\cite[Theorem 3]{ACMDirichlet}, see also
  \cite[Theorem~5.14]{ACMBook}), $\E$ is a proper, convex and lower
  semicontinuous functional on $L^{2}(\Omega)$ with dense domain. Thus,
  by \cite[Proposition~3.12]{BrezisBook}, for every $v_{0}\in
  L^{2}(\Omega)$, there is a unique solution of problem~\eqref{dirichlet1}.

Now, due to Corollary~\ref{decahilbert}, we can show that for
$N\le 2$, every solution $v$ of problem~\eqref{dirichlet1} with
initial value $v_{0}\in L^{2}(\Omega)$ has finite \emph{extinction time}
\begin{displaymath}
 T^{\ast}(v_{0}):=\inf\Big\{t>0\,\Big\vert\;
 v(s)= 0 \quad\text{for all $s\ge t$}\Big\}.
\end{displaymath}

\begin{theorem}\label{extintiondirichlet}
 Suppose $N\le 2$ and for $v_0 \in L^2(\Omega)$, let
  $v : [0,+\infty)\to L^2(\Omega)$ be the unique strong solution of
  problem~\eqref{dirichlet1}. Then,
    \begin{displaymath}
      T^{\ast}(v_0)\le
      \begin{cases}
       \displaystyle\inf_{s>0}\Big(s+ S_{1}\,\abs{\Omega}^{1/2}\,\E(v(s))\Big) & \text{if $N=1$,}\\[3pt]
        \displaystyle\inf_{s>0}\Big(s+ S_{2}\,\E(v(s)) \Big) & \text{if $N=2$,}
      \end{cases}
    \end{displaymath}
  where $S_N$ is the best constant in
  Sobolev inequality~\eqref{1poincaree}, and
  \begin{equation}\label{cadirichlet1}
   \norm{v(t)}_{L^{2}(\Omega)} \le
  \begin{cases}
    \tilde{c}\, (T^{\ast}(v_0) - t) & \text{if $0\le t \le T^{\ast}(v_0)$,}\\
    0 & \text{if $ t > T^{\ast}(v_0)$,}
  \end{cases}
\end{equation}
for some $\tilde{c}>0$.
\end{theorem}

\begin{proof}
  By the Sobolev inequality for $BV$-functions (see~\cite[Theorem
  3.47]{AFPBook}),
  \begin{equation}\label{1poincaree}
    \norm{u}_{L^{1^{\ast}}(\R^N)} \le S_{N}\,
    \displaystyle\int_{\R^N} \abs{D u}
    \qquad \text{for every $u \in BV(\R^N)$,}
  \end{equation}
  where $1^{\ast} = \infty$ if $n=1$, $1^{\ast} = 2$ if
  $n=2$. Since $\Omega$ is an extension domain, for
  every $v \in BV(\Omega)$, the extension $\hat{v}$ of $v$ given by
  \begin{displaymath}
  \hat{v}(x):=
  \begin{cases}
    v(x) & \text{if $x \in \Omega$}\\
    0  &\text{if $x \in \R^N \setminus \overline{\Omega}$,}
  \end{cases}
 \end{displaymath}
 belongs to $BV(\R^{N})$ and satisfies
 \begin{displaymath}
    \displaystyle\int_{\R^N} \abs{D\hat{v}} =
    \displaystyle\int_{\Omega} \abs{D v} + \int_{\partial \Omega} \abs{v}.
 \end{displaymath}
  (cf~\cite[Theorem 3.89]{AFPBook}) and so by
  inequality~\eqref{1poincaree}, we find that
 \begin{equation}
   \label{12poincaree}
    \norm{v}_{L^{1^{\ast}}(\Omega)} \le S_{N}\,
    \left(\displaystyle\int_{\Omega}\abs{D v}
      + \int_{\partial \Omega} \abs{v}\right)
  \end{equation}
  for all $v \in BV(\Omega)$. Note that for
  $v\in BV(\Omega)\cap L^{2}(\Omega)$, the right-hand side
  in~\eqref{12poincaree} is $\E(v)$. Moreover, $\varphi\equiv 0$ is
  the (unique) global minimizer of $\E$. Thus and since for $N\le 2$,
  the Sobolev exponent $1^{\ast}\ge 2$, we can apply H\"older's
  inequality to conclude from~\eqref{12poincaree} that
  \begin{equation}
         \label{poinc12bis}
         \norm{v}_{L^{2}(\Omega)}
        \le C\, \E(v\vert \varphi)
           \qquad \text{for all $v\in D(\E)$,}
 \end{equation}
 where the constant $C=S_{1}\,\abs{\Omega}^{1/2}$ if $N=1$ and
 $C=S_{2}$ if $N=1$. Due to the (abstract)
 Sobolev-inequality~\eqref{poinc12bis}, $\E$
 satisfies, in fact, an \emph{entropy-transportation
   inequality}~\eqref{eq:13-transport0} for $\Psi(s)= C\,s$,
 ($s\in \R$), which by Corollary~\ref{equivalent1}, is equivalent to
 the \L{}ojasiewicz-Simon inequality
 \begin{displaymath}
    1 \le C\,\norm{\partial^{\circ}\E(v)}_{H}=C\,\abs{D^- \E}(v),
    \qquad (v\in [\E>0]).
 \end{displaymath}
 Thereby, we have shown that the energy functional $\E$ given
   by~\eqref{eq:97} satisfies a global \L{}ojasiewicz-Simon
 inequality~\eqref{eq:38} with exponent $\alpha=1$ at the
   equilibrium point $\varphi\equiv 0$. Moreover, since the energy $\E$ is convex, this equilibrium point
 is a global minimizer of $\E$. Thus, if $v(t)\in D(\E)$ for every
 $t>0$, then $\varepsilon>0$ in the neighborhood
 $B(\varphi,\varepsilon)$ in Corollary~\ref{decahilbert} can be chosen
 arbitrarily large. But by to the theory of gradient flows in Hilbert spaces
 (cf~\cite{BrezisBook}), for every initial value $v_{0}\in
 L^{2}(\Omega)$, the gradient flows $v$ of $\E$ with initial datum
 $v(0)=v_{0}$ satisfies $v(t)\in D(\E)$ for all $t>0$. Therefore, we
 can conclude this proof by applying Corollary~\ref{decahilbert}.
\end{proof}

\begin{remark}
  \label{rem:talenti}
  It was shown in~\cite{Talenti} that in dimension $N=2$,
  Sobolev inequality~\eqref{1poincaree} has the sharp constant
  $S_{2}= \frac{1}{\sqrt{2 \pi}}$ with Sobolev exponent $1^{\ast}=2$.
  Thus, Corollary~\ref{decahilbert} applied to
  $v_{0}\in D(\E)=BV(\Omega)$ with $t_0 = 0$
  and $\alpha = 1$ yields that in dimension $N=2$, the
  \emph{extinction time} $T^{\ast}(v_0)$ fulfills
  \begin{displaymath}
      T^{\ast}(v_0) \le \frac{1}{\sqrt{2 \pi}} \int_\Omega  \abs{D v_0}.
  \end{displaymath}

  In particular, if $E \subsetneq\Omega$ is a set of finite
  perimeter and $v_0 = a\, \mathds{1}_E$, then
  \begin{equation}
      \label{eq:99}
      T^{\ast}(v_0) \le   \frac{a}{\sqrt{2 \pi}}  {\rm Per}(E).
  \end{equation}

  Let us point out that the upper bound of $T^{\ast}(v_0)$ given by~\eqref{eq:99} is not
  optimal. In fact, it was shown in~\cite{ACDM} that if
  $0\in \Omega$, $a>0$ and for sufficiently small $R>0$, then
    \begin{displaymath}
      v(t,x) := \frac{2}{R} \left[\frac{a\, R}{2} - t \right]^+
      \mathds{1}_{B(0,R)}(x)\qquad\text{for every $x\in \Omega$, $t>0$,}
    \end{displaymath}
    is the unique solution of~\eqref{dirichlet1} with initial
    datum $v_0 = a\, \mathds{1}_{B(0,R)}$. But, $v$ has the
    extinction time $T^{\ast}(v_{0})=\frac{a\, R}{2}$, which
    is smaller than
    the upper bound
    \begin{displaymath}
      \frac{a}{\sqrt{2 \pi}} {\rm Per}(B(0,R)) = a\, R \sqrt{2
        \pi}\qquad\text{given by~\eqref{eq:99}.}
   \end{displaymath}
\end{remark}

%
%

\subsubsection{{\bfseries Finite extinction time to equilibrium of the Neumann-Total
  Variational Flow}}
\label{subsubsec:Extinction-Neumann}

In \cite{ACMBook}, the following \emph{Neumann problem}
\begin{equation}
  \label{Neumann1}
  \begin{cases}
    \mbox{}\hspace{10pt}v_t = {\rm div} \left(\frac{Dv}{\vert Dv \vert} \right) & \text{in
      $\Omega\times (0,+\infty)$,}\\
    \mbox{}\hspace{7pt}\tfrac{\partial v}{\partial \nu}=0 & \text{on
      $\partial\Omega\times (0,+\infty)$,}\\
    v(0)=v_{0} & \text{on $\Omega$,}
  \end{cases}
\end{equation}
related to the \emph{total variational flow} was studied. Here,
$\Omega \subseteq \R^N$ is a bounded connected extension domain and if
$\nu$ is the outward pointing unit normal vector at $\partial\Omega$
then $\tfrac{\partial v}{\partial \nu}$ denote the \emph{co-normal
  derivative} $\frac{Dv}{\vert Dv \vert} \cdot \nu$ associated with
$\mathcal{A}:=-{\rm div} \left(Dv/\vert Dv \vert \right)$. It was
shown in~\cite[Theorem~2.3]{ACMBook} that for the energy functional
$\E : L^2(\Omega) \rightarrow (-\infty, + \infty]$ given by
\begin{equation}
  \label{eq:100}
  \E(v):=
  \begin{cases}
    \displaystyle\int_\Omega \abs{D v} & \text{if
      $v \in BV(\Omega) \cap
      L^2(\Omega)$,}\\
    +\infty & \text{if otherwise,}
  \end{cases}
\end{equation}
problem~\eqref{Neumann1} can be rewritten as an abstract initial value
problem~\eqref{ACP1bis} in the Hilbert space
$H=L^{2}(\Omega)$. Moreover, $\E$ is a proper, convex and lower
semicontinuous functional on $L^{2}(\Omega)$ with dense domain. Thus,
by \cite[Proposition~3.12]{BrezisBook}, for every
$v_{0}\in L^{2}(\Omega)$, there is a unique strong solution
of~\eqref{Neumann1}. Due to~\cite[Theorem~2.20]{ACMBook}, for every
gradient flow $v : [0,+\infty)\to L^{2}(\Omega)$ of $\E$, the
$\omega$-limit set is given by
\begin{displaymath}
  \omega(v)=\Big\{\overline{v_{0}}=\tfrac{1}{\abs{\Omega}}\int_{\Omega} v_{0}\,\dx\Big\}.
\end{displaymath}

Now, as an application of Corollary~\ref{decahilbert} and due to the
Poincar\'e-Sobolev inequality (\cite[Remark 3.50]{AFPBook})
\begin{equation}
    \label{Poincaree}
      \norm{v-\overline{v}}_{1^{\ast}}\le
      C_{PS_{N}}\,\int_{\Omega}\abs{D v}=C_{PS_{N}}\,\E(v\vert \overline{v}),\qquad(v\in
      BV(\Omega)),
 \end{equation}
 every solution $v$ of problem~\eqref{Neumann1} arrives \emph{in finite
   time} to the mean value $\overline{v_{0}}$ of its initial condition
 $v(0)=v_{0}\in L^{2}(\Omega)$. We omit the details of the proof since it
 follows the same idea as the previous one.

\begin{theorem}\label{extintionNeumann}
  Suppose $N\le 2$ and for $v_0 \in L^2(\Omega)$, let
  $v : [0,+\infty)\to L^2(\Omega)$ be the unique strong solution of
  problem~\eqref{Neumann1}. Then,
    \begin{displaymath}
      T^{\ast}(v_0)\le
      \begin{cases}
        \displaystyle\inf_{s>0}\Big(s+ C_{PS_{1}}\,\abs{\Omega}^{1/2}\,\E(v(s))\Big) & \text{if $N=1$,}\\[3pt]
        \displaystyle\inf_{s>0}\Big(s+ C_{PS_{2}}\,\E(v(s)) \Big) & \text{if $N=2$,}
      \end{cases}
    \end{displaymath}
    where $C_{PS_{N}}$ is the best constant in
  Poincar\'e-Sobolev inequality~\eqref{Poincaree}, and
\begin{equation}\label{CANeumann1}
  \norm{ v(t) - \overline{v_0}}_{2} \le
  \begin{cases}
    \tilde{c}\,(T^{\ast}(v_0) - t) & \text{if $0\le t \le T^{\ast}(v_0)$,}\\
    0 & \text{if $ t > T^{\ast}(v_0)$,}
  \end{cases}
\end{equation}
for some $\tilde{c}>0$.
\end{theorem}

\begin{remark}
  Similarly to Remark~\ref{rem:talenti}, 
  Corollary~\ref{decahilbert} yields that for every $v_{0}\in D(\E)$,
  the \emph{time $T^{\ast}(v_0)$ to the equilibrium} fulfills
  \begin{displaymath}
    T^{\ast}(v_0) \le C_{PS_{N}}\, \int_\Omega  \abs{D v_0}.
  \end{displaymath}
  We emphasize that this estimate of $T^{\ast}(v_0)$ is not
  contained in \cite{ACMBook}. In particular, if $E \subsetneq\Omega$
  is a set of finite perimeter, $a\in \R$ and $v_0 = a\,
  \mathds{1}_E$, then for the gradient
  flow $v : [0,+\infty)\to L^{2}(\Omega)$ of the energy $\E$ given
  by~\eqref{eq:100} and with initial value
  $v(0)=v_{0}$, the extinction time $T^{\ast}(v_0)$ fulfills
  \begin{displaymath}
    T^{\ast}(v_0) \le C_{PS_{N}}\, \abs{a}\, {\rm Per}(E).
  \end{displaymath}
\end{remark}

\begin{remark}[\emph{Extinction time in fourth order total variational
    flow problems}]
  \label{rem:gigakohn}
   It is important to mention the rigorous
    study~\cite{GigaKohn} on the extinction time $T^{\ast}$ of solutions of the total variational
    flow equipped with Dirichlet (including $\Omega=\R^{N}$), Neumann,
    and periodic boundary conditions. The results in~\cite[Theorem~2.4
    and Theorem~2.5]{GigaKohn} only depend on the $L^{N}$-norm of the
    initial data. Furthermore, in~\cite{GigaKohn} upper estimates are
    established on the extinction time $T^{\ast}$ of solutions of the
    \emph{fourth-order total variation flow equation}
    \begin{equation}
      \label{eq:101}
      v_t = - \Delta \left[ {\rm div} \left(\frac{Dv}{\vert Dv \vert} \right)\right].
    \end{equation}
    After adding the right choice of boundary conditions,
    equation~\eqref{eq:101} can be realized as a gradient
    system~\eqref{ACP1} in $H=H^{-1}(\Omega)$ for the energy
    \begin{displaymath}
      \E(v):=
      \begin{cases}
      \displaystyle  \int_{\Omega}\abs{D v} & \text{if $v\in BV(\Omega)\cap
          H^{-1}(\Omega)$,}\\
        +\infty & \text{if otherwise,}
      \end{cases}
    \end{displaymath}
    for every $v\in H^{-1}(\Omega)$. Due to \emph{entropy-transportation inequality}
    \begin{displaymath}
      \norm{v}_{H^{-1}(\Omega)} \le C\, \int_{\Omega} \abs{D v}=C\,\E(v),
    \end{displaymath}
    (for instance, cf~\cite[(29)]{GigaKohn} in dimension $N=4$),
    Corollary~\ref{equivalent1} and Corollary~\ref{decahilbert}  yield
    the existence of upper estimates of the extinction time $T^{\ast}(v_{0})$
    for solutions of~\eqref{eq:101}.
\end{remark}

%
%
\subsection{Gradient flows in spaces of probability measures.}
\label{GradProbmeasures}
In this second part of Section~\ref{sec:application}, we turn to the
stability analysis of $p$-gradient flows in spaces of probability measures.

We assume that the reader is familiar with
the basics of optimal transport and refer for further reading to the
standard literature (see, for instance,~\cite{MR1964483},
\cite{Villani2}, \cite{AGS-invent}, or \cite{MR3409718}). Here, we
only recall the notations and results which are important to us for
this article.\medskip

Throughout this part, let $1<p<\infty$ with H\"older conjugate 
$p^{\mbox{}_{\prime}}=\frac{p}{p-1}$, and $\mathcal{B}$ be the standard Borel
  $\sigma$-algebra generated by the induced Euclidean norm-topology of
  the space $\R^{N}$, $(N\ge 1)$. Then, we denote by $\mathcal{P}(\R^{N})$
  the space of prob\-ability measures on $(\R^{N},\mathcal{B})$, and
  $\mathcal{P}_{p}(\R^{N})$ for the linear subspace of probability measures
  $\mu\in \mathcal{P}(\R^{N})$ of \emph{finite $p$-moment} $\int_{\R^{N}} \vert x \vert^p \td\mu(x)$.
Now, for every two probability measures $\mu_{1}$,
$\mu_{2}\in \mathcal{P}_{p}(\R^{N})$, the \emph{$p$-Wasserstein
  distance} is defined by
\begin{displaymath}
W_{p}(\mu_1,\mu_2)= \left( \inf_{\pi \in \Pi(\mu_1, \mu_2)}
\int_{\R^{N}\times \R^{N}}
\abs{x- y}^p\; \td\pi(x,y)\right)^{\frac{1}{p}}.
\end{displaymath}
Here, for given $\mu_1$, $\mu_2 \in \mathcal{P}(\R^{N})$,
$\Pi(\mu_1, \mu_2)$ denotes the \emph{transport plans $\pi \in \mathcal{P}(\R^{N} \times \R^{N})$ with marginals
  $\mu_1$ and $\mu_2$} defined by
\begin{displaymath}
\Pi(\mu_1, \mu_2):= \Big\{ \pi \in \mathcal{P}(\R^{N} \times \R^{N})\; \Big\vert \;
\bm{p}_{1 \#} \pi = \mu_{1},\; \bm{p}_{2 \#}\pi = \mu_2 \Big\},
\end{displaymath}
where for $i\in \{1,2\}$, $\bm{p}_i$, denotes the
\emph{projection} from $\R^{N} \times \R^{N}$ onto the
$i$th-component $\R^{N}$. Furthermore, for a mapping $T : \R^{N}\to \R^{N}$ (respectively, $T :
\R^{N} \times \R^{N}\to \R^{N}$) be a measurable map, then the
\emph{push-forward of $\mu\in \mathcal{P}(\R^{N})$ through $T$}
(respectively, of $\mu\in \mathcal{P}(\R^{N}\times \R^{N})$ through $T$) is defined by
\begin{displaymath}
    T_{\#}\mu(B):=\mu(T^{-1}(B))\qquad\text{for every $B\in \mathcal{B}$.}
  \end{displaymath}
Then, for given $\mu\in \mathcal{P}(\R^{N}\times \R^{N})$, $\bm{p}_{1\#}\mu\in
\mathcal{P}(\R^{N})$ and $\bm{p}_{2\#}\mu\in \mathcal{P}(\R^{N})$ are the \emph{marginals}
 of $\mu$.

Since the cost function $c_{p}(x,y):=
\abs{x-y}^{p}$, ($x$,
$y\in \R^{N}$), is continuous on
$\R^{N}\times \R^{N}$ and bounded from below, for
every pair $\mu_{1}$ and
$\mu_{2}\in \mathcal{P}_p(\R^{N})$, there is probability measure
$\pi^{\ast}\in \Pi(\mu_{1},\mu_{2})$ such that
\begin{displaymath}
  W_{p}^{p}(\mu_1,\mu_2)=\int_{\R^{N}\times \R^{N}}
\abs{x- y}^p\; \td\pi^{\ast}(x,y).
\end{displaymath}
Such a measure $\pi^{\ast}$ is called an \emph{optimal transport plan}
of $\mu_{1}$ and $\mu_{2}$.

Next, let $\mathcal{P}_p^{ac}(\R^{N})$ be the subclass of measures
$\mu\in \mathcal{P}_p(\R^{N})$, which are \emph{absolutely continuous} with
respect to the Lebesgue measure $\mathcal{L}^{N}$ on $\R^{N}$, that is,
\begin{displaymath}
 \mu_{1}(B) = 0 \qquad\text{for all $B\in \mathcal{B}$ with $\mathcal{L}^{N}(B)=0$.}
\end{displaymath} 
Then, for every $\mu_{1}\in \mathcal{P}_p^{ac}(\R^{N})$ and
$\mu_{2}\in \mathcal{P}_p(\R^{N})$, there exists a map $T^{\ast} :
\R^{N}\to \R^{N}$ such that satisfying $T^{\ast}_{\#}\mu_2=\mu_{1}$
and the $p$-Wasserstein
distance can be rewritten as
\begin{equation}
  \label{eq:109}
  W_{p}^{p}(\mu_1,\mu_2)= \int_{\R^{N}}
\abs{x- T^{\ast}(x)}^p\; \td\mu_{1}(x).
\end{equation}
The map $T^{\ast}_{\#}$ is unique and called the \emph{optimal
  transport map} of $\mu_{1}$ and $\mu_{2}$ (cf~\cite{Brenier1} for the case $p=2$
and \cite{G-Mc} for $1<p<\infty$).\medskip

Next, consider the \emph{free energy}
$\E : \mathcal{P}_p(\R^N)\to (-\infty,+\infty]$ composed by
\begin{equation}
  \label{eq:10}
  \E=\HE_{F}+\HE_{V}+\HE_W,
\end{equation}
with $\HE_{F}$ the \emph{internal energy} given by
  \begin{displaymath}
    \HE_{F}(\mu)=
    \begin{cases}
      \displaystyle\int_{\R^{N}}F(\rho)\,\dx & \text{if $\mu=\rho\, \mathcal{L}^{N}$,}\\
      +\infty & \text{if $\mu\in \mathcal{P}_p(\R^N)\setminus \mathcal{P}_p^{ac}(\R^N)$,}
    \end{cases}
 \end{displaymath}
$\HE_{V}$ the \emph{potential energy} defined by
  \begin{displaymath}
    \HE_{V}(\mu)=
    \begin{cases}
     \displaystyle \int_{\R^{N}}V\,\td\mu & \text{if $\mu=\rho\, \mathcal{L}^{N}$,}\\
      +\infty & \text{if $\mu\in \mathcal{P}_p(\R^N)\setminus \mathcal{P}_p^{ac}(\R^N)$,}
    \end{cases}
 \end{displaymath}
 and $\HE_W$ the \emph{interaction energy} defined by
 \begin{displaymath}
    \HE_{W}(\mu)=
    \begin{cases}
      \displaystyle\tfrac{1}{2}\int_{\R^{N}\times\R^{N}}W(x-y) \,\td
      (\mu\otimes\mu)(x,y)  & \text{if $\mu=\rho\, \mathcal{L}^{N}$,}\\
      +\infty & \text{if
        $\mu\in \mathcal{P}_p(\R^N)\setminus
        \mathcal{P}_p^{ac}(\R^N)$.}
    \end{cases}
 \end{displaymath}
  Here, we assume that the function
  \begin{enumerate}
  \item[{\bf (F)} :] $F : [0,+\infty)\to \R$ is a convex differential function
    satisfying
    \begin{align}
      \label{eq:102}
     &  F(0)=0,\quad\liminf_{s\downarrow
        0}\frac{F(s)}{s^{\alpha}}>-\infty\quad\text{for some
        $\alpha>N/(N+p)$,}\\
      \label{eq:103}
     & \text{the map $s\mapsto s^{N}F(s^{-N})$ is convex and non
        increasing in $(0,+\infty)$,}
    \end{align}
    there is a $C_F>0$ such that
    \begin{align}
      \label{eq:105}
    &\qquad F(s+\hat{s})\le
      C_{F}\,(1+F(s)+F(\hat{s}))
      \quad\text{for all $s$, $\hat{s}\ge 0$, and}\\
     \label{eq:104}
    &\qquad\lim_{s\to
      +\infty}\frac{F(s)}{s}=+\infty\qquad\qquad\text{(super-linear
      growth at infinity);}
    \end{align}
   \item[{\bf (V)} :]  $V : \R^{N}\to (-\infty,+\infty]$ is a proper, lower
     semicontinuous, 
     $\lambda$-convex for some $\lambda\in \R$, and the effective
     domain $D(V)$ has a convex, nonempty interior
     $\Omega:=\inter D(V)\subseteq \R^{N}$.
    \item[{\bf (W)} :]  $W : \R^{N}\to [0,+\infty)$ is a convex,
      differentiable, and even function and there is a $C_{W}>0$ such that
      \begin{equation}
        \label{eq:106}
        W(x+\hat{x})\le C_{W}\,(1+W(x)+W(\hat{x})) \quad
        \text{for all $x$, $\hat{x}\in \R^{N}$.}
      \end{equation}
  \end{enumerate}

\begin{remark}
  The study of gradient flows generated by the free energy functional
  $\E$ given by~\eqref{eq:10} on
  the Wasserstein space $(\mathcal{P}_2(\R^N), W_2)$ was done
  independently in \cite{CMVArchive} and \cite{AGS-ZH}.
\end{remark}

We note that the conditions~\eqref{eq:102} on $F$ ensure that for every
  $\rho\in L^{1}(\R^{N})$, the \emph{negative
  part} $F^{-}(s)=\max\{-F(s),0\}$ satisfies $F^{-}(\rho)\in
L^{1}(\R^{N})$, that is,~\eqref{eq:102} provides growth bounds on
$F^{-}$ (cf~\cite[Remark~9.3.7]{AGS-ZH}), due to
condition~\eqref{eq:103}, the functional $\HE_{F}$ is geodesically
convex on $\mathcal{P}_p(\R^N)$, and thanks to
condition~\eqref{eq:104}, $\HE_{F}$ is lower semicontinuous on $\mathcal{P}_p(\R^N)$
(cf~\cite[Remark~9.3.8]{AGS-ZH}). The so-called \emph{doubling}
condition~\eqref{eq:105} is needed to calculate the directional
derivative of $\HE_{F}$ (cf~\cite[Lemma~10.4.4]{AGS-ZH}).

Concerning the functional $\HE_{V}$, the hypotheses that $V$ is
convex, proper and lower semicontinuous imply that $\HE_{V}$ is bounded
from below by an affine support function and hence, by Cauchy-Schwarz' and
Young's inequality, for every $1<p<+\infty$, there are $A$, $B>0$ such
that $V(x)\ge -A-B\abs{x}^{p}$. From this, we can deduce that
$\HE_{V}$ is lower semicontinuous on $\mathcal{P}_p(\R^N)$
(cf~\cite[Lemma~5.1.7]{AGS-ZH}). By \cite[Proposition~9.3.2]{AGS-ZH},
$\HE_{V}$ is $\lambda$-geodesically convex in $\mathcal{P}_p(\R^N)$ if
$p\le 2$ and $\lambda\ge 0$, or if $p\ge 2$ and $\lambda\le 0$.

Similarly to $\HE_{V}$, the convexity of $W$, the fact that $W$ is
proper, and the lower semicontinuity of $W$, imply that $\HE_{W}$ is
lower semicontinuous in $\mathcal{P}_p(\R^N)$ and
by~\cite[Proposition~9.3.5]{AGS-ZH}, $\HE_{W}$ is geodesically convex
on $\mathcal{P}_p(\R^N)$. The doubling condition~\eqref{eq:106} is sufficient
for characterizing the decreasing slope $\abs{D^{-}\E_{W}}$ of
$\HE_{W}$ (cf~\cite[Theorem~10.4.11]{AGS-ZH}).

For an open set $\Omega\subseteq \R^{N}$, we denote by
  $\mathcal{P}_p(\Omega)$ is the closed subspace of probability
  measures $\mu\in \mathcal{P}_p(\R^N)$ with support
  $\textrm{supp}(\mu)\subseteq \overline{\Omega}$. Moreover,
  $\mathcal{P}_p^{ac}(\Omega):=\mathcal{P}_p(\Omega)\cap
  \mathcal{P}_p^{ac}(\R^N)$.

Under the hypotheses~{\bf (F)}, {\bf (V)} and~{\bf (W)},
Proposition~\ref{propo:lconvex-slope} implies that the descending
slope $\abs{D^{-}\E}$ of $\E$ is lower semicontinuous. Moreover,
by~\cite[Theorem~10.4.13]{AGS-ZH}, $\abs{D^{-}\E}$ can be
characterized as follows.

\begin{proposition}
  \label{propo:Fokkerdescending-slope}
  Suppose, the functions $F$, $V$ and $W$ satisfy the
  hypotheses~{\bf (F)}, {\bf (V)} and~{\bf (W)}, and
  $\E : \mathcal{P}_p(\R^N)\to (-\infty,+\infty]$ is the functional
  given by~\eqref{eq:10}. Then, for
  $\mu=\rho \mathcal{L}^{N}\in D(\E)$, one has
  $\mu\in D(\abs{D^{-}\E})$ if and only if
  \begin{equation}
    \label{eq:108}
    P_{F}(\rho)\in W^{1,1}_{loc}(\Omega),\quad
    \rho\,\xi_{\rho} = \nabla P_{F}(\rho)+\rho \nabla V+\rho
    (\nabla W)\ast \rho
  \end{equation}
  for some $\xi_{\rho}\in
  L^{p^{\mbox{}_{\prime}}}(\R^{N},\R^{N};\td\mu)$, where
  $P_{F}(x):=xF'(x)-F(x)$ is the associated ``pressure function'' of $F$. Moreover,
  the vector field $\xi_{\rho}$ satisfies
  \begin{equation}
    \label{eq:115}
    \abs{D^{-}\E}(\mu)= \left(\int_{\R^{N}}\abs{\xi_{\rho}(x)}^{p^{\mbox{}_{\prime}}}\,\td\mu\right)^{\frac{1}{p'}}.
  \end{equation}
\end{proposition}

By following the idea of the proof of~\cite[Proposition~4.1]{JKO} and
replacing $p=2$ by general $1<p<+\infty$, one sees that under the
hypotheses~{\bf (F)}, {\bf (V)} and~{\bf (W)}, the free energy
$\E : \mathcal{P}_p(\R^N)\to (-\infty,+\infty]$ given by~\eqref{eq:10}
satisfies the hypotheses of Theorem~\ref{exist1}. Therefore, by the
regularity result in~\cite[Theorem~11.3.4]{AGS-ZH}, for every
$\mu_{0}\in D(\E)$, there is a $p$-gradient flow
$\mu : [0,+\infty)\to \mathcal{P}_p(\R^N)$ of $\E$ with initial value
$\lim_{t\downarrow 0}\mu(t)=\mu_{0}$. Moreover, for every $t>0$,
$\mu(t)=\rho(t)\,\mathcal{L}^{N}$ with
$\textrm{supp}(\rho(t))\subseteq \overline{\Omega}$, and $\rho$ is a
distributional solution of the following quasilinear
parabolic-elliptic boundary-value problem
\begin{equation}
\label{eq:107}
\begin{cases}
  \rho_{t}+ \divergence (\rho\,
  \bm{U}_{\rho})=0 &\qquad \text{in
                                                $(0,+\infty)\times
                                                \Omega$,}\\
  \mbox{}\hspace{1,89cm}\bm{U}_{\rho}=-\abs{\xi_{\rho}}^{p^{\mbox{}_{\prime}}-2}\xi_{\rho}
  &\qquad \text{in
                                                $(0,+\infty)\times
                                                \Omega$,}\\
  \mbox{}\hspace{1,34cm}\bm{U}_{\rho}\cdot\bm{n}=0  &\qquad \text{in
                                               $(0,+\infty)\times
                                                \partial\Omega$,}
\end{cases}
\end{equation}
with $P_{F}(\rho)\in L^{1}_{loc}((0,+\infty);W^{1,1}_{loc}(\Omega))$ and
\begin{displaymath}
  \xi_{\rho}=\frac{\nabla P_{F}(\rho)}{\rho}+\nabla V+ (\nabla W)\ast
  \rho\in L^{\infty}_{loc}((0,+\infty);L^{p^{\mbox{}_{\prime}}}(\Omega,\R^{N};\td\mu(\cdot))),
\end{displaymath}
where, $\bm{n}$ in~\eqref{eq:107} denotes the outward unit normal to the boundary
$\partial\Omega$ which in the case $\Omega=\R^{N}$ needs to be
neglected.

If the function $F\in C^{2}(0,+\infty)$, then one has that
\begin{displaymath}
  -\bm{U}_{\rho}=
  \abs{F''(\rho)\nabla \rho+\nabla V+(\nabla W)\ast \rho}^{p^{\mbox{}_{\prime}}-2}
  \left(F''(\rho)\nabla \rho+\nabla V+(\nabla W)\ast \rho \right).
\end{displaymath}
Thus (cf~\cite{Agueh1}, \cite{CMVArchive}, \cite{CMVIbero}), problem~\eqref{eq:107} includes the
\begin{itemize}
\item \emph{doubly nonlinear diffusion equation}
  \begin{displaymath}
    \rho_{t}- \divergence (\abs{\nabla \rho^{m}}^{p^{\mbox{}_{\prime}}-2}\,
    \nabla \rho^{m})=0
  \end{displaymath}
  ($V=W=0$, $F(s)=\tfrac{m\,s^{q}}{q(q-1)}$ for
  $q=m+1-\frac{1}{p^{\mbox{}_{\prime}}-1}$,
  $\tfrac{1}{p^{\mbox{}_{\prime}}-1}
  \neq m\ge \frac{N-(p^{\mbox{}_{\prime}}-1)}{N(p^{\mbox{}_{\prime}}-1)}$)\\[7pt]
\item \emph{Fokker-Planck equation with interaction term through porous medium}
  \begin{displaymath}
    \rho_{t}= \Delta \rho^{m}+\divergence\left(\rho (\nabla V + (\nabla W)\ast \rho)\right)
  \end{displaymath}
  ($p=2$, $F(s)=\tfrac{s^{m}}{(m-1)}$ for
  $1\neq m\ge 1-\frac{1}{N}$).
\end{itemize}

Due to Proposition~\ref{propo:Fokkerdescending-slope}, every
equilibrium point $\nu=\rho_{\infty}\mathcal{L}^{N}\in \mathbb{E}_{\abs{D^{-}\E}}$ of $\E$
can be characterized by
\begin{equation}
\label{eq:112}
\begin{cases}
  & P_{F}(\rho_{\infty})\in W^{1,1}_{loc}(\Omega)\qquad\text{with}\\[7pt]
  &\displaystyle \xi_{\rho_{\infty}}=\frac{\nabla
    P_{F}(\rho_{\infty})}{\rho_{\infty}}+\nabla V+ (\nabla W)\ast
  \rho_{\infty}=0\quad\text{ a.e. on $\Omega$.}
\end{cases}
\end{equation}

Further, for every $p$-gradient flow $\mu$ of $\E$ and equilibrium
point $\nu\in \mathbb{E}_{\abs{D^{-}\E}}$, equation~\eqref{eq:11} in
Proposition~\ref{propo:chara-p-curves} reads as follows
\begin{equation}
\label{eq:117}
  \tfrac{\td}{\dt}\E(\mu(t))=-\abs{D^{-}\E}^{p^{\mbox{}_{\prime}}}(\mu(t))
  =-\mathcal{I}_{p^{\mbox{}_{\prime}}}(\mu(t)\vert\nu),
\end{equation}
where due to Proposition~\ref{propo:Fokkerdescending-slope}, the
\emph{generalized relative Fischer information of
  $\mu$ with respect to $\nu$} is given by
\begin{displaymath}
  \mathcal{I}_{p^{\mbox{}_{\prime}}}(\mu\vert\nu)=
  \int_{\Omega}   -\bm{U}_{\rho}\cdot\xi_{\rho}\, \td\mu.
\end{displaymath}

For our next lemma, we introduce the notion of \emph{uniformly
  $\lambda$-$p$ convex} functions (cf~\cite{MR2079071,MR2016985}),
which for differentiable $f$ generalizes the notion of
$\lambda$-convexity on the Euclidean space $\mathfrak{M}=\R^{N}$
(cf~Definition~\ref{def:lambdaconvex}).

\begin{definition}
 \label{def:lambda-p-convex}
 We call a functional $f : \R^{N}\to (-\infty,+\infty]$
  \emph{uniformly $\lambda$-$p$-convex} for some $\lambda\in \R$ if
  the interior $\Omega=\textrm{int}(D(f))$ of $f$ is nonempty, $f$ is
  differentiable on $\Omega$ and for every $x\in
  \Omega$,
  \begin{displaymath}
    f(x)-f(x)\ge \nabla f(x)\cdot
    (y-x)+\lambda\,\abs{y-x}^{p}\qquad\text{for all $y\in \R^{N}$.}
  \end{displaymath}
\end{definition}

Further, we need the following definition from~\cite{AGS-ZH}.

\begin{definition}
  Let $c : \R^{N}\times\R^{N}\to [0,+\infty]$ be a proper and
  lower semicontinuous function. Then, for $u : \R^{N}\to
  [+\infty,+\infty]$, the \emph{$c$-transform} $u^{c} : \R^{N}\to
  [+\infty,+\infty]$ is defined by
  \begin{displaymath}
    u^{c}(y)=\inf_{x\in \R^{N}}\left(c(x,y)-u(x)\right)\qquad\text{for
    every $y\in \R^{N}$.}
  \end{displaymath}
  A function $u : \R^{N}\to  [+\infty,+\infty]$ is called
  \emph{$c$-concave} if there is a function $v : \R^{N}\to
  [+\infty,+\infty]$ such that $u=v^{c}$.
\end{definition}

With these preliminaries, we can state the following result
  which generalizes~\cite[Theorem~2.4, Theorem~4.1]{MR2079071} and
~\cite[inequality~(5)]{MR2016985} (see also~\cite[Theorem~2.1]{MR2053603}).

\begin{lemma}
 Suppose, the functions $F$, $V$ and $W$
  satisfy the hypotheses~{\bf (F)},
\begin{enumerate}
  \item[{\bf (V$\mbox{}^{\ast}$)}] $V : \R^{N}\to (-\infty,+\infty]$ is
    proper, lower semicontinuous function, the effective domain $D(V)$
    of $V$ has nonempty interior $\Omega:=\inter D(V)\subseteq \R^{N}$,
    and $V$ is uniformly $\lambda_{V}$-$p$-convex
    for some $\lambda_{V}\in \R$;
  \item[{\bf (W$\mbox{}^{\ast}$)}] $W : \R^{N}\to [0,+\infty)$ is a
    differentiable, even function satisfying~\eqref{eq:106}, and for
    some $\lambda_{W}\ge0$, $W$ is uniformly $\lambda_{W}$-$p$-convex.
  \end{enumerate}
  Further, let $F\in C^{2}(0,\infty)\cap C[0,+\infty)$ and
  $\E : \mathcal{P}_p(\R^N)\to (-\infty,+\infty]$ be the functional
  given by~\eqref{eq:10}. Then, for every probability
  measures $\mu_{1}=\rho_{1}\mathcal{L}^{N}$,
  $\mu_{2}=\rho_{2}\mathcal{L}^{N}\in \mathcal{P}_p^{ac}(\Omega)$ with
  $\rho_{2}\in W^{1,\infty}(\Omega)$ and $\inf\rho_{2}>0$, one has
  \begin{equation}
    \label{eq:111bis}
    \begin{split}
      \E(\mu_{1}\vert\mu_{2})&\ge
      \int_{\Omega}(T^{\ast}(x)-x)\cdot\xi_{\rho_{2}} \td\mu_{2}+
      (\lambda_{V}+\tfrac{\lambda_{W}}{2})\,W_{p}^{p}(\mu_{1},\mu_{2})\\
      &\qquad
      +\tfrac{p\,\lambda_{W}}{2}
          \int_{\Omega\times\Omega}\nabla\theta(x)\cdot(T^{\ast}(y)-y)\,
          \rho_{2}(x)\,\rho_{2}(y)\dx\dy,
    \end{split}
  \end{equation}
  where $T^{\ast}$ is the optimal transport map
  satisfying~\eqref{eq:109} with $T^{\ast}_{\#}\mu_2=\mu_{1}$ and
  \begin{displaymath}
  x-T^{\ast}(x)=\abs{\nabla \theta(x)}^{p^{\mbox{}_{\prime}}-2}\nabla\theta(x),
\end{displaymath}
for a $c_{p}$-concave function $\theta$ with $c_{p}(x,y):=\tfrac{1}{p}\abs{x-y}^{p}$.
\end{lemma}

\begin{proof}
 Under the hypotheses of this lemma, \cite[Theorem~2.4]{MR2079071}
  yields
  \begin{equation}
\label{eq:110bis}
    \HE_{F+V}(\mu_{1}\vert\mu_{2})\ge
    \int_{\Omega}(T^{\ast}(x)-x)\cdot (F''(\rho)\nabla \rho+\nabla V)\td\mu_{2}+
    \lambda_{V}\,W_{p}^{p}(\mu_{1},\mu_{2})
  \end{equation}
  where we set $\HE_{F+V}=\HE_{F}+\HE_{V}$. Next, we deal with the
  interaction energy $\HE_{W}$. For this, we follow an idea given
  in~\cite{MR2079071}. Since $T^{\ast}_{\#}\mu_2=\mu_{1}$, we
  can rewrite
  \begin{displaymath}
    \HE_{W}(\mu_{1})=\tfrac{1}{2}\int_{\Omega\times\Omega}
    W(T^{\ast}(x)-T^{\ast}(y)) \,\rho_{2}(x)\,\rho_{2}(y)\dx\,\dy
  \end{displaymath}
  Thus, and since $W$ is uniformly $\lambda_{W}$-$c_{p}$-convex for
      some $\lambda_{W}\ge 0$, we have
      \begin{align*}
        \HE_{W}(\mu_{1})
        &\ge \HE_{W}(\mu_{2}) + \tfrac{1}{2}\int_{\Omega\times\Omega}
          \Big[\nabla W(x-y)\cdot \big((T^{\ast}(x)-x)\big.\Big.\\
        &\mbox{}\hspace{5cm}\Big.\big.-(T^{\ast}(y)-y)\big)\Big]\,\rho_{2}(x)\,\rho_{2}(y)\dx\dy\\
        &\mbox{}\hspace{1,2cm}+\tfrac{\lambda_{W}}{2}
          \int_{\Omega\times\Omega}\abs{(T^{\ast}(x)-x)-(T^{\ast}(y)-y)}^{p}\,\rho_{2}(x)\,\rho_{2}(y)\,\dx\dy
      \end{align*}
    and since by hypothesis, $\nabla W$ is odd,
    \begin{align*}
       \HE_{W}(\mu_{1})
        &\ge \HE_{W}(\mu_{2}) + \int_{\Omega}\nabla (W\ast \rho_{2})\cdot (T^{\ast}(x)-x)\;\td\mu_{2}(x)\\
        &\mbox{}\hspace{1,5cm}+\tfrac{\lambda_{W}}{2}
          \int_{\Omega\times\Omega}\abs{(T^{\ast}(x)-x)-(T^{\ast}(y)-y)}^{p}\,\rho_{2}(x)\,\rho_{2}(y)\,\dx\dy.
      \end{align*}
   Due to the elementary inequality $\abs{a-b}^{p}\ge
   \abs{a}^{p}-p\abs{a}^{p-2}a\cdot b$, ($a$, $b\in \R^{N}$),
   \begin{align*}
        \HE_{W}(\mu_{1})
        &\ge \HE_{W}(\mu_{2}) + \int_{\R^{N}}\nabla (W\ast
          \rho_{2})\cdot (T^{\ast}(x)-x)\;\td\mu_{2}(x)
          +\tfrac{\lambda_{W}}{2} W_{p}^{p}(\mu_{1},\mu_{2})\\
        & - \tfrac{p\,\lambda_{W}}{2}
          \int_{\Omega\times\Omega}\abs{T^{\ast}(x)-x}^{p-2}(T^{\ast}(x)-x)(T^{\ast}(y)-y)\,\rho_{2}(x)\,\rho_{2}(y)\dx\dy.
      \end{align*}
    Combining this inequality with~\eqref{eq:110bis} yields the desired inequality~\eqref{eq:111bis}.
\end{proof}

We note that if the interaction term $W$ satisfies hypothesis~{\bf
  (W)} then, in particular, {\bf (W$\mbox{}^{\ast}$)} holds with
$\lambda_{W}=0$. Thus, we can announce for the free energy
$\E=\HE_{F}+\HE_{V}+\HE_W$ the following \emph{ET-inequality},
\emph{generalized Log-Sobolev inequality} (cf~\cite[Proposition~1.1]{MR2016985}), and
\emph{generalized $p$-HWI inequality} (cf~\cite[Theorem~1.2]{MR2016985}
, \cite[Theorem~3]{OttoVillani} and \cite[Theorem 2.1]{CMVIbero}).

\begin{theorem}
  \label{thm:inequalitiesWp}
  Suppose that the functions $F$, $V$ and $W$
  satisfy the hypotheses~{\bf (F)}, {\bf (V$\mbox{}^{\ast}$)} with
  $\lambda_{V}\in \R$ and {\bf
  (W)}. Further, suppose $F\in C^{2}(0,\infty)\cap C[0,+\infty)$ and let
  $\E : \mathcal{P}_p(\R^N)\to (-\infty,+\infty]$ be the functional
  given by~\eqref{eq:10}. Then, the following statements hold.
  \begin{enumerate}
  \item {\bf (ET-inequality)} For an equilibrium point
    $\nu=\rho_{\infty}\mathcal{L}^{N}\in
    \mathbb{E}_{\abs{D^{-}\E}}$ of $\E$ with
    $\rho_{\infty}\in W^{1,\infty}(\Omega)$, $\inf_{\Omega}\rho_{\infty}>0$,
    one has that
    \begin{equation}
      \label{eq:113}
      \lambda_{V}\,W_{p}^{p}(\mu,\nu)  \le \E(\mu\vert \nu)\qquad
      \text{for every $\mu=\rho\mathcal{L}^{N}\in D(\E)$.}
    \end{equation}
    \item {\bf ($p$-Talagrand transportation inequality)} If $\lambda_{V}>0$,
      then ET-inequality~\eqref{eq:113} is
      equivalent to the $p$-Talagrand inequality
     \begin{equation}
     \label{eq:118}
     W_{p}(\mu,\nu) \le
     \tfrac{1}{\lambda_{V}^{1/p}}\,\sqrt[p]{\E(\mu\vert \nu)}, \qquad
     (\mu=\rho\mathcal{L}^{N}\in D(\E)),
   \end{equation}
    where $\nu=\rho_{\infty}\mathcal{L}^{N}\in
    \mathbb{E}_{\abs{D^{-}\E}}$, $(\rho_{\infty}\in W^{1,\infty}(\Omega))$. 

   \item {\bf (generalized \L{}S-inequality)} For every
     $\hat{\lambda}>0$, and $\mu_{1}=\rho_{1}\mathcal{L}^{N}$, $\mu_{2}=\rho_{2}\mathcal{L}^{N}\in
  \mathcal{P}_p^{ac}(\Omega)$ with $\rho_{2}\in W^{1,\infty}(\Omega)$, $\inf_{\Omega}\rho_{2}>0$, one has that
  \begin{equation}
    \label{eq:114}
    \E(\mu_{2}\vert \mu_{1})+(\lambda_{V}-\hat{\lambda})\,
    W_{p}^{p}(\mu_{1},\mu_{2})
    \le  \frac{p-1}{p^{p^{\mbox{}_{\prime}}}}
    \frac{1}{\hat{\lambda}^{1/(p-1)}}\, \abs{D^{-}\E}^{p^{\mbox{}_{\prime}}}(\mu_{2}).
  \end{equation}
  \item {\bf (generalized Log-Sobolev inequality)} If $\lambda_{V}>0$,
    then for every $\mu_{1}=\rho_{1}\mathcal{L}^{N}$, $\mu_{2}=\rho_{2}\mathcal{L}^{N}\in
  \mathcal{P}_p^{ac}(\Omega)$ with $\rho_{2}\in W^{1,\infty}(\Omega)$,
  $\inf_{\Omega}\rho_{2}>0$, and $\nu\in \mathbb{E}_{\abs{D^{-}\E}}$, one has that
  \begin{equation}
    \label{eq:116}
    \E(\mu_{2}\vert \mu_{1})
    \le  \frac{p-1}{p^{p^{\mbox{}_{\prime}}}}
    \frac{1}{\lambda_{V}^{1/(p-1)}}
    \mathcal{I}_{p^{\mbox{}_{\prime}}}(\mu_{2}\vert\nu).
  \end{equation}
\item {\bf ($p$-HWI inequality)} For every
  $\mu_{1}=\rho_{1}\mathcal{L}^{N}$,
  $\mu_{2}=\rho_{2}\mathcal{L}^{N}\in \mathcal{P}_p^{ac}(\Omega)$ with
  $\rho_{2}\in W^{1,\infty}(\Omega)$ and $\inf_{\Omega}\rho_{2}>0$, one has
  that
    \begin{equation}
      \label{eq:120}
    \E(\mu_{2}\vert
    \mu_{1}) + \lambda_{V}\,W_{p}^{p}(\mu_{1},\mu_{2}) \le
    \mathcal{I}^{1/p^{\mbox{}_{\prime}}}_{p^{\mbox{}_{\prime}}}(\mu_2\vert\nu)
    \,W_{p}(\mu_{1},\mu_{2}).
  \end{equation}
  \end{enumerate}
\end{theorem}

\begin{remark}
  If for $\lambda_{V}>0$, $\E$ satisfies an
    entropy-transportation inequality~\eqref{eq:113} at an equilibrium
    point
    $\nu=\rho_{\infty}\mathcal{L}^{N}\in
    \mathbb{E}_{\abs{D^{-}\E}}$ of $\E$ with
    $\rho_{\infty}\in W^{1,\infty}(\Omega)$, then $\nu$ is
    the unique minimizer of $\E$.
\end{remark}


\begin{remark}[{\bfseries The case $V=W=0$.}]
  It is well known that for $\E$ given by~\eqref{eq:10} with $V=W=0$ a Sobolev inequality
  holds (which again implies a Log-Sobolev inequality of the
  form~\eqref{eq:116}). For further details,
  we refer the interested reader to~\cite{MR2016985} and~\cite{Agueh1}.
\end{remark}

\begin{proof}[Proof of Theorem~\ref{thm:inequalitiesWp}]
  Here, we follow an idea in~\cite{MR2016985}. Thus, we
  only provide a sketch of the proof. Inequality~\eqref{eq:113} follows
  directly from~\eqref{eq:111bis} by taking $\mu_{2}=\nu$ and
  applying the characterization~\eqref{eq:112} for the equilibrium
  point $\nu\in \mathbb{E}_{\abs{D^{-}\E}}$. Talagrand
  inequality~\eqref{eq:118} is equivalent to
  ET-inequality~\eqref{eq:113} by simply taking
  $p$th root or \emph{vice versa} $p$th power. Next, for every $\hat{\lambda}>0$, Young's
  inequality yields that
  \begin{displaymath}
    (x-T^{\ast}(x))\cdot\xi_{\rho_{2}}\le
    \hat{\lambda}\,\abs{x-T^{\ast}(x)}^{p}
    +\frac{p-1}{p^{p^{\mbox{}_{\prime}}}}
    \frac{1}{\hat{\lambda}^{1/(p-1)}}\,\abs{\xi_{\rho_{2}}}^{p^{\mbox{}_{\prime}}}.
  \end{displaymath}
  Applying this to~\eqref{eq:111bis} and using~\eqref{eq:115}, one sees that
  \begin{equation}
    \label{eq:119}
     \lambda_{V}\,W_{p}^{p}(\mu_{1},\mu_{2})  \le \E(\mu_{1}\vert
     \mu_{2}) + \hat{\lambda} \,W_{p}^{p}(\mu_{1},\mu_{2}) +\tfrac{p-1}{p^{p^{\mbox{}_{\prime}}}}
    \frac{1}{\hat{\lambda}^{1/(p-1)}}\,\abs{D^{-}\E}^{p^{\mbox{}_{\prime}}}(\mu_{2}).
  \end{equation}
  From this follows the generalized
  \L{}S-inequality~\eqref{eq:114}. Now, by choosing $\hat{\lambda}=\lambda_{V}>0$ in~\eqref{eq:114} and
  by~\eqref{eq:117}, one obtains the Log-Sobolev
  inequality~\eqref{eq:116}. Finally, we show $p$-HWI
  inequality~\eqref{eq:120}. For this, one minimizes the function
  \begin{displaymath}
    \Psi(\hat{\lambda})=\hat{\lambda}
    \,W_{p}^{p}(\mu_{1},\mu_{2}) +\frac{p-1}{p^{p^{\mbox{}_{\prime}}}}
    \frac{1}{\hat{\lambda}^{1/(p-1)}}\, \abs{D^{-}\E}^{p^{\mbox{}_{\prime}}}(\mu_{2})
  \end{displaymath}
  over $(0,+\infty)$. $\Psi$ attains its minimum at
  \begin{displaymath}
    \hat{\lambda}_{0}= \frac{1}{p}\,\frac{\abs{D^{-}\E}(\mu_{2})}{W_{p}^{p-1}(\mu_{1},\mu_{2})}.
  \end{displaymath}
  Inserting $\hat{\lambda}_{0}$ into~\eqref{eq:119} and by
  identity~\eqref{eq:117}, one obtains~\eqref{eq:120}.
\end{proof}

Our next corollary shows that even if $V$ fails to be uniformly
$\lambda_{V}$-$p$-convex for some $\lambda_{V}>0$, \emph{equivalence}
between entropy transportation inequality~\eqref{eq:113},
\L{}ojasiewicz-Simon inequality~\eqref{eq:38}, and the logarithmic
Sobolev inequality~\eqref{eq:116} holds for the free energy functional
$\E$ given by~\eqref{eq:10} (cf~\cite[Corollary~3.1]{OttoVillani},
\cite[Proposition~3.6]{MR2079071}). Our next result, is a special case
of Corollary~\ref{equivalent1} adapted to the framework in
$\mathcal{P}_p(\R^N)$.

\begin{corollary}[{\bf Equivalence between ET-, \L{}S-
    and Log-Sobo\-lev}]
  \label{cor:equiv-ET-LS-LogSob-RN}
  Suppose that the functions $F$, $V$ and $W$ satisfy the
  hypotheses~{\bf (F)}, {\bf (V)} and {\bf (W)}. Further, suppose
  $F\in C^{2}(0,\infty)\cap C[0,+\infty)$ and let
  $\E : \mathcal{P}_p(\R^N)\to (-\infty,+\infty]$ be the functional
  given by~\eqref{eq:10}. Then, the following statements
  hold.
\begin{enumerate}
\item If for $\nu\in \mathbb{E}_{\abs{D^{-}\E}}$, there is
  a $\hat{\lambda}>0$ such that $\E$
  satisfies entropy transportation inequality
  \begin{equation}
    \label{eq:121}
    W_{p}(\mu,\nu)\le \hat{\lambda}\,(\E(\mu\vert
    \nu))^{\frac{1}{p}}
    \qquad\text{for all $\mu\in D(\E)$,}
    \end{equation}
    then $\E$ satisfies the \L{}ojasiewicz-Simon inequality
    \begin{equation}
      \label{LS}
      \E(\mu\vert \mu_\infty)^{1 -\frac{1}{p}} \le \hat{\lambda}\,
      \abs{D^{-}\E}(\mu)
      \qquad\text{for all $\mu\in D(\abs{D^{-}\E})$,}
    \end{equation}
    or equivalently, $\E$ satisfies the Log-Sobolev inequality
    \begin{equation}
      \label{LogS}
      \E(\mu\vert \mu_\infty)^{1 -\frac{1}{p}} \le
      \hat{\lambda}^{\frac{1}{1-\frac{1}{p}}}\,
      \mathcal{I}_{p^{\mbox{}_{\prime}}}(\mu\vert\nu)
      \qquad\text{for all $\mu\in D(\abs{D^{-}\E})$,}
    \end{equation}

  \item If for $\nu\in \mathbb{E}_{\abs{D^{-}\E}}$, there is
  a $\hat{\lambda}>0$ such that $\E$ satisfies Log-Sobolev
  inequality~\eqref{LogS}, then $\E$
  satisfies entropy transportation inequality
  \begin{displaymath}
     W_{p}(\mu,\nu)\le \hat{\lambda}p\,(\E(\mu\vert
    \nu))^{\frac{1}{p}}
    \qquad\text{for all $\mu\in D(\E)$,}
  \end{displaymath}
\end{enumerate}
\end{corollary}

From Theorem~\ref{thm:inequalitiesWp} and
Theorem~\ref{thm:decayrates}, we can conclude the following
exponential decay rates (cf~\cite{MR2016985,Agueh1,Otto2} and
\cite[Corollary~5.1]{MR2053603}).

\begin{corollary}[{\bf Trend to equilibrium and exponential decay
    rates}]
  \label{cor:exprate-RN}
  Suppose that the functions $F$, $V$ and $W$
  satisfy the hypotheses~{\bf (F)}, {\bf (V$\mbox{}^{\ast}$)} with
  $\lambda_{V}>0$ and {\bf
  (W)}. Further, suppose $F\in C^{2}(0,\infty)\cap C[0,+\infty)$ and let
  $\E : \mathcal{P}_p(\R^N)\to (-\infty,+\infty]$ be the functional
  given by~\eqref{eq:10}. Then, there is a unique
  minimizer $\nu=\rho_{\infty}\mathcal{L}^{N}\in \mathbb{E}_{\abs{D^{-}\E}}$ of $\E$
  satisfying~\eqref{eq:112} and for every initial value
  $\mu_{0}\in D(\E)$, the $p$-gradient flow $\mu$ of $\E$ trends to
  $\nu$ in $\mathcal{P}_p(\Omega)$ as $t\to+\infty$ and for all $t\ge 0$,
  \begin{equation}
    \label{eq:28}
    W_{p}(\mu(t),\nu)\le \tfrac{(p-1)^{1/p^{\mbox{}_{\prime}}}}{\lambda_{V}^{1/p}}
    \left(\E(\mu(t)\vert\nu)\right)^{\frac{1}{p}}\le
    \tfrac{(p-1)^{1/p^{\mbox{}_{\prime}}}}{\lambda_{V}^{1/p}}
    \left(\E(\mu_{0}\vert\nu)\right)^{\frac{1}{p}}
    e^{-\frac{tp^{\frac{1}{p}}}{p-1}\lambda_{V}^{\frac{1}{p-1}}}.
  \end{equation}
\end{corollary}

\begin{remark}
  From~\eqref{eq:28}, one can deduce strong convergence in $L^1(\R^N)$
  (cf~\cite{CJMTU,Otto2}) or even strong convergence in $BV(\R^{N})$
  (cf~\cite[Remark~22.12]{Villani2}) by using a Csiszar-Kullback(-Pinsker)
  inequality.
\end{remark}

%
%


\providecommand{\bysame}{\leavevmode\hbox to3em{\hrulefill}\thinspace}

\end{document}